\documentclass[11pt]{article}
\usepackage{arxiv}
\usepackage{enumitem}
\usepackage{mathrsfs}
\usepackage{subcaption}
\usepackage{graphicx}
\usepackage{mathabx}
\usepackage{authblk}
\usepackage{comment}

\mathtoolsset{showonlyrefs}

\newcommand{\dist}{\mathrm{dist}}
\newcommand{\spec}{\mathsf{spec}}
\newcommand{\SNR}{\mathsf{SNR}}
\newcommand{\SA}{\mathsf{Sym}}
\newcommand{\Tsf}{\mathsf{Tr}}
\newcommand{\Card}{\mathrm{Card}}
\DeclareMathOperator{\matop}{mat}

\DeclareMathOperator{\Diag}{Diag}

\newcommand{\pto}{\xrightarrow{\mathbb P}}

\title{Sharp Spectral Thresholds for Multi-View Spiked Wigner Models}
\author[1]{Xiaodong Yang \thanks{xyang@g.harvard.edu}}
\author[1]{Subhabrata Sen \thanks{subhabratasen@fas.harvard.edu}}
\author[1,2]{Yue M. Lu \thanks{yuelu@seas.harvard.edu}}

\affil[1]{Department of Statistics, Harvard University}
\affil[2]{Applied Mathematics, Harvard University}

\date{\today}

\begin{document}

\maketitle

\begin{abstract}
Motivated by multimodal estimation, we study a multi-view spiked Wigner model in which several noisy matrix observations contain correlated latent spikes. We derive a spectral estimator for the latent spikes by linearizing approximate message passing (AMP). Our main result is an explicit sharp transition formula for its spectrum: for $L \geq 2$ views, letting $\lambda$ be the $L$-dimensional vector of spike strengths and $B$ the $L\times L$ limiting Gram matrix of the spikes, the critical parameter is
\[
    \SNR\rbr{\lambda,B}
    =
    \lambda_{\max}\sbr{\Diag(\sqrt{\lambda}) (B \odot B) \Diag(\sqrt{\lambda})}.
\]
When $\SNR\rbr{\lambda,B}<1$, the linearized AMP matrix has no outlier beyond the right edge of its bulk spectrum. When $\SNR\rbr{\lambda,B}>1$, an informative outlier is pinned at the distinguished point $1$, and the associated eigenvector has explicit, nontrivial overlaps with the latent signals. Thus $\SNR\rbr{\lambda,B}=1$ gives the exact spectral weak-recovery threshold for the linearized AMP method. To establish our results, we analyze the correlated Gaussian noise matrix through a matrix Dyson equation and combine this deterministic description with finite-rank perturbation arguments adapted to the multi-view spike structure. We also show that, for a broad class of spike priors, the spectral threshold $\SNR(\lambda,B)=1$ coincides with the information-theoretic threshold for weak recovery, ruling out a statistical-computational gap for this class of priors.
\end{abstract}

{\setlength{\parskip}{0pt plus 1pt}
\tableofcontents
}

\section{Introduction}
Multi-view data arise when several related measurements are collected on the
same underlying units. In many applications, each view is individually noisy,
while the signals shared across views are correlated. A natural statistical
question is how to combine the views so that weak information in individual
measurements can be aggregated into a reliable estimator. We study this
question through a stylized multiview spiked matrix model, which is motivated
by multimodal PCA and related multiview inference problems
\cite{flury1984common,nandy2024multimodal,yang2024fundamental}.

For a fixed number of views $L\geq 1$, we observe symmetric matrices
\begin{equation}\label{eq:observation model}
    A^{(l)}
    = \sqrt{\frac{\lambda_l}{n}} X^{(l)} X^{(l)\top}
      + W^{(l)} \in \RR^{n \times n},
    \qquad 1\leq l\leq L.
\end{equation}
Here $X^{(l)}\in\RR^n$ is the latent signal in view $l$, $W^{(l)}$ is an
independent Gaussian Orthogonal Ensemble noise matrix, and $\lambda_l$ is the
signal-to-noise parameter of the corresponding view. The spike vectors are
generally correlated across views; writing
$X=[X^{(1)},\ldots,X^{(L)}]$, their empirical covariance converges to a fixed
matrix $B$. We work in the asymptotic regime where $n\to\infty$ while
$L$, $\lambda=(\lambda_1,\ldots,\lambda_L)$, and $B$ remain fixed. This is the
critical regime in which the signal and noise terms in
\eqref{eq:observation model} have comparable spectral size.

When $L=1$, the model reduces to the classical rank-one spiked Wigner model,
or an idealized setting for principal component analysis
\cite{johnstone2001distribution}. In that case, the top eigenvalue and
eigenvector undergo the Baik--Ben Arous--Peche transition: above a critical
signal-to-noise level, an outlier eigenvalue separates from the bulk and the
associated eigenvector has nontrivial correlation with the signal
\cite{baik2005phase,knowles2013isotropic,benaych2011eigenvalues}. The
corresponding information-theoretic thresholds for weak recovery are also well
understood \cite{deshpande2017asymptotic,lelarge2019fundamental,
perry2018optimality}. Our goal is to identify the analogous sharp spectral
transition in the multiview setting, where the views must be combined according
to both their individual strengths $\lambda_l$ and their correlation structure
$B$.

The model \eqref{eq:observation model} was studied in recent work by the first
two authors, who identified recovery limits and introduced an approximate
message passing (AMP) algorithm for signal recovery
\cite{bolthausen2014iterative,feng2022unifying,yang2024fundamental}. This line
of work was motivated in part by community detection in multilayer network
models \cite{yang2024fundamental}; we return to this connection in
Section~\ref{sec:prior_works}. The AMP algorithm reaches the desired recovery
threshold when suitably initialized, but it requires a warm start. This
motivates the central question of the present paper: can one construct a
principled spectral method which reaches the same threshold and provides such
an initialization?

We answer this question by studying the matrix obtained from
linearizing the multiview AMP iteration. Our main result gives an explicit
phase transition parameter:
\begin{equation}\label{eq:summary SNR def}
    \SNR(\lambda,B)
    = \lambda_{\max}\!\left(
        \Diag(\sqrt{\lambda}) (B\odot B)\Diag(\sqrt{\lambda})
      \right).
\end{equation}
When $\SNR(\lambda,B)<1$, the linearized AMP matrix has no  outliers beyond the right edge of the spectrum. When
$\SNR(\lambda,B)>1$, an outlier eigenvalue appears at $1$ and  the associated
eigenvector has nonzero asymptotic overlap with the latent signals.  We
characterize these overlaps explicitly. For several natural distributions on the spikes, this
spectral threshold coincides with the information-theoretic threshold, so the
model exhibits no statistical-computational gap in these cases.

The proof has two main ingredients. First, we analyze the bulk spectrum of the
linearized AMP matrix through a matrix Dyson equation. The correlated Gaussian
noise has a finite-dimensional Kronecker structure, and the matrix Dyson
equation gives the deterministic description needed for both the bulk law and
the spike-direction resolvent estimates. Second, we identify the special point
$z=1$ in this deterministic equation and connect the resulting condition to the
replica-symmetric free energy of the Gaussian-prior problem. This yields the
explicit threshold $\SNR(\lambda,B)=1$.

The rest of the paper is organized as follows.
Section~\ref{sec:main results} introduces the linearized AMP matrix and states the
main results. Section~\ref{sec:proof_outline} gives a proof roadmap.
Section~\ref{sec:formal present self-cosnistent eq} studies the associated
matrix Dyson equation. Sections~\ref{sec:results on noise mat}
and~\ref{sec:outlier} prove the bulk law and characterize the outlier
eigenvalues and eigenvectors. Sections~\ref{sec:optimizing RS free energy}
and~\ref{sec:emergence of outlier 1} establish the explicit phase transition
formula and prove the main informative-eigenvector result.
Section~\ref{sec:derivation} derives the spectral operator from the linearized
AMP iteration. We close with a short
discussion.

\section{Main results}\label{sec:main results}
We state our main results in this section. We first fix notation and standing assumptions,
then introduce the linearized AMP matrix which is the key object in our work. 


\paragraph{Notation.}
For $N\in\NN$, write $[N]=\cbr{1,\ldots,N}$. We denote by
$\CC^+=\cbr{z\in\CC:\Im z>0}$ the upper half-plane, and write
$\tr(A)$ for the unnormalized trace of a square matrix $A$. Normalized traces
will always be written with their normalizing factor, for example
$n^{-1}\tr(A)$ or $(nL)^{-1}\tr(A)$. For
$\lambda=(\lambda_1,\ldots,\lambda_L)$, set
$\Lambda:=\Diag(\lambda)$.
For an $N\times N$ self-adjoint matrix $A$, we write
$\sigma_1(A)\ge\cdots\ge\sigma_N(A)$ for its eigenvalues, counted with
multiplicity.

The maps $\diag:\RR^{L\times L}\to\RR^L$ and
$\Diag:\RR^L\to\RR^{L\times L}$ extract the diagonal of a matrix and form the
corresponding diagonal matrix. Thus $\Diag(\diag(M))$ is the diagonal matrix
obtained from $M$ by keeping its diagonal entries and setting all off-diagonal
entries to zero. We write $A\odot B$ for the Hadamard product and
$A\otimes B$ for the Kronecker product. We shall use the identity
\begin{equation}
    \diag(A\Diag(c)B^\top)=(A\odot B)c\in\CC^m,
\end{equation}
valid for all $c\in\CC^n$ and $A,B\in\CC^{m\times n}$.

Let $M_L(\RR)$ and $M_L(\CC)$ be the spaces of $L\times L$ real and complex
matrices. For $A\in M_L(\CC)$, set $A^\dagger:=\Bar{A}^\top$. We write
\begin{equation}
    \SA_L(\RR)=\cbr{A\in M_L(\RR):A^\top=A},\qquad
    \SA_L^+(\RR)=\cbr{A\in\SA_L(\RR):A\succeq 0}.
\end{equation}
The notation for symmetric matrices is extended complex-linearly:
\begin{equation}
    \SA_L(\CC)=\cbr{A+iB: A,B\in\SA_L(\RR)}
    =\cbr{A\in M_L(\CC):A^\top=A}.
\end{equation}
This convention is distinct from Hermitian symmetry and is the natural one for
the matrix Dyson equations below. For $A=A_1+iA_2\in\SA_L(\CC)$ with
$A_1,A_2\in\SA_L(\RR)$, we write $\Im A:=A_2$; for a scalar $z\in\CC$,
$\Im z$ has its usual meaning. For self-adjoint matrices $A$ and $B$, we use
the Loewner order: $A\preceq B$ means that $B-A$ is positive semidefinite, and
$A\prec B$ means that $B-A$ is positive definite.
We write $X_n\pto X$ for convergence in probability.

\paragraph{Assumptions.}
We now state the standing assumptions used throughout the paper.
\begin{definition}[Reducible covariance matrices]\label{def:reducible-covariance}
    A covariance matrix $\Sigma\in\RR^{L\times L}$ is reducible if, after a
    simultaneous permutation of the views, it can be written as a block
    diagonal matrix $\Diag(\Sigma_1,\Sigma_2)$ with two nonempty blocks.
    Otherwise it is irreducible.
\end{definition}
\begin{assumption}\label{assump:base assumption}
    The first assumption specifies the noise matrices
    in~\eqref{eq:observation model}, while the remaining assumptions concern
    the spikes and their limiting Gram matrix.
    \begin{enumerate}[label=(A\arabic*),ref=A\arabic*]
        \item\label{assump:Gaussian Wigner} For each $l\in[L],i,j\in[n]$,
        the matrix entries satisfy
        $W_{ij}^{(l)}=W_{ji}^{(l)}\sim\cN(0,1+\mathbf{1}\{i=j\})$. These
        Gaussians are independent.

        \item\label{assump:spike concentration} There exists a fixed matrix
        $B\in\SA_L^+(\RR)$ which describes the limiting Gram matrix of the
        spikes. More precisely, denoting
        $X:=\sbr{X^{(1)},\ldots,X^{(L)}}\in\RR^{n\times L}$, for any
        $\epsilon>0$ and $D>0$, and for all sufficiently large $n$,
        \begin{align}
        \mathbb{P}\Big( \nbr{X^\top X/n-B} \geq \frac{1}{n^{1/2- \epsilon}} \Big) \leq n^{-D}. \nonumber
        \end{align}

        \item\label{assump:positive-lambda} Each view has positive signal
        strength: $\lambda_l>0$ for every $l\in[L]$.

        \item\label{assump:irreducible-normalized-B} We take $B$ to be
        irreducible and positive definite, and use the normalization $B_{ll}=1$ for every
        $l\in[L]$.
    \end{enumerate}
\end{assumption}
\begin{remark}
    Assumption~\ref{assump:spike concentration} holds if the spike entries
    $\{(X_i^{(l)}: 1\leq l \leq L): 1\leq i \leq n\}$ are sampled i.i.d. from
    an underlying prior with the corresponding concentration bounds.
    The weaker empirical assumption stated above is sufficient for
    our spectral results.

    The reductions in Assumption~\ref{assump:positive-lambda} and the
    normalization in Assumption~\ref{assump:irreducible-normalized-B} are made
    without loss of generality. If $\lambda_l=0$ for some $l\in[L]$, the
    corresponding view $A^{(l)}$ in~\eqref{eq:observation model} does not
    contribute to signal recovery; we may therefore discard it and work with the
    remaining views. The condition $B_{ll}\equiv1$ is simply a choice of scale
    for the spikes, corresponding to $\nbr{X^{(l)}}_2\approx\sqrt n$; any
    non-unit diagonal entries of $B$ can be absorbed into $\lambda$ by rescaling
    the spike vectors.

    The positive
    semidefiniteness of $B$ in Assumption~\ref{assump:spike concentration} is
    automatic for a Gram or second-moment matrix, while the positive-definiteness
    requirement in Assumption~\ref{assump:irreducible-normalized-B} is a
    nondegeneracy condition. Equivalently, every nonzero linear combination of
    the view signals has positive asymptotic second moment; in statistical terms,
    each view contains a nonzero innovation beyond what can be explained
    linearly by the other views. This excludes perfectly correlated signals and,
    more generally, cases in which one signal direction is an exact linear
    combination of the others. Such singular cases lie on the boundary of the
    present formulation; one should first reduce to the lower-dimensional span
    of the signals before applying the results below.

    Finally, the reducible case will be explained after the linearized AMP
    matrix is introduced in the next subsection. The irreducibility assumption
    lets us state the outlier and eigenvector results without carrying a block
    decomposition of the views.
\end{remark}

\subsection{The linearized AMP matrix}
\label{sec:spectral_introduction}
We next introduce the linearized AMP matrix,  which is of central interest in this work.  
The guiding idea is simple. To estimate the latent vectors from the observation
matrices in~\eqref{eq:observation model}, one would like to combine information
across the $L$ views, using the covariance matrix $B$ to encode how the latent
signals in different views are related. A natural iterative implementation of
this idea is provided by approximate message passing. If this nonlinear
iteration is expanded to first order near the non-informative estimate, the
result is a linear spectral method. This linearization follows the AMP
formulation in~\cite{yang2024fundamental}; the derivation is postponed to
Section~\ref{sec:derivation}. For the main results below, the following
self-adjoint matrix is the focal object.
\begin{definition}[Linearized AMP matrix]\label{def:linearized-amp-matrix}
    Let $V=\sqrt{B}\in\RR^{L\times L}$ denote the symmetric square root of
    $B$. We define the linearized AMP matrix as
    \begin{equation}\label{eq:H}
        H = (V \otimes I_n) \Diag\rbr{\sqrt{\frac{\lambda_1}{n}} A^{(1)} - \lambda_1 I_n,\ldots,\sqrt{\frac{\lambda_L}{n}} A^{(L)} - \lambda_L I_n} (V \otimes I_n) \in \RR^{Ln \times Ln}.
    \end{equation}
\end{definition}
Our spectral estimator is based on a top eigenvector of $H$, so the first
step is to translate such an eigenvector back into the signal space. The matrix
$H$ acts on vectors in $\RR^{nL}$, while the latent signal is the matrix
$X\in\RR^{n\times L}$ in~\eqref{eq:observation model}. Thus an eigenvector
should be viewed as $L$ length-$n$ blocks, one for each view. For
$v=(v^{(1)},\ldots,v^{(L)})\in\RR^{nL}$, with
$v^{(l)}\in\RR^n$, define its matricization by
\begin{equation}\label{eq:matop-definition}
    \matop(v):=\sbr{v^{(1)},\ldots,v^{(L)}}\in\RR^{n\times L}.
\end{equation}
This operation only reshapes the vector and does not change its norm. Thus, if
$\nu_n$ is a unit eigenvector, then $\matop(\nu_n)$ has Frobenius norm one,
whereas the latent signal matrix $X$ has columns of norm of order $\sqrt n$.
To put the estimator on the same scale as the signal, we use the rescaled
matrix-valued estimator
\begin{equation}\label{eq:matrix-estimator-definition}
    \hat X(\nu_n):=\sqrt n\,\matop(\nu_n).
\end{equation}
As usual for an eigenvector, this estimator is defined only up to a global
sign. The central question is whether this estimator is correlated with
the latent signal $X$ in the asymptotic limit. We show that the answer
exhibits a sharp phase transition, governed by an explicit signal-to-noise
ratio.

\begin{remark}[Reducible covariance matrices]\label{remark:reducible-B}
    We can now explain why the irreducibility condition in
    Assumption~\ref{assump:irreducible-normalized-B}, in the sense of
    Definition~\ref{def:reducible-covariance}, loses no generality.
    If $B$ is reducible, then after permuting the views there is a partition
    $S_1\cup\cdots\cup S_k=[L]$ such that
    \begin{equation}
        B=\Diag(B_1,\ldots,B_k),\qquad
        \lambda=(\tilde\lambda_1,\ldots,\tilde\lambda_k),
        \qquad
        \dim(B_j)=\dim(\tilde\lambda_j)=|S_j|,
    \end{equation}
    where each $B_j$ is irreducible. The same permutation makes
    $V=\sqrt B$ block diagonal with blocks $V_j=\sqrt{B_j}$. Consequently,
    the matrix $H$ decomposes as a direct sum of the corresponding component
    matrices
    \begin{equation}
        H^{[j]}
        =
        (V_j\otimes I_n)
        \Diag\cbr{\sqrt{\lambda_l/n}\,A^{(l)}-\lambda_l I_n:\, l\in S_j}
        (V_j\otimes I_n),
        \qquad 1\le j\le k.
    \end{equation}
    Thus the irreducibility assumption is made without loss of generality for
    the analysis. It lets us state the outlier and eigenvector results without
    carrying this component decomposition throughout the paper.
\end{remark}

For our subsequent analysis, it is convenient to decompose $H$ into a low-rank
``signal component'' and a ``noise component.'' Recalling the observation model
in~\eqref{eq:observation model}, we write
\begin{equation}\label{eq:H_H0H1}
    H \; = \; H_0 \;+\; H_1,
\end{equation}
where
\begin{align}
    H_0 &\bydef (V \otimes I_n) \Diag\sbr{ \frac{\lambda_1}{n} X^{(1)} \rbr{ X^{(1)} }^\top, \ldots, \frac{\lambda_L}{n} X^{(L)} \rbr{ X^{(L)} }^\top } (V \otimes I_n), \label{eq:form of H0}\\
    H_1 &\bydef (V \otimes I_n) \Diag\rbr{ \sqrt{\frac{\lambda_1}{n}}W^{(1)}, \ldots, \sqrt{\frac{\lambda_L}{n}}W^{(L)} }(V \otimes I_n) - [V \Diag(\lambda) V] \otimes I_n. \label{eq:form of H1}
\end{align}

\subsection{The bulk law and spectral phase transition}
\label{sec:spectral_results}
We now state the limiting spectral law for the linearized AMP matrix $H$.
\begin{theorem}\label{thm:main ESD}
    Suppose that Assumptions~\ref{assump:Gaussian Wigner}--\ref{assump:irreducible-normalized-B}
    hold. There exists a compactly supported probability measure $\mu$ on
    $\RR$ such that, if
    \[
        \hat\mu_H:=\frac{1}{nL}\sum_{j=1}^{nL}\delta_{\sigma_j(H)}
    \]
    is the empirical spectral measure of $H$, then $\hat\mu_H$ converges
    weakly to $\mu$ in probability. Let
    \begin{equation}\label{eq:sigma-plus-def}
        \sigma_+:=\max\supp(\mu)
    \end{equation}
    denote the upper edge of this limiting bulk spectrum. Then, for every
    fixed $\epsilon>0$,
    \[
        \PP(
        \#\cbr{j\in[nL]:\sigma_j(H)>\sigma_++\epsilon}\le L
        )
        \to 1.
    \]
\end{theorem}
Theorem~\ref{thm:main ESD} is proved in
Section~\ref{sec:results on noise mat}. The limiting measure $\mu$ can be
specified explicitly through a matrix Dyson equation, a standard deterministic
self-consistent equation for correlated random matrix ensembles
\cite{ajanki2019stability,alt2019location}. Because our model has an $L$-view
block structure, this equation naturally produces a matrix-valued object before
one obtains the scalar spectral law. More precisely,
Section~\ref{sec:formal present self-cosnistent eq} constructs a compactly
supported $\SA_L^+(\RR)$-valued measure $\xi$ on $\RR$ with total mass
$\xi(\RR)=B$, characterized by this matrix Dyson equation. The limiting
measure $\mu$ in Theorem~\ref{thm:main ESD} is
\begin{equation}\label{eq:informal mu definition}
    \mu=\frac{1}{L}\tr\rbr{B^{-1}\xi}.
\end{equation}
In this sense, the scalar bulk law is the normalized trace projection of the
matrix-valued measure $\nu$. We write its matrix-valued Stieltjes transform as
\begin{equation}\label{eq:informal M stieltjes}
    M(z):=-\int_\RR \frac{\xi(\ud\tau)}{\tau-z},
    \qquad z\in\CC^+,
\end{equation}
so that $-(1/L)\tr\sbr{B^{-1}M(z)}$ is the Stieltjes transform of $\mu$.
Section~\ref{sec:formal present self-cosnistent eq} identifies $M(z)$ as the
valid solution of the corresponding matrix Dyson equation.


Note that the low-rank perturbation $H_0$ can create at most $L$ outlying
eigenvalues in the spectrum of $H$. We next study the outlying eigenvalues and
the recovery performance of the corresponding eigenvectors. The relevant
parameter is the signal-to-noise ratio $\SNR(\lambda,B)$ defined
in~\eqref{eq:summary SNR def}. It gives a sharp spectral phase transition in
the spectrum of $H$: when $\SNR\rbr{\lambda,B}<1$, there is no outlier
eigenvalue above the right edge $\sigma_+$ of the bulk spectrum; in contrast,
if $\SNR\rbr{\lambda,B} > 1$, there is an outlier
eigenvalue of $H$ near $1$ with high probability.
\begin{theorem}\label{thm: main uninformative}
    Suppose that Assumptions~\ref{assump:Gaussian Wigner}--\ref{assump:irreducible-normalized-B}
    hold, and recall that $\sigma_+$ is the upper edge of the limiting bulk
    spectrum defined in~\eqref{eq:sigma-plus-def}. Then $\sigma_+<1$ whenever
    $\SNR(\lambda,B)\neq1$, and the top eigenvalue of $H$ undergoes the
    following phase transition.
    \begin{itemize}
        \item[(i)] If $\SNR\rbr{\lambda,B}<1$, then no eigenvalue separates
        from the right edge of the bulk:
        \[
            \sigma_1(H)\pto \sigma_+ .
        \]

        \item[(ii)] If $\SNR\rbr{\lambda,B}>1$, then an outlier emerges at
        the point $1$:
        \[
            \sigma_1(H)\pto 1 .
        \]
        Moreover, with high probability as $n \to \infty$, this top eigenvalue is
        simple and separated from the rest of the spectrum: there exists
        $\rho\in[\sigma_+,1)$ such that $\sigma_2(H)\pto\rho$.
    \end{itemize}

\end{theorem}
The proof is given in
Subsection~\ref{subsec:proof-main-spectral-theorems}, after the deterministic
phase-transition analysis in Section~\ref{sec:optimizing RS free energy}.

In the informative phase $\SNR(\lambda,B)>1$, the last theorem still leaves
open the possibility of additional outliers to the right of the bulk spectrum.
Our next result allows us to characterize the number of outlying eigenvalues. 
\begin{proposition}\label{prop:main-outlier-count}
    Let $M_+$ denote the real boundary value of the matrix-valued Stieltjes
    transform~\eqref{eq:informal M stieltjes} at the right bulk edge
    $\sigma_+$; the existence of this boundary value is shown in
    Subsection~\ref{subsec:counting-outliers-right}. Define $L_0$ to be the
    number of eigenvalues of
    \[
        \sqrt{\Lambda} \sbr{B \odot M_+} \sqrt{\Lambda}
    \]
    that are strictly larger than $1$. Then there exists $\epsilon_0>0$ such
    that, for every fixed $\epsilon\in(0,\epsilon_0)$ and $D>0$,
    \begin{equation}
        \PP\rbr{
        \#\cbr{j\in[nL]:\sigma_j(H)>\sigma_+ + \epsilon}=L_0
        }
        \ge 1-n^{-D}
    \end{equation}
    for all sufficiently large $n$. In particular, if
    $\SNR\rbr{\lambda,B}>1$, then $L_0\ge1$ and the first eigenvalue after the
    outlier group returns to the bulk edge:
    \[
        \sigma_{L_0+1}(H)\pto\sigma_+ .
    \]
\end{proposition}
Proposition~\ref{prop:main-outlier-count} is proved in
Subsection~\ref{subsec:counting-outliers-right}.
Unlike the leading outlier, whose emergence is governed by the simple formula
$\SNR(\lambda,B)>1$, the later outliers do not appear to admit an equally
simple closed-form threshold. Proposition~\ref{prop:main-outlier-count}
nevertheless gives a finite-dimensional characterization of how many such
outliers occur, and this criterion can be evaluated numerically; see the
illustration in Figure~\ref{fig:top-eigenval-overlap}.

We now turn from eigenvalues to recovery. In the informative regime
$\SNR\rbr{\lambda,B}>1$, let $\nu_n$ be a top eigenvector of $H$ and use the
matrix-valued estimator $\hat X(\nu_n)=\sqrt n\,\matop(\nu_n)$ introduced
in~\eqref{eq:matrix-estimator-definition}. The next theorem identifies its asymptotic
overlap with the latent signal matrix $X$.
\begin{theorem}\label{thm: main informative}
    Suppose that Assumptions~\ref{assump:Gaussian Wigner}--\ref{assump:irreducible-normalized-B}
    hold and that $\SNR\rbr{\lambda,B}>1$.
    The matrix-valued Stieltjes transform~\eqref{eq:informal M stieltjes}
    admits a real value $M(1)\in\SA_L^{+}(\RR)$ at $z=1$, and this value
    satisfies $M(1) \prec B$.

    Let $\nu_n\in\RR^{nL}$ be a unit eigenvector associated with the largest
    eigenvalue of $H$, and set $\hat X=\sqrt n\,\matop(\nu_n)$. Then, after choosing the
    global sign of $\nu_n$ appropriately,
    \begin{equation}\label{eq:main-overlap-limits}
    \begin{aligned}
        \frac{1}{n} \hat X^\top X
        &\pto
        \Sigma_1
        :=
        \frac{1}{\sqrt{c^{\ast}}} V^{-1}\sbr{B-M(1)}, \\
        \frac{1}{n}\hat X^\top \hat X
        &\pto
        \Sigma_2
        :=
        \frac{1}{c^{\ast}}\sbr{I-V^{-1}M(1)V^{-1}}.
    \end{aligned}
    \end{equation}
    where
    \[
        c^{\ast}
        :=
        \sum_{l=1}^L \lambda_l M_{ll}(1)\sbr{1-M_{ll}(1)}
        >0.
    \]
    In particular, $\Sigma_1\neq0$, so the estimator $\hat X$ has nontrivial
    asymptotic correlation with the latent signal.
\end{theorem}
The proof is given in
Subsection~\ref{subsec:proof-main-spectral-theorems}; the overlap formulas are
obtained from the general outlier analysis in Section~\ref{sec:outlier}.


Together, Theorems~\ref{thm: main uninformative}
and~\ref{thm: main informative} give a sharp spectral transition for the
model~\eqref{eq:observation model}. Below the threshold, the top eigenvalue
does not separate from the right edge of the bulk. Above the threshold, a
distinguished outlier appears at $1$, and Theorem~\ref{thm: main informative}
shows that its eigenvector yields a matrix-valued estimator with nonvanishing
overlap with the latent signal.

\subsection{Information-theoretic thresholds}
\label{sec:lower_bounds}

The spectral results above identify $\SNR(\lambda,B)=1$ as the threshold for
the linearized AMP estimator. It is then natural to ask whether this threshold
is only algorithmic, or whether it also reflects an intrinsic statistical
barrier. Building on the free-energy analysis in \cite{yang2024fundamental},
we show that, for several natural classes of spikes, the same threshold is
information-theoretic: below it, no estimator can have nonvanishing overlap
with the planted signal, while above it Theorem~\ref{thm: main informative}
gives weak recovery by the spectral estimator. Thus, for these problems, the
spectral method is optimal at the level of the weak-recovery threshold. 
We formulate the lower bound
under the following distributional assumption on the spikes.

\begin{assumption}\label{assump:distributional-spike}
    Every $(X_i^{(1)},\ldots,X_i^{(L)})$ is drawn i.i.d. from a distribution $p(x)$ with mean zero and second moment $B$, such that Assumption~\ref{assump:spike concentration} holds.
\end{assumption}
The following proposition gives two sufficient conditions under which the
information-theoretic threshold coincides with the spectral threshold
$\SNR(\lambda,B)\lessgtr1$. We
expect that $\SNR(\lambda,B)=1$ is the fundamental weak-recovery threshold
more generally, but a full characterization of the class of priors which satisfy this equivalence is an intriguing open question. Additionally, based on similar results in other high-dimensional models \cite{krzakala2013spectral,perry2018optimality,lelarge2019fundamental}, we conjecture $\SNR(\lambda,B)=1$ to be the algorithmic threshold for weak recovery in the multi-view spiked matrix problem \eqref{eq:observation model}.  

\begin{proposition}\label{prop:impossibility result strictly sub-gaussian}
    Under Assumption~\ref{assump:distributional-spike}, suppose further that $p(\cdot)$ satisfies at least one of the two conditions below:
    \begin{enumerate}
        \item[\hypertarget{IT-concave}{$(\sharp)$}] \textbf{Concave state evolution mapping.} Define the state evolution mapping $T:[0,+\infty)^L\to[0,+\infty)^L$ as
        \begin{equation}
            T_l(\gamma) = \EE\sbr{ \mathscr{X}_l \, \frac{\int_{\RR^L} x_l \exp\cbr{ -\frac{1}{2}\sum_{l=1}^L (\sqrt{\gamma_l}\mathscr{X}_l+\mathscr{W}_l-\sqrt{\gamma_l}x_l)^2 } \ud p(x)}{\int_{\RR^L} \exp\cbr{ -\frac{1}{2}\sum_{l=1}^L (\sqrt{\gamma_l}\mathscr{X}_l+\mathscr{W}_l-\sqrt{\gamma_l}x_l)^2 } \ud p(x)} },
        \end{equation}
        where the expectation is evaluated with respect to independent random vectors   $\mathscr{X}\sim p(x)$ and $\mathscr{W}\sim\cN(0,I_L)$. We require that for any fixed $\gamma\in[0,+\infty)^L$ with $\gamma\neq0$, the mapping $t\in[0,+\infty)\mapsto T_l(t\gamma)$ is strictly concave for any $l\in[L]$.

        \item[\hypertarget{IT-strict-sg}{$(\flat)$}] \textbf{Strictly sub-Gaussian overlaps.} For any $a\in(0,+\infty)^{L}$ and $\mu\in\RR^L$,
        \begin{equation}\label{eq: def of strictly sub-gaussian overlap}
            \log\left[\mathbb{E}\exp\left(\sum_{l=1}^L \mu_l a_l \mathscr{X}^{(l)}\widebar{\mathscr{X}}^{(l)}\right)\right] \le \frac{1}{2}\|\mu\|^2\sigma_{\max}\sbr{  \Diag(a)B^{\odot 2}\Diag(a) },
        \end{equation}
        where $\mathscr{X},\widebar{\mathscr{X}}$ are drawn independently from $p(x)$. This definition originates from Definition A.10 of \cite{yang2024fundamental}.
    \end{enumerate}
    Then, whenever $\SNR(\lambda,B)<1$, every estimator
    $\hat{X}(A)\in\RR^{n \times L}$ computed from the observation
    model~\eqref{eq:observation model} has vanishing overlap with the planted
    signal:
    \begin{equation}\label{eq:no-overlap-main-it}
        \frac{1}{n} \hat{X}(A)^\top X \pto 0_{L \times L}.
    \end{equation}
\end{proposition}

The conditions \hyperlink{IT-concave}{$(\sharp)$} and \hyperlink{IT-strict-sg}{$(\flat)$} can be difficult to interpret at first glance. The following corollary presents concrete priors that are covered by the sufficient conditions introduced above.

\begin{corollary}
    \label{remark:IT example}
    We provide two canonical examples that satisfy the previous conditions:
    \begin{enumerate}[label=(\alph*),ref=(\alph*)]
        \item\label{remark:IT example a} When $p(\cdot)=\cN(0,B)$, condition \hyperlink{IT-concave}{$(\sharp)$} holds.

        \item\label{remark:IT example b} When $x$ is supported on $\cbr{\pm1}^L$ and $\EE x=0$, many priors $p(x)$ with covariance matrices $B$ satisfy condition \hyperlink{IT-strict-sg}{$(\flat)$}. Interesting instances arise from community detection on multiple networks, and include: (i) $B=(1-\rho^2)I_L + \rho^2 \mathbf{1}_L\mathbf{1}_L^\top$ which corresponds to the hierarchical structure in \cite{chen2022global}; (ii) $B_{l_1l_2}=\rho^{|l_1-l_2|}$ which corresponds to the dynamical structure in \cite{matias2017statistical}.
    \end{enumerate}
\end{corollary}

Combining Corollary~\ref{remark:IT example} with
Theorem~\ref{thm: main informative}, these canonical priors exhibit no gap
between the spectral weak-recovery threshold and the information-theoretic
threshold for the problem~\eqref{eq:observation model}. The proofs of
Proposition~\ref{prop:impossibility result strictly sub-gaussian} and
Corollary~\ref{remark:IT example} are given in
Appendix~\ref{sec:IT proof}.

\subsection{Related work}
\label{sec:prior_works}
In this section, we situate our results in the wider context of high-dimensional estimation and highlight some connections with recent progress in this area.

\paragraph{Spiked matrices and PCA.}
Principal Component Analysis (PCA) is a classical technique for recovering low-dimensional hidden signals from high-dimensional data. Spiked random matrices \cite{johnstone2001distribution} have emerged as a natural testbed for studying the fundamental limits of signal recovery in this context. The celebrated Baik-Ben Arous-Peche transition \cite{baik2005phase} characterizes the behavior of the top eigenvalue and the associated eigenvector: above a critical threshold, the spectrum has an outlier eigenvalue and the top eigenvector is correlated with the latent signal, while below the threshold the top eigenvector is uncorrelated with the signal. Our results generalize this phenomenon to the multi-view spiked matrix setting \eqref{eq:observation model}.

\paragraph{Multi-view spiked models.}
Multi-view spiked models of the form \eqref{eq:observation model} have been investigated recently by several authors (see e.g. \cite{nandy2024multimodal,yang2024fundamental,reeves2020information,keup2025optimal} and references therein). In \cite{nandy2024multimodal}, these models are motivated by applications in single-cell data analysis. They also arise naturally in past work by the first two authors on multi-layer network models \cite{yang2024fundamental}. Prior works have developed recovery algorithms based on Approximate Message Passing (AMP) \cite{nandy2024multimodal, yang2024fundamental,rossetti2023approximate}. A practical challenge is that AMP algorithms require a warm start, which is not immediately available. Prior work has investigated optimal recovery in spiked matrices by AMP algorithms initialized with an appropriate spectral estimator \cite{montanari2021estimation}. It would be interesting to explore the performance of AMP algorithms initialized with the spectral estimator introduced here. This is significantly beyond the scope of the current manuscript; we leave it for future work.

Recent work~\cite{li2025algorithmic} develops algorithms for signal estimation in two-view spiked random matrix models based on subgraph counts of specific cycles. Subgraph-counting algorithms are typically quasi-polynomial time and are therefore less efficient than the spectral methods developed here. Finally, \cite{mergny2025spectral} analyzes the performance of the Partial Least Squares (PLS) algorithm in this setup. This algorithm is suboptimal compared to the spectral method introduced here and requires a higher SNR to guarantee signal recovery.

\paragraph{Connections with community detection.} The multi-view spiked Wigner model \eqref{eq:observation model} is closely connected with the community detection problem on multi-layer networks. In this problem, one observes $L \geq 2$ graphs on the same set of $[n]$ vertices. Let $\{E_l : l \in [L]\}$ denote the edges in layer $l$. The graphs are constructed as follows: each vertex $i \in [n]$ has a ``local'' community assignment in each layer $l \in [L]$, denoted $X_i^{(l)} \in \{\pm 1\}$. Given the local community assignments, the edges are sampled as
\begin{align}
    \mathbb{P}( \{i,j \} \in E_{l}) = \begin{cases}
     \frac{a_l}{n} &\textrm{if}\,\,\, X_i^{(l)}= X_j^{(l)}, \\
     \frac{b_l}{n}  &\textrm{otherwise}.
    \end{cases} \nonumber 
\end{align}
The edges are all independent, and $a_l$ and $b_l$ are $n$-dependent sequences that govern the connection probabilities in layer $l$. The main task in this context is to recover the local communities $\{X_i^{(l)} : i \in [n], l \in [L]\}$ from the observed graphs. To connect this setting with~\eqref{eq:observation model}, we encode the data through the layer-wise adjacency matrices $A^{(l)}(i,j)=\mathbf{1}(\{i,j\}\in E_l)$. Setting
\begin{align}
    d_l = \frac{a_l + b_l}{2}, \,\,\,\, \lambda^{(l)} = \frac{(a_l - b_l)^2}{2 (a_l + b_l)}, \nonumber
\end{align}
one obtains the usual centered and rescaled adjacency matrix
\begin{align}
    \bar{A}^{(l)}(i,j)
    =
    \frac{A^{(l)}(i,j)-d_l/n}{\sqrt{(d_l/n)(1-d_l/n)}}. \nonumber
\end{align}
Assuming $1\ll d_l \ll n$ and $\lambda_l = O(1)$, the leading signal term of $\bar{A}^{(l)}(i,j)$ is of size $\sqrt{\lambda_l/n}\,X_i^{(l)}X_j^{(l)}$, while the variance is of order one. In this sense, the Gaussian model~\eqref{eq:observation model} is the dense random-matrix analogue of the multi-layer community-detection model. This connection is well known for spiked random matrices \cite{deshpande2017asymptotic}, and was exploited in past work by the first two authors~\cite{yang2024fundamental}. Given the ubiquity of universality phenomena in random matrix theory, we expect that the spectral algorithm introduced here can be adapted directly to community detection on multi-layer networks, provided the observed networks are sufficiently dense. We leave a rigorous proof of this universality to future work.

This conceptual connection has also been leveraged in other recent works. Building on the ideas of \cite{li2025algorithmic}, the follow-up work \cite{gong2026fundamental} extends the walk-based approach to establish the information-theoretic threshold in the sparse graph version of the multi-layer network model investigated in \cite{yang2024fundamental}. This resolves several questions left open in \cite{yang2024fundamental}.

\paragraph{Spectral methods from linearized message passing.}
Our spectral algorithm is obtained by linearizing an appropriate AMP algorithm near the ``uninformative'' fixed point. This general strategy has been successful for several high-dimensional inference problems. For spiked matrices, this recipe reduces to the naive spectral method, which computes the top eigenvector of the observed matrix. In other problems, this approach suggests spectral operators that attain the information-theoretic threshold for signal recovery. Notable examples include the non-backtracking operator for community detection \cite{krzakala2013spectral}, spectral methods for single-index models \cite{lu2020phase, mondelli2018fundamental,luo2019optimal}, contextual stochastic block models \cite{bandeira2024matrix}, and spiked models with heterogeneity \cite{mergny2024spectral}.

The linearized operators are analyzed using different techniques depending on the application. For dense linearized operators, prior work uses techniques from AMP \cite{zhang2022precise}, develops new tools from free probability \cite{bandeira2024matrix}, or uses well-established results from random matrices \cite{mergny2024spectral}. In our setting, the correlated multi-view noise leads naturally to a matrix Dyson equation and to a finite-dimensional outlier equation. The main task is to combine these deterministic random-matrix tools with the spike structure carefully enough to obtain explicit eigenvalue and eigenvector formulas.

It is particularly fruitful to compare our results to the recent work \cite{mergny2024spectral}, which studies the recovery of $X\in\RR^N$ from $ Y=\sqrt{\frac{1}{N}}XX^{\top} + H\odot \rbr{\Delta^{\odot 1/2}}$, where $H$ has i.i.d. entries and $\Delta$ admits a block structure. The entries of the spike $X$ are assumed to be i.i.d. The authors study the spectral operator obtained by linearizing AMP. While the strategies adopted in the two works are similar, there are crucial differences in the random-matrix objects that appear. In contrast to the setting in \cite{mergny2024spectral}, we study the effect of correlations among the entries of the spike vector. The resulting noise matrix \eqref{eq:form of H1} has a structured correlated form, and its deterministic description is expressed through the matrix Dyson equation used throughout this paper. The spiked matrix with heterogeneous noise has also been analyzed using the Lehner formula \cite{lehner1999computing} in \cite{bandeira2024matrix}. It would be interesting to understand whether that free-probabilistic perspective can also be connected to the present multi-view model.


Prior analyses of spectral methods using Approximate Message Passing (AMP) \cite{zhang2024matrix,zhang2024spectral} have also influenced our analysis. In these works, one tracks the top eigenvector using an appropriate AMP algorithm, and the performance of the spectral estimator is characterized through AMP state evolution. In our setting, it is challenging to rigorously establish this correspondence between the spectral estimator and an appropriate AMP algorithm, since we lack detailed information on spectral gaps in the bulk spectrum. However, state evolution suggests simplified expressions for the asymptotic overlaps in Theorem~\ref{thm: main informative}. Inspired by this prediction, we provide an independent proof of the correctness of these expressions. We defer additional details to Appendix~\ref{sec:state evolution heuristic}.

\paragraph{Concurrent work.}
While completing this manuscript, we became aware of concurrent work by
Du, Hu, and Lepsveridze~\cite{du2026twoview}. Their paper develops 
spectral algorithms for several correlated two-view models, including
high-dimensional canonical correlation analysis and the two-view correlated
spiked Wigner and Wishart models. In particular, their correlated spiked Wigner
model overlaps with the special case $L=2$ of our setting.

There are several important differences between the two works. First, in the
two-view setting, \cite{du2026twoview} constructs spectral procedures that are
agnostic to the unknown model parameters: after replacing the correlation
parameter by its critical value, they search over candidate signal strengths
and detect separation from a common bulk edge. This parameter-free feature is
statistically attractive, since the covariance structure of the latent signals
may not be known a priori. By contrast, the matrix studied in the
present paper uses the limiting Gram matrix $B$ explicitly. On the other hand,
our analysis applies to an arbitrary fixed number of views $L\ge2$ and to a
general positive definite limiting Gram matrix $B$. Extending parameter-agnostic
spectral methods beyond the special two-view structure remains an interesting
open problem.

Second, one of the main contributions of the present paper is the explicit
phase-transition formula $\SNR(\lambda,B)=1$, with $\SNR(\lambda,B)$ defined
in~\eqref{eq:summary SNR def}, for an arbitrary fixed number of views. This
formula is written directly in terms of the signal-strength vector $\lambda$
and the full limiting Gram matrix $B$. In the overlapping two-view Wigner case,
the threshold in~\cite{du2026twoview} is expressed in the native parameters of
their model; after translating notation, it agrees with the $L=2$
specialization of our formula. For general $L$, however, the problem is
genuinely multidimensional. Our proof identifies the threshold by connecting
the outlier equation for the linearized AMP matrix with the variational
structure of the replica-symmetric free energy with a Gaussian prior
\cite{yang2024fundamental}. This step requires a separate deterministic
analysis of the matrix Dyson equation at $z=1$.



\paragraph{Matrix Dyson equation.}
To characterize the emergence of outliers in our setting, we analyze the
noise matrix $H_1$ in~\eqref{eq:form of H1} through a matrix Dyson equation.
This connects our work to a broad random matrix literature on correlated
ensembles and Kronecker random matrices
\cite{ajanki2019stability,alt2019location,alt2020correlated,alt2020dyson}.
In particular, the Hermitian Kronecker framework of~\cite{alt2019location}
contains the noise component $H_1$ and provides important input on the location
of its spectrum. What remains specific to the present paper is the way this noise
analysis is coupled to the finite-rank signal component: we need deterministic
equivalents adapted to the spike directions, matrix-valued overlap formulas,
and a separate finite-dimensional analysis of the outlier equation. These are
the ingredients that lead to the explicit phase-transition formula
$\SNR(\lambda,B)=1$.

\subsection{Numerical results}
\label{sec:experiments}
We close this section with numerical illustrations of the preceding results.
The experiments are designed to show the three main spectral phenomena in a
finite-dimensional sample: the deterministic bulk law, the emergence of the
distinguished outlier at the phase transition, and the agreement between the
predicted and empirical eigenvector overlaps. The theoretical curves are
computed from the finite-dimensional deterministic equations associated with
the matrix Dyson equation developed later in the paper, while the empirical
points come from direct diagonalization of the linearized AMP matrix
$H$ in~\eqref{eq:H}.

Figure~\ref{fig:density comparison} compares the empirical spectral
distribution with the limiting bulk density in three regimes. In the left
panel, the signal-to-noise ratio is below the transition, and the spectrum is
well described by the bulk law with no visible separated eigenvalue. The middle
panel is close to the transition and illustrates the movement of the right
bulk edge toward the critical location. In the right panel, the model is in the
informative regime: the bulk is still accurately captured by the limiting
measure, while an additional eigenvalue separates and appears near the
distinguished point $1$. The agreement in all three panels supports both the
bulk law in Theorem~\ref{thm:main ESD} and the spectral transition described
in Theorem~\ref{thm: main uninformative}.

\begin{figure}[t]
    \centering
    \includegraphics[width=\linewidth]{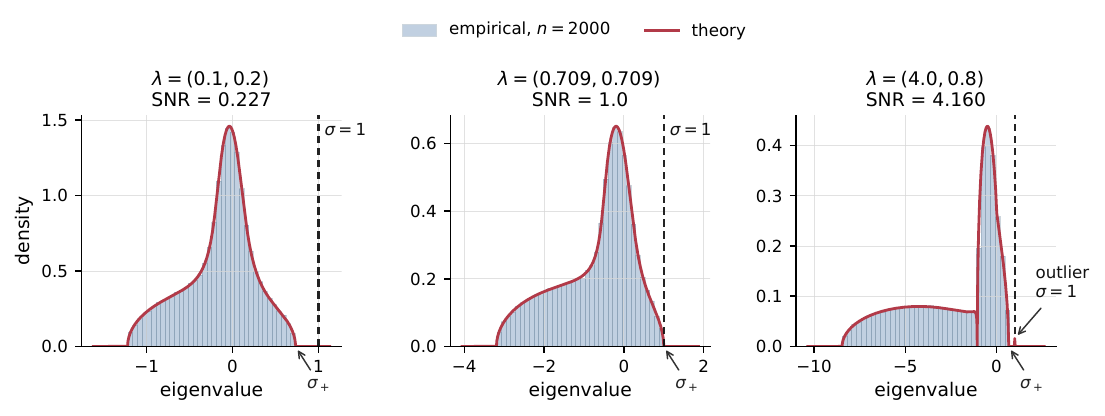}
    \caption{Comparing empirical spectral distribution with theoretical predictions. Throughout, we take $L=2$ and $B=0.64 \mathbf{1}_2\mathbf{1}_2^\top + 0.36 I_2$. Going from left to right, we take $\lambda=(0.1,0.2)$, $\lambda=(0.709,0.709)$ and $\lambda=(4.0,0.8)$ respectively. To simulate the histograms, we set $n=2000$ and sample Bernoulli spikes $X\in\{\pm 1\}^{n\times L}$ with the given covariance $B$. Then the spectrum is computed for $H$ defined in~\eqref{eq:H}. The arrows mark the right edge $\sigma_+$ of the bulk spectrum; in the right panel, the small bump at $\sigma=1$ highlights the outlier.}\label{fig:density comparison}
\end{figure}

Figure~\ref{fig:top-eigenval-overlap} gives a more detailed view of the
outlier behavior and the associated recovery formulas along a one-parameter
ray $\lambda=t(0.5,0.8,1.3)$ for $t>0$. Panel (a) tracks the leading
eigenvalues as $t$ increases. The first dashed line is the threshold
$\SNR(\lambda,B)=1$: after this point the leading outlier is pinned at $1$, as
predicted by Theorem~\ref{thm: main uninformative}. The later dashed lines mark
the emergence of additional outliers. These subsequent thresholds are not
governed by the simple scalar criterion $\SNR(\lambda,B)=1$, but they are
captured by the finite-dimensional outlier equations underlying
Proposition~\ref{prop:main-outlier-count}. Panel (b) compares the scalar
overlaps between the eigenvector and the signal directions in the rank-$L$
spike component. Panels (c) and (d) then display the matrix-valued overlaps of
the signal-scale estimator $\hat X=\sqrt n\,\matop(\nu_n)$: panel (c) shows
the Frobenius norms of the two overlap matrices, while panel (d) shows the
diagonal entries of $(1/n)\hat X^\top X$. These are the quantities predicted
by Theorem~\ref{thm: main informative}.

\begin{figure}[!tbp]
    \centering
    \includegraphics[width=0.98\linewidth]{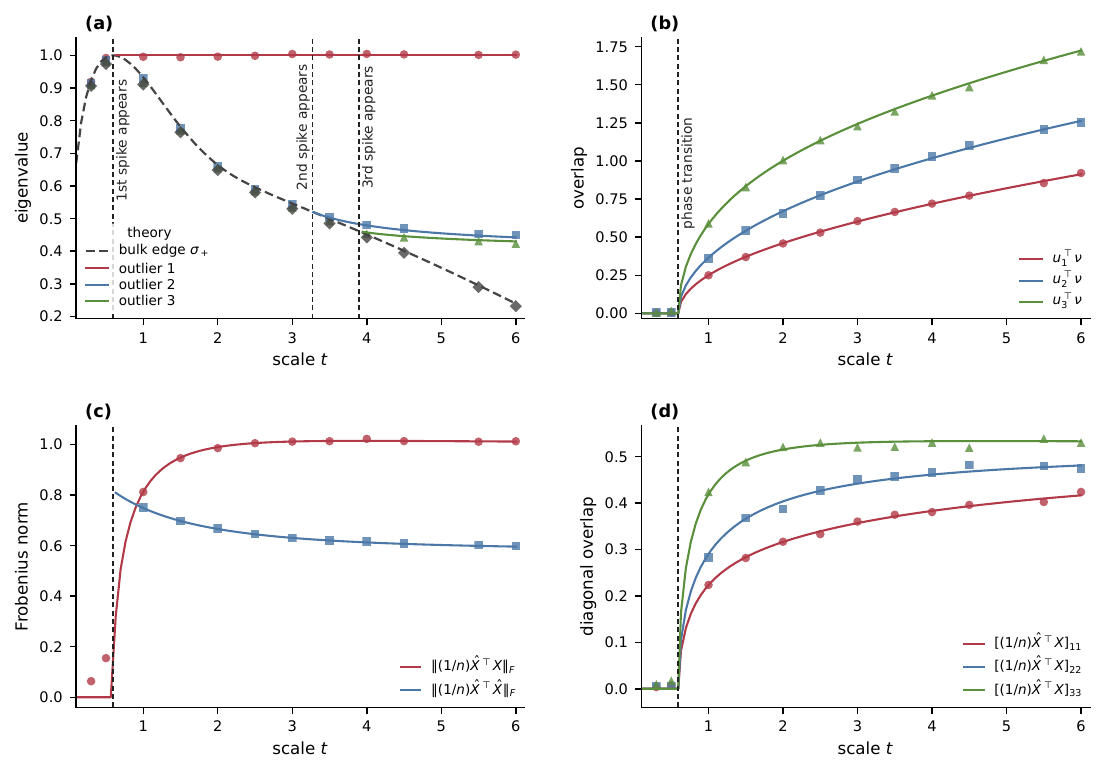}
    \caption{Comparing empirical top eigenvalues and overlaps with theoretical predictions. Throughout, we take $L=3$ and $B=0.64 \mathbf{1}_3\mathbf{1}_3^\top + 0.36 I_3$. In all panels, we use $\lambda=t(0.5,0.8,1.3)$ with $0.1\le t\le 6$. For $t\in\cbr{0.3, 0.5, 1.0, 1.5, 2.0, 2.5, 3.0, 3.5, 4.0, 4.5, 5.5, 6.0}$, each scattered point averages $10$ independent numerical simulations with $n=3000$; the curves are theoretical predictions. Panel (a) shows the top $4$ eigenvalues, with the $4$-th tracking the bulk edge $\sigma_+$. The vertical dashed lines mark where the first, second, and third spike eigenvalues appear. Panel (b) compares $u_l^\top \nu_n$ for each $l\in[L]$, where $\nu_n$ is the eigenvector whose eigenvalue is closest to $1$ and $u_l$ is the $l$th signal direction in the rank-$L$ factorization of the spike component. Panel (c) compares $\|(1/n)\hat X^\top X\|_F$ with $\|\Sigma_1\|_F$ and $\|(1/n)\hat X^\top\hat X\|_F$ with $\|\Sigma_2\|_F$, where $\hat X=\sqrt n\,\matop(\nu_n)$ is the signal-scale matrix estimator and $\Sigma_1,\Sigma_2$ are defined in~\eqref{eq:main-overlap-limits}. Panel (d) compares the diagonal entries of $(1/n)\hat X^\top X$ with those of $\Sigma_1$.}\label{fig:top-eigenval-overlap}
\end{figure}

\section{Proof outlines}\label{sec:proof_outline}

This section gives the proof roadmap for the main spectral claims of Section~\ref{sec:spectral_results}. Our argument has three main components. First, we identify the limiting bulk spectrum of the linearized AMP operator $H$ in~\eqref{eq:H} by using the decomposition $H=H_0+H_1$ in~\eqref{eq:H_H0H1} and analyzing the resolvent of the noise component $H_1$ through the matrix Dyson equation developed in Section~\ref{sec:formal present self-cosnistent eq}. Second, we use the finite-rank structure of $H_0$ to reduce the outlier eigenvalues, together with the associated eigenvector overlaps, to a deterministic $L$-dimensional equation. Third, we analyze this deterministic equation at the special point $z=1$ and show that its behavior changes exactly when the signal-to-noise ratio $\SNR(\lambda,B)$ defined in~\eqref{eq:summary SNR def} crosses one. The goal of this section is to make the proof readable before the technical estimates begin, and to serve as an atlas for where those estimates are established.

\paragraph{Stochastic domination.}
We will use stochastic domination throughout the proofs as a convenient way to
quantify the typical size of random quantities. We use the standard notion from
\cite{erdHos2013averaging,erdHos2017dynamical,erdos2013local}. Let
\[
    X=\rbr{X^{(n)}(u):n\in\NN,u\in U_n},\qquad
    Y=\rbr{Y^{(n)}(u):n\in\NN,u\in U_n},
\]
where $U_n$ may depend on $n$. We say that $X$ is stochastically dominated by
$Y$, uniformly in $u\in U_n$, if for every $\epsilon>0$ and $D>0$,
\[
    \sup_{u\in U_n}
    \PP\sbr{\abr{X(u)}>n^\epsilon\abr{Y(u)}}\le n^{-D}
\]
for all sufficiently large $n\ge n_0(\epsilon,D)$. In this case we write
$X=\cO_{\prec}(Y)$.

\subsection{The bulk spectrum}
\label{subsec:proof-roadmap-bulk}

We first explain the mechanism behind the bulk law in Theorem~\ref{thm:main ESD}; the formal proof is given in Section~\ref{sec:results on noise mat}. Recall from~\eqref{eq:H} and~\eqref{eq:H_H0H1} that the linearized AMP operator can be decomposed as $H=H_0+H_1$. Here $H_0$ is the signal component and has rank at most $L$, while $H_1$ is the noise component. This distinction is useful because the two pieces affect different parts of the spectrum. The finite-rank matrix $H_0$ can create a bounded number of outlying eigenvalues, but it cannot change the limiting empirical spectral distribution. Consequently, the bulk spectrum of $H$ is governed by the noise matrix $H_1$.

The main task is therefore to identify the limiting spectral distribution of $H_1$. This is done through its resolvent, for $z\in\CC^+$,
\begin{equation}\label{eq:G_H1}
    G(z) = (H_1-zI_{nL})^{-1}.
\end{equation}
The deterministic approximation to $G(z)$ is expressed in terms of the matrix-valued Stieltjes transform $M(z)$ introduced in~\eqref{eq:informal M stieltjes}. In the notation of Section~\ref{sec:formal present self-cosnistent eq}, this is the case $\Gamma=0$ of the perturbed matrix Dyson equation introduced there. Recall from~\eqref{eq:informal mu definition} that $\mu$ is obtained from the positive semidefinite matrix-valued measure $\nu$ by taking the normalized trace against $B^{-1}$. The technical input we use here is the resolvent comparison below.

\begin{proposition}[Resolvent deterministic equivalent]\label{prop:bulk-roadmap-input}
    Fix $\epsilon>0$, and let $\cU(\epsilon)$ be the spectral domain in~\eqref{eq:domain for unfirom convergence}. Uniformly for $z\in\cU(\epsilon)$, define
    \begin{equation}
        \cG(z):=[-V^{-1}M(z)V^{-1}]\otimes I_n.
    \end{equation}
    Then for every deterministic matrix $A\in\CC^{nL\times nL}$,
    \begin{equation}
        \tr\sbr{A(G(z)-\cG(z))}
        =
        \cO_{\prec}\rbr{
            \frac{\|A\|_F}{\sqrt n(\Im z)^4}
        }.
    \end{equation}
\end{proposition}

Proposition~\ref{prop:equivalent} in Section~\ref{sec:results on noise mat} is a more general version of Proposition~\ref{prop:bulk-roadmap-input}. It says that the random resolvent $G(z)$ is asymptotically indistinguishable, when tested against deterministic matrices, from the block-constant deterministic matrix $\cG(z)$. Choosing $A=I_{nL}/(nL)$ gives the convergence of normalized traces in the form
\begin{equation}
    \frac{1}{nL}\tr G(z)
    +
    \frac{1}{L}\tr\sbr{B^{-1}M(z)}
    =
    \cO_{\prec}\rbr{\frac{1}{n(\Im z)^4}}.
\end{equation}
By the Stieltjes representation from Proposition~\ref{prop:feasibility}, the deterministic term is the Stieltjes transform of the self-consistent density of states $\mu$ introduced in~\eqref{eq:informal mu definition}. This identifies the limiting empirical spectral distribution of $H_1$. Section~\ref{sec:results on noise mat} also proves the companion no-outside-spectrum estimate, Lemma~\ref{lemma:no-eigenvalue outside support}, which rules out eigenvalues of $H_1$ away from $\supp(\mu)$ with high probability.

\begin{remark}
    The noise matrix $H_1$ is a correlated Gaussian matrix with a Kronecker
    structure, so the matrix Dyson equation literature
    \cite{ajanki2019stability,erdos2019matrix,alt2019location,alt2020correlated,alt2020dyson}
    provides the natural starting point for the analysis. In this paper we keep
    the resulting deterministic equation in an explicit finite-dimensional form,
    which is useful both for the stability estimates in
    Section~\ref{sec:formal present self-cosnistent eq} and for the
    spike-direction deterministic equivalents in
    Section~\ref{sec:results on noise mat}.
\end{remark}

\subsection{Outlier eigenvalues and eigenvectors}
We next explain how the finite-rank signal component creates the outliers stated in Theorems~\ref{thm: main uninformative} and~\ref{thm: main informative}. Recall from~\eqref{eq:H_H0H1} that $H=H_0+H_1$, where $H_1$ is the noise matrix whose bulk law was discussed above. The signal component $H_0$ has rank at most $L$ and can be factored as $H_0=UU^\top$, with $U\in\RR^{nL\times L}$ given by
\begin{equation}
    U=\sbr{u_1,\ldots,u_L}\in\RR^{nL\times L},\text{  with  } u_l = \sqrt{\frac{\lambda_l}{n}}\begin{bmatrix}
        V_{1l} \\
        V_{2l} \\
        \cdots \\
        V_{Ll}
    \end{bmatrix} \otimes X^{(l)}\in\RR^{nL\times 1},\quad\forall l\in[L].
\end{equation}
Consequently, the question of eigenvalues outside the bulk reduces to a fixed-dimensional one. Recall from~\eqref{eq:G_H1} that $G(z)$ denotes the resolvent of the noise component, initially for $z\in\CC^+$. The same formula extends to real $z$ outside $\spec(H_1)$. Whenever the inverses below exist, the Woodbury formula gives
\begin{align}
    (H-zI_{nL})^{-1} &= G(z) - G(z) U \cQ_n(z)^{-1} U^{\top} G(z), \label{eq:Resolvent of spiked matrix model}\\
    \cQ_n(z) &:= I_{L} + U^\top G(z) U \in \CC^{L\times L}.
\end{align}
Thus, away from the spectrum of $H_1$, a point $z$ is an eigenvalue of $H$ precisely when $\cQ_n(z)$ is singular. This is the same determinant reduction that appears in finite-rank deformations of Wigner matrices~\cite{benaych2011eigenvalues}, but here the limiting resolvent is matrix-valued because the views are correlated.

The deterministic equivalent from Proposition~\ref{prop:bulk-roadmap-input}, together with the spike concentration assumption in Assumption~\ref{assump:spike concentration}, gives the deterministic limit of the finite-dimensional matrix $\cQ_n(z)$:
\begin{equation}
    \cQ(z) := I_L - \sqrt{\Lambda} \sbr{B \odot M(z)} \sqrt{\Lambda} \in \CC^{L\times L}.
\end{equation}
Here $M(z)$ is the matrix-valued Stieltjes transform introduced in~\eqref{eq:informal M stieltjes}; equivalently, it is the $\Gamma=0$ valid solution of the matrix Dyson equation in Section~\ref{sec:formal present self-cosnistent eq}. The measure $\mu$ whose support defines the bulk was introduced in~\eqref{eq:informal mu definition} and Definition~\ref{def:bulk dist limit}. The candidate outlier locations are therefore the real solutions outside the bulk support:
\[
    \det \cQ(\sigma)=0,\qquad \sigma\in\RR\backslash\supp(\mu).
\]
This determinant reduction is rigorously established in Subsection~\ref{subsec:outlier-equation} and formalized as follows.

\begin{proposition}[Outlier eigenvalue equation]\label{prop:sketch-statement-outlier}
    Let
    \[
        \Psi_0=\cbr{\sigma\in\RR\backslash\supp(\mu):\det\sbr{\cQ(\sigma)}=0},
    \]
    with roots counted according to the order of the zero of $\det\cQ$. For any fixed $\epsilon>0$ small enough, and for all sufficiently large $n$, the elements of $\Psi_0$ are in one-to-one correspondence with the eigenvalues of $H$ outside $\supp\rbr{\mu}+(-\epsilon,\epsilon)$, up to multiplicity.
    If $\sigma\in\Psi_0$ has multiplicity $\deg(\sigma)\ge1$, and $\sigma_n$ is a corresponding eigenvalue of $H$, then
    \begin{equation}
        \abr{\sigma-\sigma_n}=\cO_{\prec}\rbr{n^{-\frac{1}{2\deg(\sigma)+8}}}.
    \end{equation}
\end{proposition}

We also characterize the eigenvectors associated with simple outliers. The next proposition identifies the limiting rank-one projection onto the outlier eigenspace, tested against arbitrary deterministic directions.

\begin{proposition}[Outlier eigenvector alignment]\label{prop:sketch-statement-overlap}
    Suppose that $\sigma\in\Psi_0$ is simple, and let $\sigma_n$ be the corresponding eigenvalue of $H$ from Proposition~\ref{prop:sketch-statement-outlier}. Let $\nu_n\in\RR^{nL}$ be a unit eigenvector associated with $\sigma_n$. Let $w\in\RR^L$ be the normalized null vector of $\cQ(\sigma)$, so that $\cQ(\sigma)w=0$ and $\nbr{w}=1$. Then, uniformly for deterministic unit vectors $\xi_1,\xi_2\in\CC^{nL}$,
    \begin{equation}
        \left|
        (\xi_1^\top \nu_n) (\nu_n^\top \xi_2)
        -
        \frac{(\xi_1^\top \cG(\sigma) U w) ( w^\top U^\top\cG(\sigma) \xi_2)}{w^\top \cQ^\prime(\sigma) w}
        \right|
        =
        \cO_{\prec}(n^{-1/10}).
    \end{equation}
\end{proposition}

The proof of Proposition~\ref{prop:sketch-statement-overlap}, given in Subsection~\ref{subsec:outlier-projection}, is a contour version of the preceding determinant reduction. For the sketch, suppose that $\sigma\in\Psi_0$ is simple and that $\sigma_n$ is the corresponding eigenvalue of $H$. Choose a small circle $\partial B(\sigma,r)$ enclosing $\sigma_n$ and no other eigenvalue of $H$. The spectral projection onto the associated one-dimensional eigenspace is then recovered from the resolvent by Cauchy's formula:
\begin{align}
    &\quad \frac{1}{2\pi i}\oint_{\partial B(\sigma,r)} \xi_1^\top (H-zI_{nL})^{-1} \xi_2 \ud z \\
    &= \frac{1}{2\pi i}\oint_{\partial B(\sigma,r)} \sbr{ \frac{\xi_1^{\top}\nu_n\nu_n^{\top}\xi_2}{\sigma_n-z} + \sum_{(\Tilde{\sigma},\Tilde{\nu})} \frac{\xi_1^{\top}\Tilde{\nu}\Tilde{\nu}^{\top}\xi_2}{\Tilde{\sigma}-z} } \ud z = -\xi_1^{\top}\nu_n\nu_n^{\top}\xi_2,
\end{align}
where the summation in the intermediate expression is over all other eigenpairs $(\Tilde{\sigma},\Tilde{\nu})$ of $H$.
All other eigenvalues lie outside the contour, so their contributions integrate to zero. On the other hand, using~\eqref{eq:Resolvent of spiked matrix model} and replacing the random objects by their deterministic equivalents yields
\begin{align}
    \frac{1}{2\pi i}\oint_{\partial B(\sigma,r)} \xi_1^\top (H-zI_{nL})^{-1} \xi_2 \ud z &= \frac{1}{2\pi i}\oint_{\partial B(\sigma,r)} \xi_1^\top \sbr{\cG(z) - \cG(z) U \cQ(z)^{-1} U^{\top} \cG(z)} \xi_2 \ud z + \cO_{\prec}(n^{-1/10})\\
    &= -\frac{1}{2\pi i}\oint_{\partial B(\sigma,r)} \sbr{ \xi_1^\top \cG(z) U} \cQ(z)^{-1} \sbr{U^{\top} \cG(z) \xi_2} \ud z+\cO_{\prec}(n^{-1/10}).
\end{align}
Since $\sigma$ is a simple zero of $\det\cQ(z)$, the matrix $\cQ(z)^{-1}$ has a simple pole at $\sigma$. If $w$ spans the null space of $\cQ(\sigma)$, the residue calculation gives
\begin{equation}
    \frac{1}{2\pi i}\oint_{\partial B(\sigma,r)} \sbr{ \xi_1^\top \cG(z) U} \cQ(z)^{-1} \sbr{U^{\top} \cG(z) \xi_2} \ud z = \frac{\sbr{\xi_1^\top \cG(\sigma) U w} \sbr{ w^\top U^\top\cG(\sigma) \xi_2} }{ w^\top \cQ^\prime(\sigma) w }.
\end{equation}
This explains why the null vector $w$ of the $L$-dimensional matrix $\cQ(\sigma)$ controls the direction of the high-dimensional eigenvector $\nu_n$.

For signal recovery, the scalar overlap in Proposition~\ref{prop:sketch-statement-overlap} is not the most natural final object. As explained in Section~\ref{sec:spectral_introduction}, the signal-scale estimator is the rescaled matricized eigenvector $\hat X=\sqrt n\,\matop(\nu_n)$ from~\eqref{eq:matrix-estimator-definition}, because the target signal is the matrix $X\in\RR^{n\times L}$ in~\eqref{eq:observation model}. Equivalently, at the unit-eigenvector scale used in the spectral proof, the overlap $\matop(\nu_n)^\top X/\sqrt n=(1/n)\hat X^\top X$ gives all view-by-signal correlations at once: its $(a,b)$ entry measures the alignment between the $a$th block of the eigenvector and the $b$th latent signal. The next proposition records these matrix-valued overlaps; this is the form needed for the recovery formulas and is one of the places where the multi-view structure of the model appears explicitly.

\begin{proposition}[Matrix overlaps]\label{prop:sketch-matrix-overlaps}
    Under the assumptions of Proposition~\ref{prop:sketch-statement-overlap}, after choosing the global sign of $\nu_n$ consistently,
    \begin{equation}
        \left\|
        \frac{1}{\sqrt n}\matop(\nu_n)^\top X
        -
        \frac{1}{\sqrt{w^\top \cQ^\prime(\sigma)w}}
        V^{-1}M(\sigma)\Diag(w\odot\sqrt{\lambda})B
        \right\|
        =
        \cO_{\prec}(n^{-1/10}).
    \end{equation}
    Moreover, the block self-overlap matrix $\matop(\nu_n)^\top\matop(\nu_n)=(1/n)\hat X^\top\hat X$ equals an explicit $L\times L$ matrix determined by $M(\sigma)$, $B$, $\lambda$, and $w$, up to an error $\cO_{\prec}(n^{-1/20})$; the formula is given in Proposition~\ref{prop:detailed-matrix-overlaps} in Subsection~\ref{subsec:matrix-overlaps}.
\end{proposition}

The proof of Proposition~\ref{prop:sketch-matrix-overlaps} is completed in Subsection~\ref{subsec:matrix-overlaps}. The overlap with $X$ follows by testing the eigenvector formula against the spike directions. The block self-overlap $\matop(\nu_n)^\top\matop(\nu_n)$ is more subtle. Its entries compare different view blocks of the same eigenvector, and after the spectral projection is written as a contour integral, these entries lead to expressions with two resolvents and a deterministic block matrix inserted between them. The deterministic equivalent for a single resolvent is not enough to evaluate such quantities. This is exactly why we introduced the perturbation $\Gamma$ in the matrix Dyson equation~\eqref{eq: self-consistent eq matrix}. By differentiating the deterministic equivalent with respect to $\Gamma$, we extract the needed two-resolvent correlations and obtain the self-overlap formula.

\subsection{The phase transition threshold}
The preceding subsection gives a general asymptotic description of outliers: outside the bulk support, their possible locations are the real roots of
\[
    \det\cQ(\sigma)=0,\qquad
    \cQ(\sigma)=I_L-\sqrt{\Lambda}\sbr{B\odot M(\sigma)}\sqrt{\Lambda}.
\]
It also explains how the corresponding eigenvector overlaps are recovered from the null space of $\cQ(\sigma)$. To prove Theorems~\ref{thm: main uninformative} and~\ref{thm: main informative}, we must therefore identify when this deterministic equation creates the particular outlier responsible for weak recovery.

The special feature of the present model is that this root-creation problem has an explicit phase-transition threshold: it occurs exactly when the signal-to-noise ratio $\SNR(\lambda,B)$ defined in~\eqref{eq:summary SNR def} crosses one. The deterministic reason is encoded in the solution of the matrix Dyson equation at the special point $z=1$. Below the threshold, the determinant equation has no root at $1$; above the threshold, $1$ becomes the distinguished outlier location and the associated eigenvector has nontrivial overlap with the signal. This subsection explains the deterministic mechanism behind this transition and points to the later sections where the calculations are carried out.

The main input is Proposition~\ref{prop:special solution at 1}, proved in Section~\ref{sec:optimizing RS free energy}. It states, in particular, that when $\SNR(\lambda,B)<1$, the valid continuation of the matrix-valued Stieltjes transform satisfies
\[
    M(1)=B,
\]
whereas when $\SNR(\lambda,B)>1$, the valid solution satisfies
\[
    0\prec M(1)\prec B.
\]
Thus the phase transition is already visible at the level of the deterministic matrix Dyson equation. The point $M=B$ is the trivial solution at $z=1$; it is stable below the threshold and is replaced, above the threshold, by a smaller positive definite solution.

We next translate this deterministic transition into the spectral transition for $H$. If $\SNR(\lambda,B)<1$, then $M(1)=B$ gives
\[
    \cQ(1)=I_L-\sqrt{\Lambda}\sbr{B\odot B}\sqrt{\Lambda}\succ0,
\]
so the determinant equation has no root at $1$. In the informative regime, Proposition~\ref{prop:special solution at 1} gives $M(1)\prec B$, and Section~\ref{sec:emergence of outlier 1} uses the self-consistent equation at $z=1$ to construct an explicit nonzero vector in the kernel of $\cQ(1)$. Hence $\det\cQ(1)=0$, and Proposition~\ref{prop:sketch-statement-outlier} turns this deterministic root into an outlier eigenvalue of $H$ near $1$.

We regard the explicit phase-transition formula~\eqref{eq:summary SNR def}, with threshold $\SNR(\lambda,B)=1$, as one of the main contributions of the paper. It is not a direct consequence of general matrix Dyson equation theory. That theory provides the deterministic description of the correlated Gaussian noise, but the phase transition also requires locating the real roots of the finite-dimensional determinant equation $\det\cQ(\sigma)=0$ and explaining why the relevant threshold reduces to the explicit formula $\SNR(\lambda,B)=1$. Our model lies between two familiar regimes: it is more structured than a general correlated Gaussian ensemble, because the covariance structure comes from the multi-view AMP linearization, but it is still substantially richer than the one-dimensional BBP setting, where the relevant self-consistent equation can be solved in closed form.

The key additional idea is to connect the matrix Dyson equation at the special point $z=1$ with the TAP variational structure of the underlying inference problem. Section~\ref{sec:optimizing RS free energy} makes this connection through the functional~\eqref{eq:RS energy}, a finite-dimensional counterpart of the replica-symmetric free energy for the Gaussian version of the observation model~\eqref{eq:observation model}; see, for example, the general discussion in~\cite{montanari2024friendly}. After a change of variables, solving the MDE at $z=1$ becomes a problem about local minimizers of this explicit objective. The trivial point corresponds to $M=B$, and the Hessian at this point changes sign exactly when $\SNR(\lambda,B)$ crosses one.

This variational viewpoint explains why the same stability criterion that governs the uninformative TAP solution also governs the creation of the spectral outlier at $1$. More broadly, it suggests a potentially deeper connection between matrix Dyson equations, linearized AMP operators, and free-energy landscapes in other high-dimensional inference problems. Related uses of linearized AMP and spectral methods appear in~\cite{zhang2024matrix,zhang2024spectral,mergny2024spectral}; our analysis indicates that the associated deterministic MDEs may sometimes be understood more explicitly by uncovering the variational structure inherited from the underlying inference model.

The subsequent technical sections separate the deterministic and spectral parts of this argument. Section~\ref{sec:optimizing RS free energy} proves the deterministic input, namely Proposition~\ref{prop:special solution at 1}, by analyzing the variational functional~\eqref{eq:RS energy} and showing how the MDE solution at $z=1$ changes when $\SNR(\lambda,B)$ crosses one. Section~\ref{sec:emergence of outlier 1} then converts this deterministic information into the spectral statements of Theorems~\ref{thm: main uninformative} and~\ref{thm: main informative}: it counts the outliers to the right of the bulk edge, identifies the outlier at $1$ when $\SNR(\lambda,B)>1$, and simplifies the eigenvector-overlap formulas from Proposition~\ref{prop:detailed-matrix-overlaps}.

\section{The matrix Dyson equation: existence and stability}\label{sec:formal present self-cosnistent eq}

The purpose of this section is to introduce and analyze the deterministic self-consistent equation that controls the noise component $H_1$ in the decomposition~\eqref{eq:H_H0H1}. This equation constructs the matrix-valued measure $\xi$ and the limiting bulk measure $\mu$ used in Theorem~\ref{thm:main ESD}, and it formalizes the Stieltjes-transform description in~\eqref{eq:informal M stieltjes}. It is also the deterministic object that appears in the resolvent estimates of Section~\ref{sec:results on noise mat}.

\begin{definition}[Matrix Dyson equation and valid solution]
    For a spectral parameter $z\in\CC$ and a deterministic perturbation $\Gamma\in\SA_L(\RR)$, we consider the matrix Dyson equation
    \begin{equation}\label{eq: self-consistent eq matrix}
        M^{-1} = zB^{-1} - V^{-1} \Gamma V^{-1} + \Diag\rbr{\lambda} - \Diag\cbr{ \lambda \odot \diag(M) }, \quad M \in \SA_L(\CC). \tag{\textsf{mat}}
    \end{equation}
        When $z\in\CC^+$, we call $M(z;\Gamma)$ a valid solution if it satisfies \eqref{eq: self-consistent eq matrix} and $\Im M(z;\Gamma)\prec 0$.
\end{definition}
When $\Gamma=0$, the valid solution will be the matrix-valued Stieltjes
transform $M(z)$ introduced in~\eqref{eq:informal M stieltjes}. The
perturbation $\Gamma$ is not needed for the bulk law, but it will be useful
later when studying overlaps between spike eigenvectors and the target signals,
where one differentiates resolvents with respect to finite-dimensional
perturbations.

In this section, we establish two types of facts. First, existence, uniqueness,
and the Stieltjes representation of the solution to
\eqref{eq: self-consistent eq matrix} are standard consequences of the matrix
Dyson equation framework
\cite{ajanki2019stability,erdos2019matrix,alt2019location,alt2020correlated,alt2020dyson}.
Second, we record the quantitative stability estimate needed later in the
resolvent analysis. The proof uses the explicit finite-dimensional form of the
variance operator in our model; see
Remark~\ref{remark:variance-operator-structure}.


\subsection{Existence, uniqueness, and the limiting measure}

The self-consistent equation~\eqref{eq: self-consistent eq matrix} belongs to the general class of matrix Dyson equations studied in prior work~\cite{ajanki2019stability,erdos2019matrix,alt2019location,alt2020correlated,alt2020dyson}, after the change of variables given in Appendix~\ref{sec:proof of stability}. We use the general existence and representation theory for these equations to obtain the valid solution in the upper half-plane and the corresponding matrix-valued Stieltjes representation.

\begin{proposition}\label{prop:feasibility}
    \begin{itemize}
        \item[(i)] For any fixed $\Gamma\in\SA_L(\RR)$ and any $z\in\CC^{+}$, \eqref{eq: self-consistent eq matrix} admits a unique valid solution $M(z;\Gamma)\in\SA_L(\CC)$.

        \item[(ii)] For any fixed $\Gamma\in\SA_L(\RR)$, there exists a compactly-supported $\SA_L^+(\RR)$-valued measure $\xi_{\Gamma}$ such that $\xi_{\Gamma}(\RR)=B$, and for any $z\in\CC^+$,
        \begin{equation}\label{eq:Stieltjes representation}
            M(z;\Gamma) = - \int_{\RR} \frac{\xi_{\Gamma}(\ud \tau)}{\tau-z} \in \SA_L(\CC).
        \end{equation}
        
    \end{itemize}
\end{proposition}
Proposition~\ref{prop:feasibility} is a direct consequence of existing results on matrix Dyson equations, after reducing~\eqref{eq: self-consistent eq matrix} to the standard form. We provide the proof in Appendix~\ref{sec:proof of stability} for completeness.

\begin{definition}\label{def:bulk dist limit}
    Let $\xi_{\Gamma}$ be the measure introduced in Proposition~\ref{prop:feasibility}. We call $\xi_{\Gamma}$ the generating measure with perturbation $\Gamma$, and define
    \begin{equation}\label{eq:def mu gamma}
        \mu_{\Gamma} = \frac{1}{L}\tr\rbr{B^{-1}\xi_{\Gamma}}.
    \end{equation}
    In the terminology of \cite{alt2019location,ajanki2019stability}, $\mu_{\Gamma}$ is the \emph{self-consistent density of states}.
\end{definition}

\begin{remark}
When $\Gamma=0$, we suppress the perturbation from the notation:
\begin{equation}
    M(z):=M(z;0),\qquad \xi:=\xi_0,\qquad \mu:=\mu_0.
\end{equation}
The same convention applies to any $\Gamma$-dependent object introduced below; for example, $\cU(\epsilon)$ means $\cU_0(\epsilon)$.
With this convention, $\xi$ and $\mu$ agree with the limiting measures
introduced informally in~\eqref{eq:informal mu definition}.
\end{remark}

The measure $\mu_{\Gamma}$ is a compactly-supported probability measure on $\RR$, since
\begin{equation*}
    \mu_{\Gamma}(\RR)=\frac{1}{L}\tr\sbr{B^{-1}\xi_{\Gamma}(\RR)}=1.
\end{equation*}
Since $B^{-1}\succ0$ and $\xi_{\Gamma}$ is positive semidefinite valued, the scalar measure $\mu_{\Gamma}$ and the matrix-valued measure $\xi_{\Gamma}$ have the same support. The Stieltjes representation~\eqref{eq:Stieltjes representation} implies that $M(z;\Gamma)$ extends analytically from $\CC^+$ to $\CC\backslash\supp(\mu_{\Gamma})$. The following proposition gives the corresponding criterion for locating $\supp(\mu_{\Gamma})$.

\begin{proposition}\label{prop:analytical extension}
    Fix $\Gamma\in\SA_L(\RR)$. An interval $(a,b)\subseteq\RR$ is outside $\supp(\mu_{\Gamma})$ if and only if $M(z;\Gamma)$ admits a continuous extension onto $(a,b)$ with $M(\sigma;\Gamma)\in\SA_L(\RR)$ for any $\sigma\in(a,b)$.
\end{proposition}
\begin{proof}
    If $(a,b)\cap\supp(\mu_{\Gamma})=\emptyset$, the continuous extension follows trivially from~\eqref{eq:Stieltjes representation}. Conversely, suppose that $M(z;\Gamma)$ admits the prescribed extension.
    Using the inverse formula of Stieltjes transforms \cite[Lemma 2.1.4]{erdos2019matrix}, for any $(a_1,b_1)\subset(a,b)$, it holds that
    \begin{equation}
        \frac{1}{2}\sbr{\xi_{\Gamma}(\cbr{a_1})+\xi_{\Gamma}(\cbr{b_1})} + \xi_{\Gamma}((a_1,b_1)) = -\lim_{\eta\to0^+}\frac{1}{\pi}\int_{a_1}^{b_1} \Im M(E+i\eta;\Gamma)\ud E = 0.
    \end{equation}
    Since $\xi_{\Gamma}$ is a positive semidefinite matrix-valued measure, all terms on the left-hand side are positive semidefinite matrices. Hence it further implies that each term vanishes individually, i.e.,
    \begin{equation}
        \xi_{\Gamma}(\{a_1\}) = 0, \qquad \xi_{\Gamma}(\{b_1\}) = 0, \qquad \xi_{\Gamma}((a_1,b_1)) = 0.
    \end{equation}
    Therefore, $\xi_{\Gamma}$ vanishes on every subinterval of $(a,b)$. Since $\supp(\xi_{\Gamma})=\supp(\mu_{\Gamma})$, this shows that $(a,b)\cap \supp(\mu_{\Gamma})=\emptyset$.
\end{proof}

\subsection{Stability of the matrix Dyson equation}

We next prove the quantitative stability estimate that will be used in the
resolvent analysis of Section~\ref{sec:results on noise mat}. Roughly
speaking, the estimate says that if a matrix nearly satisfies the quadratic
form of the matrix Dyson equation~\eqref{eq: self-consistent eq matrix}, then
it must be close to the valid solution $M(z;\Gamma)$, uniformly for spectral
parameters away from $\supp(\mu_{\Gamma})$. The next remark records the
particular structure of the variance operator that will be used in the proof.

\begin{remark}[Structure of the variance operator]\label{remark:variance-operator-structure}
    General matrix Dyson equations are often written in the form
    \begin{equation}
        -M^{-1} = zI_{N} - A + \cS[M],
    \end{equation}
    where $A$ is deterministic and $\cS$ encodes the covariance structure of
    the random matrix. For the equation~\eqref{eq: self-consistent eq matrix},
    the corresponding finite-dimensional variance operator has the simple form
    \begin{equation}
        \cS[R]=\Diag\cbr{\lambda\odot\diag(R)}.
    \end{equation}
    It acts only on the diagonal of $R$ and returns a diagonal matrix. The
    proof below uses this explicit form directly. This is convenient for our
    purposes because the later outlier and eigenvector calculations also take
    place in the same fixed-dimensional space.
\end{remark}

\paragraph{The stability domain.}
For any $\epsilon>0$, we define the $\epsilon$-stability domain by
\begin{equation}\label{eq:domain for unfirom convergence}
    \cU_{\Gamma}(\epsilon) := \cbr{z\in\CC^+: \dist\rbr{z,\supp\mu_{\Gamma}} \ge \epsilon, \abr{z} \le \epsilon^{-1}},
    \qquad
    \cU(\epsilon):=\cU_0(\epsilon).
\end{equation}
This is the region where the limiting spectral parameter stays a fixed distance away from the bulk spectrum and where all stability estimates below are uniform.
    
\begin{proposition}[Stability of \eqref{eq: self-consistent eq matrix} in its quadratic form]\label{prop:stability}
    There exist constants $c,C>0$, depending only on $(B,\lambda,\Gamma)$ and $\epsilon$, such that the following stability result holds uniformly over $z\in \cU_{\Gamma}(\epsilon)$. If $(\widetilde{M},\Delta)$ satisfies
    \begin{equation}\label{eq:perturbed quadratic eq}
       I_L = \rbr{ zB^{-1} - V^{-1} \Gamma V^{-1} + \Lambda - \Diag\cbr{ \lambda \odot \diag(\widetilde{M}) } } \widetilde{M} + \Delta,\quad \Im \widetilde{M} \prec 0,\quad \nbr{\Delta} \le c,
    \end{equation}
    then $\nbr{\widetilde{M} - M(z;\Gamma)}\le C\nbr{\Delta}$, where $M(z;\Gamma)$ is the valid solution from Proposition~\ref{prop:feasibility}.
\end{proposition}

Note that our proof of Proposition~\ref{prop:stability} is adapted from the arguments in \cite{alt2019location}. Since the general ensembles in \cite{alt2019location} allow inhomogeneous variance profiles in each component, their self-consistent equation is expressed through a vector Dyson equation with matrix-valued variables. Our object is conciser, so we append a self-contained proof here to illustrate the key ideas in establishing the stability result.

The key point in Proposition~\ref{prop:stability} is the invertibility of the linearization of the matrix Dyson equation. Throughout the rest of this subsection, $\Gamma$ is fixed and we write $M=M(z;\Gamma)$.

\begin{definition}[Stability operator]\label{def:stability operator}
    Let $\cS,\cC_R:M_L(\CC)\to M_L(\CC)$ be given by
    \begin{alignat}{2}
        \cS[A] &= \Diag\cbr{ \lambda \odot \diag( A ) }, \qquad&
        \cC_R[A] &= RAR.
    \end{alignat}
    For the valid solution $M=M(z;\Gamma)$, define the stability operator
    \begin{equation}\label{eq:linear matrix operator}
        \cL_z[A] := A - M \Diag\cbr{ \lambda \odot \diag( A ) } M
        = \rbr{\mathrm{Id}-\cC_M\cS}[A].
    \end{equation}
\end{definition}

The operator $\cL_z$ appears by differentiating the quadratic form of the self-consistent equation. Indeed, if $\widetilde M(\Delta)$ is a local solution of~\eqref{eq:perturbed quadratic eq} with $\widetilde M(0)=M$, then for a direction $R\in M_L(\CC)$,
\begin{equation}\label{eq:directional derivative equation}
    \cL_z\sbr{ \nabla_{R}\widetilde{M}(0) } = - M R.
\end{equation}
Thus the desired stability follows once $\cL_z^{-1}$ is bounded uniformly on the domain~\eqref{eq:domain for unfirom convergence}. The next lemma provides this input.

\begin{lemma}[Invertibility of the stability operator]\label{lemma:bounds on cL}
    Uniformly over $z\in\CC^+$, the following bounds hold:
    \begin{itemize}
        \item[(i)] The solution and its inverse satisfy
        \begin{align}
            \nbr{M(z;\Gamma)} &\le \frac{C(B,\Gamma)}{\dist\rbr{z,\supp \mu_{\Gamma}}}, \label{eq:upper bound on M norm}\\
            \nbr{M(z;\Gamma)^{-1}} &\le C(B,\lambda,\Gamma)\rbr{1+\abr{z}+\nbr{M(z;\Gamma)}}. \label{eq:upper bound on M inverse norm}
        \end{align}

        \item[(ii)] The inverse of the stability operator satisfies
        \begin{equation}\label{eq:cL inverse bound}
            \nbr{\cL_z^{-1}} \le C(B,\lambda,\Gamma) \frac{\nbr{M(z;\Gamma)}\nbr{M(z;\Gamma)^{-1}}^{9}}{\dist\rbr{z,\supp \mu_{\Gamma}}^8}.
        \end{equation}
    \end{itemize}
\end{lemma}
In particular, the right-hand side of~\eqref{eq:cL inverse bound} is uniformly bounded on $\cU_{\Gamma}(\epsilon)$.

\begin{proof}
    The bound~\eqref{eq:upper bound on M norm} follows from the Stieltjes-transform representation of the valid solution, while~\eqref{eq:upper bound on M inverse norm} follows directly from~\eqref{eq: self-consistent eq matrix}. Indeed,
    \begin{align}
        \nbr{M^{-1}} &\le \nbr{zB^{-1}-V^{-1}\Gamma V^{-1}} + \nbr{\Diag\rbr{\lambda}} + \nbr{\Diag\cbr{ \lambda \odot \diag(M) }} \\
        &\le C(B,\lambda,\Gamma)\rbr{1+\abr{z}+\nbr{M}}.
    \end{align}
    The nontrivial part is the bound on $\cL_z^{-1}$. The argument follows the
    saturated-operator method of~\cite{ajanki2019stability,alt2019location},
    specialized to the finite-dimensional operator
    $\cS[R]=\Diag\cbr{\lambda\odot\diag(R)}$.

    \medskip
    \noindent\textit{Step 1. Auxiliary operators.}
    Define
    \begin{equation}
        \cK_{R} : M_L(\CC) \to M_L(\CC),  \quad \cK_{R}[A] = R^{\dagger} A R, \quad \forall R\in M_L(\CC).
    \end{equation}
    If $R$ is Hermitian, in particular if $R\in\SA_L(\RR)$, then $\cC_R=\cK_R$.
    We use the identity
    \begin{equation*}
        \Im\rbr{M^{-1}}=- (M^{\dagger})^{-1} [\Im M] M^{-1}.
    \end{equation*}
    Taking the imaginary part of~\eqref{eq: self-consistent eq matrix} gives
    \begin{equation}\label{eq:imaginary part of M}
        -\Im M = \Im z \cdot \cK_{M} [B^{-1}] - \cK_{M} \cS\sbr{\Im M}.
    \end{equation}
    Introduce
    \begin{equation}\label{eq:define auxiliiary matrices for stability}
        T := \cC_{\sqrt{-\Im M}}^{-1}[\Re M]+iI, \quad W := (T^{\dagger}T)^{1/4}, \quad U := T (T^{\dagger}T)^{-1/2}.
    \end{equation}
    Since $T$ is normal and $T^{\dagger}T=TT^{\dagger}$,
    \begin{equation}
        M = \rbr{W\sqrt{-\Im M}}^{\dagger} U^{\dagger} \rbr{W\sqrt{-\Im M}}.
    \end{equation}
    Consequently,
    \begin{equation}
        \cK_{M} = \cC_{\sqrt{-\Im M}} \cC_{W} \cK_{U^{\dagger}} \cC_{W} \cC_{\sqrt{-\Im M}}.
    \end{equation}
    Define the saturated operator $\cF:M_L(\CC)\to M_L(\CC)$ by
    \begin{equation}\label{eq:saturated operator}
        \cF := \cC_{W} \cC_{\sqrt{-\Im M}} \cS \cC_{\sqrt{-\Im M}} \cC_{W}.
    \end{equation}
    The operator $\cF$ preserves the cone of positive semidefinite matrices. Hence a Perron-Frobenius theorem for cone-preserving operators gives a normalized matrix $F$ with $\nbr{F}=1$, $F\succeq0$, and $\cF[F]=\nbr{\cF}F$.

    \medskip
    \noindent\textit{Step 2. The saturated equation.}
    Applying the inverse of $\cC_{\sqrt{-\Im M}}\cC_W\cK_{U^\dagger}$ to both sides of~\eqref{eq:imaginary part of M}, we obtain
    \begin{equation}
        W^{-2} = \Im z \cdot \cC_W\cC_{\sqrt{-\Im M}}\sbr{B^{-1}}+\cF\sbr{W^{-2}}.
    \end{equation}
    For the left-hand side we used the commutation of $W$ and $U$, which implies
    \begin{equation}
        \cK_{U^{\dagger}}^{-1}\cC_W^{-1}\cC_{\sqrt{-\Im M}}^{-1}\sbr{-\Im M}
        = U^\dagger W^{-2}U=W^{-2}.
    \end{equation}
    Taking the matrix inner product with $F$ yields
    \begin{equation}\label{eq:1-F_norm}
        1 - \nbr{\cF} =
        \frac{\Im z \inner{F}{\cC_W\cC_{\sqrt{-\Im M}}\sbr{B^{-1}}}}
        {\inner{F}{W^{-2}}}.
    \end{equation}

    \medskip
    \noindent\textit{Step 3. Bounding the saturation gap.}
    The Stieltjes representation and the fact that $\supp(\mu_{\Gamma})\subset\RR$ imply
    \begin{equation}\label{eq:upper bound on -Im M}
        -\Im M \preceq \frac{C(B,\Gamma) \Im z}{\dist\rbr{z,\supp \mu_{\Gamma}}^2} I_L.
    \end{equation}
    Since $\cK_M\cS$ is positivity preserving and $\Im M\prec0$, equation~\eqref{eq:imaginary part of M} also gives
    \begin{equation}\label{eq:lower bound on -Im M}
        -\Im M \succeq \Im z \cK_M[B^{-1}]
        \succeq \sigma_{\max}(B)^{-1}\Im z \cdot M^\dagger M
        \succeq \sigma_{\max}(B)^{-1}\Im z \nbr{M^{-1}}^{-2}I_L.
    \end{equation}
    From the definition of $W$ in~\eqref{eq:define auxiliiary matrices for stability},
    \begin{align}
        W^4
        &= \cC_{\sqrt{-\Im M}}^{-1}\rbr{\cC_{\Re M}+\cC_{-\Im M}}\sbr{(-\Im M)^{-1}} \notag\\
        &\succeq c(B,\Gamma)\rbr{\Im z}^{-1}\dist\rbr{z,\supp\mu_{\Gamma}}^2
        \cC_{\sqrt{-\Im M}}^{-1}\sbr{MM^\dagger+M^\dagger M} \notag\\
        &\succeq c(B,\Gamma)\nbr{M^{-1}}^{-2}\rbr{\Im z}^{-2}\dist\rbr{z,\supp\mu_{\Gamma}}^4 I_L, \label{eq:lower bound on W}
    \end{align}
    where we used $MM^\dagger+M^\dagger M\succeq2\nbr{M^{-1}}^{-2}I_L$ and~\eqref{eq:upper bound on -Im M}. Similarly,
    \begin{align}
        W^4
        &= \cC_{\sqrt{-\Im M}}^{-1}\rbr{\cC_{\Re M}+\cC_{-\Im M}}\sbr{(-\Im M)^{-1}} \notag\\
        &\preceq C(B)\rbr{\Im z}^{-1}\nbr{M^{-1}}^2
        \cC_{\sqrt{-\Im M}}^{-1}\sbr{MM^\dagger+M^\dagger M} \notag\\
        &\preceq C(B)\nbr{M}^2\rbr{\Im z}^{-2}\nbr{M^{-1}}^4 I_L. \label{eq:upper bound on W}
    \end{align}
    We now bound the numerator and denominator in~\eqref{eq:1-F_norm}. For the numerator,
    \begin{align}
        \inner{F}{\cC_W\cC_{\sqrt{-\Im M}}\sbr{B^{-1}}}
        &\ge c(B)\inner{F}{\cC_W\sbr{-\Im M}} \notag\\
        &\ge c(B)\Im z\nbr{M^{-1}}^{-2}\inner{F}{W^2} \notag\\
        &= c(B)\Im z\nbr{M^{-1}}^{-2}\tr\sbr{\sqrt{F}W^2\sqrt{F}} \notag\\
        &\ge c(B,\Gamma)\nbr{M^{-1}}^{-3}\dist\rbr{z,\supp\mu_{\Gamma}}^2\tr(F).
    \end{align}
    For the denominator,
    \begin{equation}
        \inner{F}{W^{-2}} \le C(B,\Gamma)\nbr{M^{-1}}\rbr{\Im z}\dist\rbr{z,\supp\mu_{\Gamma}}^{-2}\tr(F).
    \end{equation}
    Combining these estimates with~\eqref{eq:1-F_norm}, we obtain
    \begin{equation}\label{eq:saturation gap}
        1-\nbr{\cF}\ge c(B,\Gamma)\frac{\dist\rbr{z,\supp\mu_{\Gamma}}^4}{\nbr{M^{-1}}^4}.
    \end{equation}

    \medskip
    \noindent\textit{Step 4. Inverting the stability operator.}
    By the definitions~\eqref{eq:define auxiliiary matrices for stability} and~\eqref{eq:saturated operator},
    \begin{equation}
        \cL_z =
        \cC_{\sqrt{-\Im M}}\cC_W\cC_{U^\dagger}
        \rbr{\cC_U-\cF}
        \cC_W^{-1}\cC_{\sqrt{-\Im M}}^{-1}.
    \end{equation}
    Since $U$ is unitary, $\nbr{\cC_{U^\dagger}}=1$, and~\eqref{eq:saturation gap} implies
    \begin{equation}
        \nbr{\rbr{\cC_U-\cF}^{-1}}
        \le \rbr{1-\nbr{\cF}}^{-1}
        \le C(B,\Gamma)\frac{\nbr{M^{-1}}^4}{\dist\rbr{z,\supp\mu_{\Gamma}}^4}.
    \end{equation}
    Finally, using~\eqref{eq:upper bound on -Im M}, \eqref{eq:lower bound on -Im M}, \eqref{eq:lower bound on W}, and~\eqref{eq:upper bound on W},
    \begin{align}
        \nbr{\cC_{\sqrt{-\Im M}}} &\le C(B,\Gamma)\rbr{\Im z}\dist\rbr{z,\supp\mu_{\Gamma}}^{-2}, \\
        \nbr{\cC_{\sqrt{-\Im M}}^{-1}} &\le C(B)\rbr{\Im z}^{-1}\nbr{M^{-1}}^2, \\
        \nbr{\cC_W} &\le C(B)\nbr{M}\rbr{\Im z}^{-1}\nbr{M^{-1}}^2, \\
        \nbr{\cC_W^{-1}} &\le C(B,\Gamma)\nbr{M^{-1}}\rbr{\Im z}\dist\rbr{z,\supp\mu_{\Gamma}}^{-2}.
    \end{align}
    Multiplying the displayed bounds gives~\eqref{eq:cL inverse bound}.
\end{proof}

\begin{lemma}[Quantitative implicit function theorem, Lemma 4.10 of \cite{ajanki2019stability}]\label{lemma:implicit function}
    Let $T:\CC^{A}\times\CC^{D}\to\CC^{A}$ be continuously differentiable with $T(0,0)=0$ and invertible partial derivative $\nabla^{(1)} T(0,0) \in M_A(\CC)$.
    Assume $\CC^{A}$ and $\CC^{D}$ are equipped with norms $\|\cdot\|$ and let the induced operator norms be denoted by the same symbol. Suppose there exist constants $\delta>0$ and $C_1,C_2<\infty$ such that:
    \begin{enumerate}
        \item $\big\| \rbr{\nabla^{(1)}T(0,0)}^{-1} \| \le C_1$;
        \item $\big\| \mathrm{Id}_{\CC^A} - \rbr{\nabla^{(1)}T(0,0)}^{-1} \nabla^{(1)}T(a,d)\big\|\le \frac{1}{2}$ for every $(a,d) \in B_\delta^{A}\times B_\delta^{D}$;
        \item $\| \nabla^{(2)}T(a,d) \| \le C_2$ for every $(a,d) \in B_\delta^{A}\times B_\delta^{D}$.
    \end{enumerate}
    Then there exists $\epsilon>0$, depending only on $(\delta,C_1,C_2)$, and a unique continuously differentiable function $f:B_\epsilon^{D}\to B_\delta^{A}$ such that
    \begin{align}
        T(f(d),d) &= 0 \qquad \text{for all } d\in B_\epsilon^{D}, \\
        Df(d) &= -(\nabla^{(1)}T(f(d),d))^{-1}\nabla^{(2)}T(f(d),d).
    \end{align}
    Therefore, $Df(d)$ is uniformly bounded by a constant depending only on $(C_1,C_2)$ for $d\in B_{\epsilon}$. Moreover, if $T$ is analytic, then $f$ is analytic.
\end{lemma}

\begin{proof}[Proof of Proposition~\ref{prop:stability}]
    We apply Lemma~\ref{lemma:implicit function} to the map $\cJ:M_L(\CC)\times M_L(\CC)\to M_L(\CC)$ defined by
    \begin{equation}
        \cJ\sbr{\widetilde{M},\Delta}
        =
        -I_L+
        \rbr{zB^{-1}-V^{-1}\Gamma V^{-1}+\Lambda-\Diag\cbr{\lambda\odot\diag(\widetilde M)}}\widetilde M
        +\Delta.
    \end{equation}
    The perturbed equation~\eqref{eq:perturbed quadratic eq} is equivalent to $\cJ[\widetilde M,\Delta]=0$, and the equation with $\Delta=0$ is $\cJ[M(z;\Gamma),0]=0$. For a direction $R\in M_L(\CC)$,
    \begin{align}
        \nabla^{(\widetilde M)}_R\cJ\sbr{\widetilde M,\Delta}
        &=
        \rbr{zB^{-1}-V^{-1}\Gamma V^{-1}+\Lambda-\Diag\cbr{\lambda\odot\diag(\widetilde M)}}R \notag\\
        &\qquad
        -\Diag\cbr{\lambda\odot\diag(R)}\widetilde M.
    \end{align}
    Evaluating at $(M(z;\Gamma),0)$ and using~\eqref{eq: self-consistent eq matrix},
    \begin{equation}
        \nabla^{(\widetilde M)}_R\cJ\sbr{M(z;\Gamma),0}
        =
        M(z;\Gamma)^{-1}\cL_z[R].
    \end{equation}
    Lemma~\ref{lemma:bounds on cL}, together with the definition of $\cU_{\Gamma}(\epsilon)$ in~\eqref{eq:domain for unfirom convergence}, gives a uniform bound on the inverse of this derivative for $z\in\cU_{\Gamma}(\epsilon)$. The derivative of $\cJ$ is continuous in $\widetilde M$, so the second hypothesis of Lemma~\ref{lemma:implicit function} holds on a sufficiently small uniform neighborhood. Finally, $\nabla^{(\Delta)}_R\cJ[\widetilde M,\Delta]=R$, so the third hypothesis is immediate.

    Lemma~\ref{lemma:implicit function} yields a local solution map $\Delta\mapsto\widetilde M(\Delta)$ with uniformly bounded derivative. Therefore $\nbr{\widetilde M-M(z;\Gamma)}\le C\nbr{\Delta}$ whenever $\nbr{\Delta}\le c$.
\end{proof}

\section{Proof of the bulk law}\label{sec:results on noise mat}

The goal of this section is to prove the bulk law in Theorem~\ref{thm:main ESD}. Recall from~\eqref{eq:H_H0H1} that the linearized AMP matrix decomposes as $H=H_0+H_1$, where $H_0$ is the signal component and has rank at most $L$, while $H_1$ is the correlated Gaussian noise component. Since adding a matrix of rank at most $L$ does not change the limiting empirical spectral distribution, the bulk law is determined by the spectrum of $H_1$. We therefore prove deterministic equivalents for the resolvent of $H_1$ and identify their limit with the solution of the matrix Dyson equation introduced in Section~\ref{sec:formal present self-cosnistent eq}. For later use in the eigenvector-overlap analysis, we carry out the argument for the slightly perturbed matrix $H_1+\Gamma\otimes I_n$, where $\Gamma\in\SA_L(\RR)$ is deterministic. The bulk law itself corresponds to the special case $\Gamma=0$; throughout the paper, notation without $\Gamma$ refers to this unperturbed case.

\subsection{Setup and main inputs}
For $z \in \CC^+$ and $\Gamma\in\SA_L(\RR)$, we define a generalized resolvent for the noise matrix $H_1$ given in~\eqref{eq:form of H1},
\begin{equation}\label{eq:resolvent def}
    G(z;\Gamma) = \rbr{ H_1 + \Gamma \otimes I_n - z I_{nL} }^{-1} \in \CC^{nL \times nL}.
\end{equation}
The rest of this section rigorously establishes the connection between $G(z;\Gamma)$ and $M(z;\Gamma)$, the unique solution to \eqref{eq: self-consistent eq matrix} defined in Section~\ref{sec:formal present self-cosnistent eq}.
We also introduce the partial trace operator $\Tsf: M_{nL}(\CC) \to M_{L}(\CC)$, defined by
\begin{equation}\label{eq:partial trace}
    A=\cbr{ A_{ij}^{lk} : i,j\in[n], l,k\in[L] } \mapsto \Tsf(A)=\cbr{ \sum_{i=1}^n A_{ii}^{lk} : l,k\in[L]}.
\end{equation}
Here $A$ is viewed as an $L \times L$ block matrix whose blocks are $n \times n$. We shall repeatedly use
\begin{align}
    \Tsf(A_0\otimes I_n) &= nA_0, &\quad A_0\in M_L(\CC), \\
    \Tsf\sbr{(A_1\otimes I_n)A_2(A_3 \otimes I_n)} &= A_1 \Tsf\sbr{A_2}A_3, &\quad A_1,A_3\in M_L(\CC), A_2\in M_{nL}(\CC).
\end{align}
We will need a concentration estimate for the normalized partial trace of the resolvent. The standard proof is given in Appendix~\ref{app:resolvent-concentration}; it yields that for every deviation level $\delta>0$,
\begin{equation}\label{eq:concentration of the partial trace}
    \PP\rbr{ \nbr{ \frac{1}{n} \Tsf\sbr{ G(z;\Gamma) } - \EE \cbr{ \frac{1}{n} \Tsf\sbr{ G(z;\Gamma) } } }_F \ge \delta } \le 2 \exp\rbr{- C \rbr{\Im z}^4 n^2 \delta^2}.
\end{equation}
Since $L$ stays fixed throughout the paper, the Frobenius norm $\|\cdot\|_F$ in~\eqref{eq:concentration of the partial trace} can be replaced by any matrix norm on $\SA_L(\CC)$, up to changing constants.
\begin{definition}[Deterministic equivalent resolvent]\label{def:deterministic-equivalent-resolvent}
    Fix $\Gamma\in\SA_L(\RR)$ and let $M(z;\Gamma)$ be the valid solution of the matrix Dyson equation~\eqref{eq: self-consistent eq matrix}. For $z\in\CC^+$, define
    \begin{equation}\label{eq:G_equivalent}
        \cG(z;\Gamma) := [- V^{-1} M(z;\Gamma) V^{-1}] \otimes I_n \in \CC^{ nL \times nL}.
    \end{equation}
    When $\Gamma=0$, we write $\cG(z):=\cG(z;0)$.
\end{definition}

The following two propositions connect the random resolvent $G(z;\Gamma)$ to the deterministic matrix $\cG(z;\Gamma)$. The first identifies the normalized partial trace, and the second gives convergence in the sense of deterministic equivalents~\cite{hachem2007deterministic}.

\begin{proposition}[Partial-trace deterministic equivalent]\label{prop:partial_trace_equivalent}\label{lemma:convergence of expected partial trace}
    Recall the domain $\cU_{\Gamma}(\epsilon)$ in~\eqref{eq:domain for unfirom convergence}. Uniformly over $z\in \cU_{\Gamma}(\epsilon)$, we have
    \begin{equation}
        \frac{1}{n} \Tsf\sbr{ G(z;\Gamma) }  = - V^{-1} M(z;\Gamma) V^{-1} + \cO_{\prec}\rbr{ \frac{1}{n(\Im z)^3} },
    \end{equation}
    where $M(z;\Gamma)$ is the unique valid solution from Proposition~\ref{prop:feasibility}.
\end{proposition}

\begin{proposition}\label{prop:equivalent}
    For any fixed $\epsilon>0$, uniformly over $z\in \cU_{\Gamma}(\epsilon)$ and deterministic matrices $A\in\CC^{nL\times nL}$, the following holds
    \begin{align}
        \tr\sbr{ A G(z;\Gamma) } &= \tr\sbr{ A \cG(z;\Gamma) } + \cO_{\prec}\rbr{ \frac{1}{\sqrt{n}\rbr{\Im z}^4} \nbr{A}_F }.
    \end{align}
\end{proposition}

We prove these two propositions in the next two subsections. We then derive a no-outside-spectrum estimate for the noise matrix and use it, together with the rank bound on $H_0$, to complete the proof of Theorem~\ref{thm:main ESD}.

\subsection{The partial trace equation}\label{subsec:partial_trace_equation}
We now prove Proposition~\ref{prop:partial_trace_equivalent}. The averaged object that naturally appears in the computation is
\[
    \cT_{G,n}:=\frac{1}{n}\Tsf\sbr{(V\otimes I_n)G(z;\Gamma)(V\otimes I_n)}.
\]
The main step is to derive a finite-dimensional relation for $\cT_{G,n}$ by applying Stein's lemma and then taking the partial trace. After concentration, this relation becomes an approximate version of the matrix Dyson equation in the variable $\cT=-M(z;\Gamma)$. The stability estimate from Proposition~\ref{prop:stability} then identifies $\EE\cT_{G,n}$ with $-M(z;\Gamma)$; a final use of concentration gives the random partial-trace estimate in Proposition~\ref{prop:partial_trace_equivalent}.

\begin{proof}[Proof of Proposition~\ref{prop:partial_trace_equivalent}]
    In the proof, for notational convenience, write $G=G(z;\Gamma)$ and $\tilde{V}_n = V \otimes I_n$.
    For convenience, write $\tilde{H}_1$ for the mean-zero random part of $H_1$,
    \begin{equation}
        \tilde{H}_1 = H_1 - \EE H_1 = \tilde{V}_n \, \Diag\rbr{ \sqrt{\frac{\lambda_1}{n}}W^{(1)}, \ldots, \sqrt{\frac{\lambda_L}{n}}W^{(L)} } \, \tilde{V}_n.
    \end{equation}
    In what follows, we will consistently use $e_i^l$ to denote the one-hot length-$nL$ vector with only the $i$-th entry in the $l$-th block being $1$. Then decompose $\tilde{H}_1$ into a sum of independent components:
    \begin{align}
        \tilde{H}_1 &=  \sum_{i=1}^n \sum_{l=1}^L \sqrt{\frac{\lambda_l}{n}} W^{(l)}_{ii} \tilde{V}_n \sbr{ e_i^l \rbr{e_i^{l}}^\top } \tilde{V}_n \\
        &\qquad + \sum_{1 \le i < j \le n} \sum_{l=1}^L \sqrt{\frac{\lambda_l}{n}} W^{(l)}_{ij} \tilde{V}_n \sbr{e_i^l \rbr{e_j^{l}}^\top + e_j^l \rbr{e_i^{l}}^\top} \tilde{V}_n, \label{eq:H1_decomposition}
    \end{align}
    Using the decomposition in \eqref{eq:H1_decomposition}, we can deduce that
    \begin{align}
        \EE\sbr{ G \tilde{H}_1 } &= \EE\Bigg\{ \sum_{i=1}^n \sum_{l=1}^L \sqrt{\frac{\lambda_l}{n}} \EE_{ii,l}\sbr{W^{(l)}_{ii} G} \tilde{V}_n \sbr{ e_i^l \rbr{e_i^{l}}^\top } \tilde{V}_n \\
        &\qquad\qquad + \sum_{1 \le i < j \le n} \sum_{l=1}^L \sqrt{\frac{\lambda_l}{n}} \EE_{ij,l}\sbr{W^{(l)}_{ij} G} \tilde{V}_n \sbr{e_i^l \rbr{e_j^{l}}^\top + e_j^l \rbr{e_i^{l}}^\top} \tilde{V}_n \Bigg\}, \label{eq:E_GH1_0}
    \end{align}
    where $\EE_{ij,l}\rbr{\cdot}$ denote the \emph{partial} expectation with respect to the scalar Gaussian $W^{(l)}_{ij}$, while keeping all the other random variables fixed. By Stein's lemma and standard resolvent identities,
    \begin{align}
        \EE_{ii,l}\sbr{W^{(l)}_{ii} G} &= -\sqrt{\frac{\lambda_l}{n}}\EE_{ii,l}\sbr{G \tilde{V}_n \sbr{ 2e_i^l \rbr{e_i^{l}}^\top } \tilde{V}_n G}, \\
        \EE_{ij,l}\sbr{W^{(l)}_{ij} G} &= -\sqrt{\frac{\lambda_l}{n}}\EE_{ij,l}\sbr{ G \tilde{V}_n \sbr{ e_i^l \rbr{e_j^{l}}^\top + e_j^l \rbr{e_i^{l}}^\top } \tilde{V}_n G}.
    \end{align}
    Substituting these expressions into \eqref{eq:E_GH1_0} yields
    \begin{align}
        &\quad \EE\sbr{ G \tilde{H}_1 } \\
        &= -\EE\Bigg\{ \sum_{i=1}^n \sum_{l=1}^L \frac{2\lambda_l}{n} \EE_{ii,l}\sbr{ G \tilde{V}_n\sbr{e_i^l \rbr{e_i^{l}}^\top}\tilde{V}_n G } \tilde{V}_n \sbr{ e_i^l \rbr{e_i^{l}}^\top }\tilde{V}_n \\
        &\qquad + \sum_{i < j} \sum_{l=1}^L \frac{\lambda_l}{n} \EE_{ij,l}\sbr{ G \tilde{V}_n\sbr{e_i^l \rbr{e_j^{l}}^\top + e_j^l \rbr{e_i^{l}}^\top}\tilde{V}_n G } \tilde{V}_n\sbr{e_i^l \rbr{e_j^{l}}^\top + e_j^l \rbr{e_i^{l}}^\top}\tilde{V}_n \Bigg\} \\
        &= -\sum_{l=1}^L \frac{\lambda_l}{n} \EE\Bigg\{ G\tilde{V}_n \Bigg[ 2\sum_{i=1}^n e_i^l \rbr{e_i^{l}}^\top \tilde{V}_n G \tilde{V}_n e_i^l \rbr{e_i^{l}}^\top \\
        &\qquad\qquad\qquad\qquad + \sum_{i<j} \rbr{e_i^l \rbr{e_j^{l}}^\top + e_j^l \rbr{e_i^{l}}^\top} \tilde{V}_n G \tilde{V}_n \rbr{e_i^l \rbr{e_j^{l}}^\top + e_j^l \rbr{e_i^{l}}^\top} \Bigg] \tilde{V}_n \Bigg\}.
    \end{align}
    By applying the following identity
    \begin{align}
        &\quad 2\sum_{i=1}^n e_i^l \rbr{e_i^{l}}^\top \tilde{V}_n G \tilde{V}_n e_i^l \rbr{e_i^{l}}^\top + \sum_{i<j} \rbr{e_i^l \rbr{e_j^{l}}^\top + e_j^l \rbr{e_i^{l}}^\top} \tilde{V}_n G \tilde{V}_n \rbr{e_i^l \rbr{e_j^{l}}^\top + e_j^l \rbr{e_i^{l}}^\top} \\
        &= \sum_{i,j=1}^n e_i^l (\tilde{V}_n G \tilde{V}_n)_{ji}^{ll} \rbr{e_j^{l}}^\top + \sum_{i,j=1}^n e_i^l (\tilde{V}_n G \tilde{V}_n)_{jj}^{ll} \rbr{e_i^{l}}^\top,
    \end{align}
    we can further rephrase $\EE\sbr{ G \tilde{H}_1 }$ by
    \begin{equation}\label{eq:E_GH1_0 another}
        \EE\sbr{ G \tilde{H}_1 } = -\EE\sbr{G \tilde{V}_n \Diag\cbr{\lambda_1 E_1,\ldots,\lambda_L E_L}\tilde{V}_n},
    \end{equation}
    where each $E_l$ is an $n\times n$ matrix whose entries are given by
    \begin{align}
        (E_l)_{ij} &= \frac{\sum_{k=1}^n (\tilde{V}_n G \tilde{V}_n)_{kk}^{ll}}{n} \1_{i=j} + \frac{(\tilde{V}_n G \tilde{V}_n)_{ji}^{ll}}{n} \\
        &= \frac{1}{n} \Tsf\sbr{ \tilde{V}_n G \tilde{V}_n }^{ll} \1_{i=j} + \frac{(\tilde{V}_n G \tilde{V}_n)_{ji}^{ll}}{n}.
    \end{align}
    Introduce a vector $\hat{\chi}=\diag\cbr{\frac{1}{n} \Tsf\sbr{ \tilde{V}_n G \tilde{V}_n }}\in\CC^L$ whose coordinates are given by
    \begin{equation}\label{eq:chi_equations_0}
        \hat{\chi}_l = \frac{1}{n} \sum_{i=1}^n (\tilde{V}_n G \tilde{V}_n)_{ii}^{ll} .
    \end{equation}
    It is straightforward to show that 
    \begin{equation}
        E_l = \hat{\chi}_l I_n + \frac{1}{n} \Delta_l, \qquad l\in[L],
    \end{equation}
    where $\Delta_1, \Delta_2,\ldots, \Delta_L$ are $n\times n$ residual matrices such that $\Delta_{l,ij} = (\tilde{V}_n G \tilde{V}_n)_{ji}^{ll}$.
    Therefore, we are able to rewrite $\tilde{V}_n \Diag\cbr{\lambda_1 E_1,\ldots,\lambda_L E_L}\tilde{V}_n$ into the following form,
    \begin{align}
        &\quad \Diag\cbr{\lambda_1 E_1,\ldots,\lambda_L E_L} \\
        &=  \Diag\cbr{ \lambda_l\rbr{ \hat{\chi}_l I_n + \frac{1}{n} \Delta_l  }; l\in[L] } \\
        &= \Diag\rbr{\lambda\odot\hat{\chi}} \otimes I_n + \frac{1}{n} \Diag\cbr{\lambda_1 \Delta_1,\ldots,\lambda_L \Delta_L}.
    \end{align}
    From definition \eqref{eq:resolvent def}, we find that
    \begin{equation}
        \tilde{V}_n G \tilde{V}_n = [\tilde{V}_n^{-1}\tilde{H}_1\tilde{V}_n^{-1} - \Lambda \otimes I_n + (V^{-1}\Gamma V^{-1}) \otimes I_n - z\rbr{B^{-1}\otimes I_n}]^{-1}.
    \end{equation}
    It further induces that
    \begin{equation}
        I_{nL} = \tilde{V}_n G \tilde{H}_1 \tilde{V}_n^{-1} - \tilde{V}_n G \tilde{V}_n \sbr{ \Lambda  \otimes I_n} + \tilde{V}_n G \tilde{V}_n \sbr{ V^{-1} \Gamma V^{-1}  \otimes I_n} - z \tilde{V}_n G \tilde{V}_n \rbr{B^{-1}\otimes I_n} .
    \end{equation}
    Taking expectation and plugging in~\eqref{eq:E_GH1_0 another}, it follows that
    \begin{equation}\label{eq:G_equation}
        I_{nL} = \EE\cbr{ \tilde{V}_n G \tilde{V}_n \cdot \sbr{-\Diag\rbr{\lambda\odot\hat{\chi}} \otimes I_n - (\Lambda-V^{-1} \Gamma V^{-1})  \otimes I_n - zB^{-1} \otimes I_n} } - \tilde{\Delta},
    \end{equation}
    where $\tilde{\Delta}$ satisfies the following
    \begin{equation}
        \tilde{\Delta} = \frac{1}{n} \EE\sbr{ \tilde{V}_n G \tilde{V}_n \cdot \Diag\cbr{\lambda_1\Delta_1,\ldots,\lambda_L\Delta_L} } .
    \end{equation}
    Subsequently, let $\cT_{G,n} := \frac{1}{n} \Tsf\sbr{ \tilde{V}_n G \tilde{V}_n }\in\SA_L(\CC)$.
    Applying the partial trace~\eqref{eq:partial trace} and noting that $\Tsf(A(A_1 \otimes I_n)) = \Tsf(A) A_1$, we conclude that
    \begin{equation}\label{eq:G_equation under partial trace}
        I_L = \EE\cbr{ \cT_{G,n} \cdot \sbr{-\Diag\rbr{\lambda\odot\diag\rbr{\cT_{G,n}}} - (\Lambda-V^{-1} \Gamma V^{-1}) - zB^{-1} } } - \frac{1}{n}\Tsf\sbr{ \tilde{\Delta} }.
    \end{equation}
    For simplicity, let
    \begin{equation}
        F(\cT) = I_L - \cT \cdot \sbr{-\Diag\rbr{\lambda\odot\diag\rbr{\cT}} - (\Lambda-V^{-1} \Gamma V^{-1}) - zB^{-1} }
    \end{equation}
    be a functional in $\cT\in\SA_L(\CC)$. Then the previous equation yields that
    \begin{equation}\label{eq:error from Stein}
        \nbr{ \EE F(\cT_{G,n}) } = \frac{1}{n}\nbr{\Tsf\sbr{\tilde{\Delta}}} \le \frac{C}{n\rbr{\Im z}^2}.
    \end{equation}
    Since  $\nbr{\cT_{G,n}} \le C(B,\lambda)/\Im z$ is almost surely bounded, 
    \begin{equation}
        \nbr{ \nabla F(c\cT_{G,n}+(1-c)\EE\cT_{G,n}) } \le  C(B,\lambda)/\Im z, \quad\forall c\in[0,1].
    \end{equation}
    Since $\cT_{G,n}=V\rbr{\frac{1}{n}\Tsf\sbr{G}}V$ and $V$ is fixed, the concentration estimate~\eqref{eq:concentration of the partial trace} gives
    \begin{align}
        \nbr{\EE F(\cT_{G,n}) - F(\EE \cT_{G,n})} &\le C/\Im z \cdot \EE\nbr{ \cT_{G,n} - \EE \cT_{G,n} } \\
        &\le C/\Im z \int_{0}^{\infty} \PP\rbr{ \nbr{ \cT_{G,n} - \EE \cT_{G,n} } \ge\delta} \ud\delta \\
        &\le \frac{C}{\Im z} \int_{0}^{\infty} \exp\rbr{- C \rbr{\Im z}^4 n^2 \delta^2} \ud\delta = \frac{C}{n\rbr{\Im z}^3}.
    \end{align}
    In combination with~\eqref{eq:error from Stein}, we end up with
    \begin{equation}
        \nbr{ F(\EE\cT_{G,n}) } \le \frac{C}{n\rbr{\Im z}^3}.
    \end{equation}
    Proposition~\ref{prop:stability} is stated for the left-handed quadratic form in~\eqref{eq:perturbed quadratic eq}. Since $\EE\cT_{G,n}$ and the bracket in $F(\EE\cT_{G,n})$ are complex symmetric, transposing the preceding display gives the same bound for the left-handed form. Applying Proposition~\ref{prop:stability} with $\widetilde M=-\EE\cT_{G,n}$ therefore implies that
    \begin{equation}
        \nbr{ \EE\cT_{G,n} + M(z;\Gamma) } \le \frac{C}{n\rbr{\Im z}^3}.
    \end{equation}
    Since $\cT_{G,n}=V\rbr{\frac{1}{n}\Tsf\sbr{G}}V$, this is equivalent to the same estimate for $\EE\cbr{\frac{1}{n}\Tsf\sbr{G}}+V^{-1}M(z;\Gamma)V^{-1}$. Lastly, the concentration inequality~\eqref{eq:concentration of the partial trace} gives the asserted stochastic domination estimate.
\end{proof}



\subsection{Stieltjes transforms and tested deterministic equivalents}
\label{sec:deterministic_equiv_proof}
We now record two consequences of the partial-trace estimate. First, taking the trace gives convergence of the empirical spectral distribution of the noise matrix. Second, a slightly more detailed version of the same Stein expansion gives the deterministic equivalent in Proposition~\ref{prop:equivalent}. We begin with the Stieltjes transforms
\begin{align}
    S_{\cL_n}(z) &:= \frac{1}{nL}\tr\sbr{ \rbr{ H_1+ \Gamma \otimes I_n - zI_{nL}  }^{-1} }, \\
    S_{\mu_{\Gamma}}(z) &:= \int_{\RR}\frac{\mu_{\Gamma}(\ud\sigma)}{\sigma-z}.
\end{align}
Here $\mu_{\Gamma}$ is the limiting distribution from Definition~\ref{def:bulk dist limit}.

\begin{lemma}\label{lemma:convergence of ESD}
    The empirical distribution $\cL_n := \frac{1}{nL}\sum_{j=1}^{nL} \delta_{\sigma_j(H_1+ \Gamma \otimes I_n)}$ converges weakly in probability to $\mu_{\Gamma}$ in Definition~\ref{def:bulk dist limit}.
\end{lemma}

\begin{proof}
    By Proposition~\ref{prop:partial_trace_equivalent}, for every fixed $z\in\CC^+$,
    \begin{equation}
        S_{\cL_n}(z)
        =\frac{1}{L}\tr\sbr{\frac{1}{n}\Tsf\sbr{G(z;\Gamma)}}
        =-\frac{1}{L}\tr\sbr{B^{-1}M(z;\Gamma)}+r_n(z),
        \qquad r_n(z)\to0\quad\text{in probability}.
    \end{equation}
    By the Stieltjes representation~\eqref{eq:Stieltjes representation} and the definition of $\mu_{\Gamma}$ in~\eqref{eq:def mu gamma}, the deterministic term equals $S_{\mu_{\Gamma}}(z)$. Hence $S_{\cL_n}(z)$ converges in probability to $S_{\mu_{\Gamma}}(z)$ for every fixed $z\in\CC^+$. The Stieltjes continuity theorem~\cite[Theorem 2.4.4]{anderson2010introduction} implies that $\cL_n$ converges weakly in probability to $\mu_{\Gamma}$.
\end{proof}

While the previous quatitative deterministic equivalent results can imply a global law on the spectrum of $H_1+\Gamma\otimes I_n$ via Lemma~\ref{lemma:convergence of ESD}, it actually requires a slightly faster rate to guarantee the behavior of individual eigenvalues. For convenience, we use the main results of \cite{alt2019location} to directly deduce the next lemma.


\begin{lemma}[No eigenvalues outside the limiting support]\label{lemma:no-eigenvalue outside support}
    For a compact set $S\subset\RR$, write
    \[
        S^{(\delta)}:=\cbr{x\in\RR:\dist(x,S)\le\delta}
    \]
    for its closed $\delta$-neighborhood.
    For any $\delta,D>0$, we have
    \begin{equation}
        \PP\sbr{ \spec\rbr{H_1+ \Gamma \otimes I_n} \subseteq \supp(\mu_{\Gamma})^{(\delta)} } \ge 1-n^{-D}
    \end{equation}
    for all sufficiently large $n$.
\end{lemma}

\begin{proof}
    It can be verified that our model is also included in the general class of random matrices studied by Section 4 of \cite{alt2019location}, namely the definition in their equation (4.2). Directly through their Theorem 4.7, there exists a $\delta_0>0$ such that for each $D>0$, there exists a constant $C_D>0$ such that
    \begin{equation}
        \PP\sbr{ \spec\rbr{H_1+ \Gamma \otimes I_n} \subseteq \cbr{ \tau\in\RR: \dist(\tau,\supp(\mu_{\Gamma}))\le N^{-\delta_0} }  } \ge 1-\frac{C_D}{N^D}.
    \end{equation}
    This result can yield the conclusion of this lemma.
\end{proof}




\begin{proof}[Proof of Proposition~\ref{prop:equivalent}]
    Write $G=G(z;\Gamma)$, $\tilde V_n=V\otimes I_n$, and
    \[
        \cT_{G,n}:=\frac{1}{n}\Tsf\sbr{\tilde V_nG\tilde V_n}.
    \]
    In the proof of Proposition~\ref{prop:partial_trace_equivalent}, Stein's identity gave the full matrix relation
    \begin{equation}
        I_{nL} = \EE\cbr{ \tilde{V}_n G \tilde{V}_n \cdot \sbr{-\Diag\rbr{\lambda\odot\hat{\chi}} \otimes I_n - (\Lambda-V^{-1} \Gamma V^{-1})  \otimes I_n - zB^{-1} \otimes I_n} } - \tilde{\Delta},
    \end{equation}
    where
    \begin{align}
        \tilde{\Delta} &= \frac{1}{n} \EE\sbr{ \tilde{V}_n G \tilde{V}_n \cdot \Diag\cbr{\lambda_1\Delta_1,\ldots,\lambda_L\Delta_L} }, \\
        (\Delta_l)_{ij} &= (\tilde V_nG\tilde V_n)_{ji}^{ll}, \qquad l\in[L],\\
        \hat{\chi}_l &= \frac{1}{n} \sum_{i=1}^n (\tilde{V}_n G \tilde{V}_n)_{ii}^{ll} ,\quad\forall l\in[L].
    \end{align}
    Since $\|\tilde V_nG\tilde V_n\|\le C(B)/\Im z$ and $L$ is fixed,
    \begin{equation}
        \nbr{\tilde{\Delta}}_F \le \frac{C(B,\lambda)}{\sqrt{n}\rbr{\Im z}^2}.
    \end{equation}
    Proposition~\ref{prop:partial_trace_equivalent} implies
    $\hat\chi+\diag\sbr{M(z;\Gamma)}
    =\cO_{\prec}\rbr{1/(n(\Im z)^3)}$; integrating this stochastic domination, using the deterministic bound $\|\cT_{G,n}\|\le C(B)/\Im z$, gives the corresponding expectation estimate. Therefore,
    \begin{align}
        &\quad \nbr{ I_{nL} - \EE\rbr{ \tilde{V}_n G \tilde{V}_n } \cdot \sbr{\Diag\rbr{\lambda\odot\diag(M(z;\Gamma))} \otimes I_n - (\Lambda-V^{-1} \Gamma V^{-1})  \otimes I_n - zB^{-1} \otimes I_n} }_F\\
        &\le \frac{C(B,\lambda)}{\Im z}\,\EE\nbr{\rbr{\hat{\chi}+\diag\sbr{M(z;\Gamma)}}\otimes I_n}_F + \nbr{\tilde{\Delta}}_F
        \le \frac{C(B,\lambda)}{\sqrt{n}\rbr{\Im z}^4}.
    \end{align}
    Then we right multiply the matrix inside the norm bracket with a deterministic matrix
    \begin{equation}
        \sbr{\Diag\rbr{\lambda\odot\diag(M(z;\Gamma))} - (\Lambda-V^{-1} \Gamma V^{-1}) - zB^{-1} }^{-1} \otimes I_n
    \end{equation}
    which is equal to $-M(z;\Gamma)\otimes I_n$ by~\eqref{eq: self-consistent eq matrix}. We obtain
    \begin{equation}
        \nbr{\EE G(z;\Gamma) - \cG(z;\Gamma)}_F \le \frac{C(B,\lambda)}{\sqrt{n}\rbr{\Im z}^4}.
    \end{equation}
    Thus
    \[
        \tr\sbr{A\rbr{\EE G(z;\Gamma)-\cG(z;\Gamma)}}
        =O\rbr{\frac{\|A\|_F}{\sqrt n(\Im z)^4}}.
    \]
    Finally, the same Gaussian concentration argument used in Appendix~\ref{app:resolvent-concentration}, now with the Lipschitz constant measured in the Frobenius norm of $A$, gives
    \[
        \tr\sbr{A\rbr{G(z;\Gamma)-\EE G(z;\Gamma)}}
        =\cO_{\prec}\rbr{\frac{\|A\|_F}{\sqrt n(\Im z)^2}}.
    \]
    Combining the last two displays proves the claim, after enlarging the constant in the weaker denominator $(\Im z)^4$.
\end{proof}

\subsection{Proof of Theorem~\ref{thm:main ESD}}
Recall from~\eqref{eq:H_H0H1} that $H=H_0+H_1$, and that $\rank(H_0)\le L$. A standard rank inequality for Hermitian matrices gives
\begin{equation}
    \sup_{x\in\RR}
    \abr{F_H(x)-F_{H_1}(x)}
    \le \frac{\rank(H-H_1)}{nL}
    \le \frac{1}{n},
\end{equation}
where $F_A(x)$ denotes the empirical distribution function of the eigenvalues of $A$. By Lemma~\ref{lemma:convergence of ESD} with $\Gamma=0$, the empirical spectral distribution of $H_1$ converges weakly in probability to $\mu$. The preceding display transfers this convergence from $H_1$ to $H$, proving the bulk convergence claimed in Theorem~\ref{thm:main ESD}.

It remains to justify the statement that at most $L$ eigenvalues can lie to the right of the limiting bulk. More precisely, fix $\delta,D>0$ and let $\sigma_+=\sup\supp(\mu)$. By Lemma~\ref{lemma:no-eigenvalue outside support} with $\Gamma=0$, with probability at least $1-n^{-D}$ for all sufficiently large $n$, the noise matrix $H_1$ has no eigenvalues in $(\sigma_++\delta,\infty)$. On this event, the rank inequality for eigenvalue counting functions gives
\begin{equation}
    \#\cbr{j:\sigma_j(H)>\sigma_++\delta}
    \le
    \#\cbr{j:\sigma_j(H_1)>\sigma_++\delta}+\rank(H-H_1)
    \le L.
\end{equation}
Hence, outside any fixed neighborhood of the right edge of the limiting support, $H$ has at most $L$ eigenvalues on the right. This proves the finite-rank nature of the possible right outliers claimed in Theorem~\ref{thm:main ESD}.

\section{Outlier eigenvalues and eigenvector overlaps}\label{sec:outlier}
This section proves Propositions~\ref{prop:sketch-statement-outlier}, \ref{prop:sketch-statement-overlap}, and~\ref{prop:sketch-matrix-overlaps} from the proof outline. We first reduce the outlier eigenvalues to the zeros of an $L\times L$ deterministic matrix, then identify the corresponding one-dimensional spectral projections, and finally convert these projection formulas into the matrix overlaps needed for signal recovery.

We use the notation introduced in Section~\ref{sec:proof_outline}: $H=H_0+H_1$ with $H_0=UU^\top$, the noise resolvent $G(z)$ from~\eqref{eq:G_H1}, the finite-dimensional matrices $\cQ_n(z)$ and $\cQ(z)$ from~\eqref{eq:Resolvent of spiked matrix model}, and the deterministic outlier set $\Psi_0$ from Proposition~\ref{prop:sketch-statement-outlier}. Throughout this section, $M(z)$ denotes the $\Gamma=0$ solution of the matrix Dyson equation~\eqref{eq: self-consistent eq matrix}.

Following the convention introduced in Lemma~\ref{lemma:no-eigenvalue outside support}, for a closed set $S\subset\RR$ and $\delta>0$ we write
\[
    S^{(\delta)}:=\cbr{x\in\RR:\dist(x,S)\le \delta}
\]
for its closed $\delta$-neighborhood.
Fix $\epsilon>0$ small enough that every point of $\Psi_0$ lies outside $\supp(\mu)^{(2\epsilon)}$; if $\Psi_0=\emptyset$, choose any sufficiently small $\epsilon$. We will also use the real-line extension of the stability domain,
\begin{equation}
    \bar{\cU}(\epsilon) = \cbr{z\in\CC: \dist\rbr{z,\supp\mu} \ge \epsilon, \abr{z} \le \epsilon^{-1}},
\end{equation}
whose intersection with $\CC^+$ is the domain $\cU(\epsilon)$ from~\eqref{eq:domain for unfirom convergence}. The proofs below are carried out on the event
\begin{equation}
    \cE = \cbr{ \spec\rbr{H_1} \subseteq \supp(\mu)^{(\epsilon/2)} }.
\end{equation}
Here $\supp(\mu)^{(\epsilon/2)}$ denotes the closed $\epsilon/2$-neighborhood of $\supp(\mu)$, in the notation introduced in Lemma~\ref{lemma:no-eigenvalue outside support}.
By Lemma~\ref{lemma:no-eigenvalue outside support}, for every $D>0$ we have $\PP(\cE)>1-n^{-D}$ for all sufficiently large $n$. Thus estimates proved with stochastic domination on $\cE$ imply the unconditional estimates stated in Section~\ref{sec:proof_outline}.

\subsection{The outlier equation}\label{subsec:outlier-equation}

This subsection proves Proposition~\ref{prop:sketch-statement-outlier}. The main point is to compare the finite-dimensional random matrix $\cQ_n(z)$ with its deterministic limit $\cQ(z)$, and then to count the zeros of their determinants outside the bulk.

\begin{lemma}\label{lemma:convergence of UGU}
    Uniformly over $z \in \bar{\cU}(\epsilon)\cap\RR^c$, we have $\nbr{\cQ_n(z)-\cQ(z)}=\cO_{\prec}\rbr{1/\rbr{\sqrt{n}\abr{\Im z}^4}}$.
\end{lemma}
\begin{proof}
    Conditional on the spikes, the entries of $U^\top G(z)U$ are tested resolvent entries of the form $\tr\sbr{A G(z)}$, with $\|A\|_F=\cO_{\prec}(1)$ by Assumption~\ref{assump:spike concentration}. Using Proposition~\ref{prop:equivalent} and symmetry across the real line, we conclude that uniformly over $z\in \bar{\cU}(\epsilon)\cap\RR^c$,
    \begin{equation}
        \nbr{\cQ_n(z) - (I_{L} + U^\top \cG(z) U)} = \cO_{\prec}\rbr{\frac{1}{\sqrt{n}\abr{\Im z}^4}}.
    \end{equation}
    Exploiting the tensor-product structures in both $\cG(z)$ and $U$, 
    \begin{align}
        &\quad U^\top \cG(z) U \\
        &= -\frac{1}{n}\diag\cbr{\sqrt{\lambda_1}X^{(1)\top},\ldots,\sqrt{\lambda_L}X^{(L)\top}} \cbr{\sbr{ M(z) }\otimes I_n} \diag\cbr{\sqrt{\lambda_1}X^{(1)},\ldots,\sqrt{\lambda_L}X^{(L)}}\\
        &= -\cbr{\frac{1}{n}\sbr{\sqrt{\lambda_1}X^{(1)},\ldots,\sqrt{\lambda_L}X^{(L)}}^\top\sbr{\sqrt{\lambda_1}X^{(1)},\ldots,\sqrt{\lambda_L}X^{(L)}}} \odot \sbr{ M(z) }.
    \end{align}
    Since $M(z)$ is uniformly bounded on $\bar{\cU}(\epsilon)$, Assumption~\ref{assump:spike concentration} gives
    \begin{equation}
        \nbr{ U^\top \cG(z) U + \sqrt{\Lambda} \sbr{B \odot M(z)} \sqrt{\Lambda} } = \cO_{\prec}\rbr{1/\sqrt{n}}.
    \end{equation}
    Therefore, we can conclude our lemma.
\end{proof}

\begin{lemma}\label{lemma:eigen-asymptotic qualitative}
    On the event $\cE$, every eigenvalue $\sigma_n\in\RR\backslash\supp(\mu)^{(\epsilon)}$ of $H$ satisfies $\det \cQ_n(\sigma_n)=0$. Moreover, for all sufficiently large $n$, these eigenvalues are in one-to-one correspondence with the elements of $\Psi_0$, counted with multiplicity.
\end{lemma}

\begin{proof}
    On $\cE$, every $\sigma\in\RR\backslash\supp\rbr{\mu}^{(\epsilon)}$ is separated from $\spec(H_1)$, so $(H_1-\sigma I_{nL})^{-1}$ is well defined. Using $H_0=UU^\top$ and Sylvester's identity,
    \begin{align}
        0 &= \det\rbr{ H_0 + H_1 - \sigma I_{nL} } \\
        &= \det\rbr{H_1 - \sigma I_{nL}} \cdot \det\rbr{ I_{nL} + \rbr{H_1-\sigma I_{nL}}^{-1}H_0 } \\
        &= \det\rbr{H_1 - \sigma I_{nL}} \cdot \det\rbr{ I_{L} + U^\top\rbr{H_1-\sigma I_{nL}}^{-1}U } = \det\rbr{H_1 - \sigma I_{nL}} \cdot \det \cQ_n(\sigma).
    \end{align}
    Since $\det\rbr{H_1 - \sigma I_{nL}} \neq 0$, we conclude that under event $\cE$, every outlier eigenvalue $\sigma_n\in \RR\backslash\supp\rbr{\mu}^{(\epsilon)}$ of $H$ must satisfy $\det \cQ_n(\sigma_n)=0$.

    Recall that under $\cE$, both $\cQ_n(z)$ and $\cQ(z)$ are holomorphic on $\bar{\cU}(\epsilon)$. We now count their zeros by the argument principle. Consider any fixed interval $[c,d]\subset\bar{\cU}(\epsilon)\cap\RR$, with both endpoints $c,d$ not zeros of $\det \cQ(z)$.
    To simplify the notation, first suppose that every zero of $\det\cQ(z)$ in the interval is simple; multiple zeros are handled by the same argument with multiplicities.
    Let $\gamma\subset\bar{\cU}(\epsilon)$ be the boundary of a rectangle with vertical sides through $c$ and $d$ and with fixed height. For some $\iota>0$ to be chosen later, we split the contour $\gamma=\gamma_1\cup\gamma_2$ by
    \begin{equation}
        \gamma_1:=\cbr{z\in\gamma:\abr{\Im z} \le n^{-\iota}}, \quad \gamma_2 := \gamma\backslash\gamma_1.
    \end{equation}
    We further choose the rectangle so that
    \begin{equation}
        \abr{\gamma_1} \le 10n^{-\iota}, \quad \abr{\gamma_2} \le 10\abr{c-d}.
    \end{equation}
    The number of zeros inside $[c,d]$ is represented by
    \begin{align}
        \Card_{[c,d],n} &= \frac{1}{2\pi i}\oint_{\gamma} \frac{\partial_z \det \cQ_n(z)}{\det \cQ_n(z)} \ud z = \frac{1}{2\pi i}\oint_{\gamma} \tr\sbr{\cQ_n^\prime(z)\cQ_n(z)^{-1}} \ud z, \\
        \Card_{[c,d]} &= \frac{1}{2\pi i}\oint_{\gamma} \frac{\partial_z \det \cQ(z)}{\det \cQ(z)} \ud z = \frac{1}{2\pi i}\oint_{\gamma} \tr\sbr{\cQ^\prime(z)\cQ(z)^{-1}} \ud z.
    \end{align}
    On $z\in\gamma_1$, since $[c,d]\subset\bar{\cU}(\epsilon)$ is bounded away from $\supp(\mu)$, the function $\tr\sbr{\cQ^\prime(z)\cQ(z)^{-1}}$ is uniformly bounded. Under $\cE$, the random quantity $\tr\sbr{\cQ_n^\prime(z)\cQ_n(z)^{-1}}$ is also uniformly bounded. Since $|\gamma_1|$ is small, the integral over this part is negligible.
    
    On $z\in\gamma_2$, Lemma~\ref{lemma:convergence of UGU} gives $\nbr{\cQ_n(z)-\cQ(z)}=\cO_{\prec}(n^{4\iota-1/2})$ since $\abr{\Im z}\ge n^{-\iota}$. This estimate is uniform in a small neighborhood of $\gamma_2$, so $\nbr{\cQ_n^\prime(z)-\cQ^\prime(z)}=\cO_{\prec}(n^{4\iota-1/2})$ as well. Since $\gamma$ is fixed with endpoints $c,d$ away from the zeros of $\cQ(z)$, we have $\min_{z\in\gamma_2}\sigma_{\min}\rbr{\cQ(z)}\ge\kappa>0$. This guarantees that $\nbr{ \cQ_n(z)^{-1}-\cQ(z)^{-1} }=\cO_{\prec}(n^{4\iota-1/2})$, and hence uniformly over $z\in\gamma_2$,
    \begin{equation}
        \abr{ \tr\sbr{\cQ_n^\prime(z)\cQ_n(z)^{-1}}-\tr\sbr{\cQ^\prime(z)\cQ(z)^{-1}} } = \cO_{\prec}(n^{4\iota-1/2}).
    \end{equation}

    Since $|\gamma|$ stays bounded in $n$, choosing $\iota=1/10$ gives
    \begin{equation}
        \abr{\Card_{[c,d],n}-\Card_{[c,d]}} = \cO_{\prec}(n^{-\iota}+n^{4\iota-1/2}) = \cO_{\prec}(n^{-1/10}).
    \end{equation}
    Since the two quantities are integers, they must be equal for all sufficiently large $n$. The same contour argument applied around each connected component containing zeros of $\det\cQ$ gives the correspondence with multiplicities.
\end{proof}

We now prove the quantitative part of Proposition~\ref{prop:sketch-statement-outlier}. The one-to-one correspondence follows from Lemma~\ref{lemma:eigen-asymptotic qualitative}.

\begin{proof}[Proof of Proposition~\ref{prop:sketch-statement-outlier}]
    For any $\sigma\in\Psi_0$ with multiplicity $m:=\deg(\sigma) \ge 1$, the function $\det\cQ(z)$ is holomorphic and has a zero of order $m$ at $\sigma$. Hence
    \begin{equation}
        \abr{ \det \cQ(z)} \ge c |z-\sigma|^m
    \end{equation}
    locally around $z \approx \sigma$. We pick a rectangular contour
    \begin{equation}
        \gamma_{\sigma,\eta}: \quad \Re z\in[\sigma-2\eta,\sigma+2\eta], \quad\Im z \in[-\eta,\eta].
    \end{equation}
    If $\eta = n^{-\frac{1}{2(m+4)}+\delta}$ for any fixed $\delta>0$, then
    \begin{equation}
        \frac{\sup_{z\in\RR^c\cap\gamma_{\sigma,\eta}} \abr{\det \cQ_n(z)-\det \cQ(z)}}{\inf_{z\in\RR^c\cap\gamma_{\sigma,\eta}} \abr{\det \cQ(z)}} = \cO_{\prec}\rbr{\frac{1}{\sqrt{n}\eta^4}\cdot\frac{1}{\eta^m}}
    \end{equation}
    is vanishing. By continuity this also controls the real endpoints of the rectangle. Once the ratio is smaller than one, Rouch\'e's theorem implies that $\det\cQ_n$ has exactly $m$ zeros inside $\gamma_{\sigma,\eta}$, counted with multiplicity. By Lemma~\ref{lemma:eigen-asymptotic qualitative}, these zeros are precisely the corresponding outlier eigenvalues of $H$, and hence $\abr{\sigma_n-\sigma}=\cO_{\prec}\rbr{n^{-\frac{1}{2m+8}}}$.
\end{proof}

\subsection{Spectral projection for a simple outlier}\label{subsec:outlier-projection}

We now prove Proposition~\ref{prop:sketch-statement-overlap}. Assume throughout this subsection that $\sigma\in\Psi_0$ is a simple outlier, and let $\sigma_n$ be the corresponding eigenvalue of $H$. Choose a positively oriented rectangular contour $\gamma\subset\bar{\cU}(\epsilon)$ enclosing $\sigma$ and no other point of $\Psi_0$. By Proposition~\ref{prop:sketch-statement-outlier}, for all sufficiently large $n$ the contour encloses $\sigma_n$ and no other eigenvalue of $H$ outside the bulk.

The spectral projection onto the eigenspace of $\sigma_n$ is recovered by Cauchy's formula:
\begin{equation}
    \xi_1^\top \nu_n \nu_n^\top \xi_2 = -\frac{1}{2\pi i}\oint_{\gamma} \xi_1^\top (H-zI_{nL})^{-1} \xi_2 \ud z.
\end{equation}
Using the Woodbury formula and the definition of $\cQ_n(z)$,
\begin{align}
    \xi_1^\top (H-zI_{nL})^{-1} \xi_2
    &= \xi_1^\top G(z) \xi_2
    - \xi_1^\top G(z) U \sbr{I_L+U^{\top} G(z) U}^{-1} U^\top G(z) \xi_2 \\
    &= \xi_1^\top G(z) \xi_2
    - \xi_1^\top G(z) U \cQ_n(z)^{-1} U^\top G(z) \xi_2.
\end{align}
On $\cE$, the spectrum of $H_1$ lies inside $\supp(\mu)^{(\epsilon/2)}$, while $\gamma$ is separated from $\supp(\mu)$. Thus $\xi_1^\top G(z)\xi_2$ has no pole inside $\gamma$, and its contour integral vanishes. Hence
\begin{equation}
    \xi_1^\top \nu_n \nu_n^\top \xi_2 = \frac{1}{2\pi i}\oint_{\gamma} \xi_1^\top G(z) U \cQ_n(z)^{-1} U^\top G(z) \xi_2 \ud z.
\end{equation}
We will compare this integral with the deterministic one
\begin{equation}\label{eq:def A(xi1,xi2)}
    A(\xi_1,\xi_2,U) := \frac{1}{2\pi i}\oint_{\gamma} \xi_1^\top \cG(z) U \cQ(z)^{-1} U^\top \cG(z) \xi_2 \ud z.
\end{equation}
For $j=1,2$, set
\begin{align}
    a_{n,j}(z) &:= U^\top G(z) \xi_j \in \CC^L, \\
    a_{j}(z) &:= U^\top \cG(z) \xi_j \in \CC^L.
\end{align}

Let the rectangle $\gamma$ have vertical sides $\Re z=\sigma\pm E$ and horizontal sides $\Im z=\pm\eta$, where $E,\eta>0$ are fixed small constants chosen so that the enclosed real interval contains no point of $\Psi_0$ other than $\sigma$ and remains separated from $\supp(\mu)$. Split
\begin{align}
    \gamma_1 &= \cbr{ z\in\gamma: \abr{\Re z-\sigma}=E,\ |\Im z|\le n^{-\iota} }, \\
    \gamma_2 &= \cbr{ z\in\gamma: \abr{\Re z-\sigma}=E,\ n^{-\iota}\le|\Im z|\le\eta }, \\
    \gamma_3 &= \cbr{ z\in\gamma: |\Re z-\sigma|\le E,\ \abr{\Im z}=\eta }.
\end{align}
On $\gamma_1$, the deterministic factors $a_1(z),a_2(z),\cQ(z)^{-1}$ are uniformly bounded. On $\cE$, the corresponding random factors $a_{n,1}(z),a_{n,2}(z),\cQ_n(z)^{-1}$ are also uniformly bounded for all sufficiently large $n$, because the vertical sides are separated from both the bulk spectrum and the outlier location. Since $|\gamma_1|=\cO(n^{-\iota})$, this part of the contour contributes $\cO_{\prec}(n^{-\iota})$.

On $\gamma_2\cup\gamma_3$, the imaginary part is at least $n^{-\iota}$. Conditional on the spikes, each entry of $a_{n,j}(z)$ is a tested resolvent entry, so Proposition~\ref{prop:equivalent} gives
\begin{equation}
    \nbr{ a_{n,j}(z) - a_{j}(z) } = \cO_{\prec}(n^{4\iota-1/2}),\qquad j=1,2,
\end{equation}
uniformly on $\gamma_2\cup\gamma_3$. This part of the contour is separated from the zero $\sigma$ of $\det\cQ$, and hence $\sigma_{\min}(\cQ(z))\ge\kappa>0$ there. Together with Lemma~\ref{lemma:convergence of UGU}, this implies
\begin{equation}
    \nbr{\cQ_n(z)^{-1}-\cQ(z)^{-1}}=\cO_{\prec}(n^{4\iota-1/2})
\end{equation}
uniformly on $\gamma_2\cup\gamma_3$. A telescoping estimate then yields
\begin{align}
    &\quad \abr{ a_{n,1}(z)^\top \cQ_n(z)^{-1} a_{n,2}(z) - a_{1}(z)^\top \cQ(z)^{-1} a_{2}(z) } \\
    &\le 2\max\cbr{\nbr{a_{1}},\nbr{a_{n,1}},\nbr{a_{2}},\nbr{a_{n,2}}} \cdot \max\cbr{\nbr{\cQ_n(z)^{-1}},\nbr{\cQ(z)^{-1}}} \cdot  \max_{j=1,2}\abr{ a_{n,j}(z) - a_{j}(z) } \\
    &\qquad + \max\cbr{\nbr{a_{1}},\nbr{a_{n,1}},\nbr{a_{2}},\nbr{a_{n,2}}}^2 \cdot \nbr{\cQ_n(z)^{-1}-\cQ(z)^{-1}} \\
    &= \cO_{\prec}(n^{4\iota-1/2}).
\end{align}
Combining the three pieces of the contour and taking $\iota=1/10$ gives
\begin{equation}
    \abr{ \xi_1^\top \nu_n \nu_n^\top \xi_2 - A(\xi_1,\xi_2,U) } = \cO_{\prec}(n^{-1/10}).
\end{equation}

It remains to evaluate the deterministic contour integral. Since $\sigma$ is a simple zero of $\det\cQ(z)$ and $w$ spans the null space of $\cQ(\sigma)$,
\begin{equation}
    \cQ(z)^{-1} = \frac{1}{z-\sigma}\cdot\frac{ww^\top}{w^\top \cQ^\prime(\sigma) w} + \text{holomorphic terms}
\end{equation}
in a neighborhood of $\sigma$. Taking the residue in~\eqref{eq:def A(xi1,xi2)} yields
\begin{equation}
    A(\xi_1,\xi_2,U)=
    \frac{(\xi_1^\top \cG(\sigma) U w)(w^\top U^\top\cG(\sigma)\xi_2)}{w^\top \cQ^\prime(\sigma) w}.
\end{equation}
This is the estimate claimed in Proposition~\ref{prop:sketch-statement-overlap}. Since $\PP(\cE)>1-n^{-D}$ for every $D>0$, the same bound holds in the stated unconditional stochastic-domination form.

\subsection{Matrix overlaps and two-resolvent equivalents}\label{subsec:matrix-overlaps}

We now convert the spectral projection estimate into the overlaps used for signal recovery. As in Section~\ref{sec:proof_outline}, split $\nu_n\in\RR^{nL}$ into its $L$ view blocks by
\begin{equation}
    \nu_n^{(l)} := (e_l^\top \otimes I_n) \nu_n \in \RR^n, \quad \matop(\nu_n) := \begin{bmatrix} \nu_n^{(1)} & \cdots & \nu_n^{(L)}\end{bmatrix}\in\RR^{n \times L}.
\end{equation}
The next proposition gives the detailed matrix-overlap formulas. In particular, part~\ref{prop:matrix-overlap-with-X} is the first assertion of Proposition~\ref{prop:sketch-matrix-overlaps}, while part~\ref{prop:matrix-self-overlap} supplies the explicit self-overlap formula referred to there.

\begin{proposition}[Detailed matrix-overlap formulas]\label{prop:detailed-matrix-overlaps}
    In the setup of Proposition~\ref{prop:sketch-statement-overlap}, let
    \begin{equation}\label{eq:sign breaking}
        s(\nu_n) = \sgn\rbr{w^\top U^\top \nu_n}\in\{\pm 1\}
    \end{equation}
    fix the global sign of $\nu_n$. Then:
    \begin{enumerate}[label=(\alph*),ref=(\alph*)]
        \item\label{prop:overlap-with-U}
        \begin{align}
            \nbr{ U^\top \nu_n - s(\nu_n) \frac{w}{\sqrt{w^\top \cQ^\prime(\sigma) w}} } &= \cO_{\prec}(n^{-1/10});
        \end{align}

        \item\label{prop:matrix-overlap-with-X}
        \begin{equation}
            \nbr{ \frac{1}{\sqrt{n}} \matop(\nu_n)^\top X - \frac{s(\nu_n)}{\sqrt{w^\top \cQ^\prime(\sigma) w}} \cbr{ V^{-1}M(\sigma)\Diag\rbr{w\odot\sqrt{\lambda}}B }  } = \cO_{\prec}(n^{-1/10});
        \end{equation}

        \item\label{prop:matrix-self-overlap}
        Let $\cP_{\sigma}$ be the linear operator on $\SA_L(\RR)$ defined by
        \begin{equation}\label{eq:P_sigma operator}
            \cP_{\sigma}[A] := M(\sigma) \cbr{ A + \Diag\rbr{\lambda\odot \sbr{\sbr{I_L - M(\sigma)^{\odot 2}\Lambda}^{-1}\diag\rbr{ M(\sigma) A M(\sigma)}} } }M(\sigma).
        \end{equation}
        Then for any $l_1,l_2\in[L]$,
        \begin{align}
            \Big| & \langle\nu_n^{(l_1)},\nu_n^{(l_2)}\rangle-  \\
            &\quad \frac{1}{2w^\top\cQ^\prime(\sigma)w} w^\top\cbr{ \rbr{\sqrt{\Lambda}B\sqrt{\Lambda}} \odot \rbr{ \cP_{\sigma}\sbr{ V^{-1} \rbr{e_{l_1}e_{l_2}^\top+e_{l_2}e_{l_1}^\top} V^{-1} } } }w \Big| = \cO_{\prec}\rbr{ n^{-1/20} }.
        \end{align}
    \end{enumerate}
\end{proposition}

\begin{proof}[Proof of Proposition~\ref{prop:detailed-matrix-overlaps}\ref{prop:overlap-with-U}]
    Repeating the contour argument in the proof of Proposition~\ref{prop:sketch-statement-overlap}, now with the spike directions $U$ inserted on both sides, gives the same deterministic integral with $U^\top\cG(z)U$ in place of the deterministic test-vector factors. From Lemma~\ref{lemma:convergence of UGU},
    \begin{equation}
        \nbr{ U^\top \cG(z) U + I_L -\cQ(z) } = \cO_{\prec}\rbr{1/\sqrt{n}}.
    \end{equation}
    Evaluating at $z=\sigma$ and using $\cQ(\sigma)w=0$, we obtain
    \begin{equation}
        \nbr{ U^\top\nu_n\nu_n^\top U-\frac{ww^\top}{w^\top\cQ^\prime(\sigma)w} }=\cO_{\prec}(n^{-1/10}).
    \end{equation}
    The right-hand side is a rank-one projection scaled by $(w^\top\cQ^\prime(\sigma)w)^{-1}$. With the sign convention~\eqref{eq:sign breaking}, this matrix estimate implies the claimed vector estimate for $U^\top\nu_n$.
\end{proof}

\begin{proof}[Proof of Proposition~\ref{prop:detailed-matrix-overlaps}\ref{prop:matrix-overlap-with-X}]
    We again use the contour argument from the proof of Proposition~\ref{prop:sketch-statement-overlap}, this time testing against the spike direction $e_{l_1} \otimes \rbr{X^{(l_2)}/\sqrt{n}}$. The following deterministic vector then appears:
    \begin{align}
        &\quad U^\top \cG(\sigma) \sbr{e_{l_1} \otimes X^{(l_2)}/\sqrt{n}} \\
        &= -\frac{1}{n}\diag\cbr{\sqrt{\lambda_1}X^{(1)\top},\ldots,\sqrt{\lambda_L}X^{(L)\top}} \cbr{\sbr{ M(\sigma)V^{-1} }\otimes I_n} \sbr{e_{l_1} \otimes X^{(l_2)}} \\
        &= -\sum_{s=1}^{L} \sqrt{\lambda_s} (e_se_s^\top M(\sigma)V^{-1}e_{l_1}) \sbr{\frac{1}{n} X^{(s)\top}X^{(l_2)}} .
    \end{align}
    Since $M(z)$ is uniformly bounded on $\bar{\cU}(\epsilon)$, Assumption~\ref{assump:spike concentration} gives
    \begin{equation}
        \nbr{ U^\top \cG(\sigma) \sbr{e_{l_1} \otimes X^{(l_2)}/\sqrt{n}} + \sqrt{\lambda} \odot \sbr{M(\sigma)V^{-1}e_{l_1}} \odot \sbr{Be_{l_2}} } = \cO_{\prec}\rbr{1/\sqrt{n}}.
    \end{equation}
    Let $\xi_{l_1,l_2}:=e_{l_1}\otimes X^{(l_2)}/\sqrt n$. The eigenvalue equation gives
    \begin{equation}
        \nu_n=-(H_1-\sigma_n I_{nL})^{-1}UU^\top\nu_n.
    \end{equation}
    Using Proposition~\ref{prop:detailed-matrix-overlaps}\ref{prop:overlap-with-U}, the deterministic equivalent for $U^\top G(\sigma_n)\xi_{l_1,l_2}$, and the preceding display together with $\abr{\sigma_n-\sigma}=\cO_{\prec}(n^{-1/10})$, we obtain
    \begin{equation}
        \nu_n^\top \xi_{l_1,l_2}
        =
        -\frac{s(\nu_n)}{\sqrt{w^\top\cQ^\prime(\sigma)w}}w^\top U^\top \cG(\sigma)\xi_{l_1,l_2}
        +\cO_{\prec}(n^{-1/10}).
    \end{equation}
    Therefore,
    \begin{align}
        \langle\nu_n^{(l_1)}, X^{(l_2)}/\sqrt{n}\rangle
        &= \frac{s(\nu_n)}{\sqrt{w^\top \cQ^\prime(\sigma) w}} w^\top \cbr{ \sqrt{\lambda} \odot \sbr{M(\sigma)V^{-1}e_{l_1}} \odot  \sbr{Be_{l_2}} } + \cO_{\prec}(n^{-1/10})\\
        &= \frac{s(\nu_n)}{\sqrt{w^\top \cQ^\prime(\sigma) w}} \sum_{s=1}^L w_s\sqrt{\lambda_s}[M(\sigma)V^{-1}]_{sl_1} B_{sl_2} \\
        &= \frac{s(\nu_n)}{\sqrt{w^\top \cQ^\prime(\sigma) w}} \cbr{ V^{-1} M(\sigma) \Diag(w\odot\sqrt{\lambda}) B }_{l_1l_2}+\cO_{\prec}(n^{-1/10}).
    \end{align}
    This proves the claimed matrix formula for $\matop(\nu_n)^\top X/\sqrt n$.
\end{proof}

\begin{lemma}\label{lemma:second-order equivalent}
    For any $T\in \SA_L(\RR)$ and $\abr{\sigma_n-\sigma}=\cO_{\prec}(n^{-1/10})$, we have
    \begin{align}
        \Big\| & U^\top \rbr{ H_1 - \sigma_n I_{nL} }^{-1} (T \otimes I_n) \rbr{ H_1 - \sigma_n I_{nL} }^{-1} U  \\
        &\quad - \rbr{\sqrt{\Lambda}B\sqrt{\Lambda}} \odot \rbr{ \cP_{\sigma}[V^{-1} T V^{-1}] } \Big\| = \cO_{\prec}\rbr{ n^{-1/20} },
    \end{align}
    where $\cP_{\sigma}$ is defined in~\eqref{eq:P_sigma operator}.
\end{lemma}

\begin{proof}
    If we take $\Gamma=\kappa T$ in the generalized resolvent~\eqref{eq:resolvent def} with some fixed $T\in\SA_L(\RR)$ and $\kappa$ in a sufficiently small fixed neighborhood of zero, then
    \begin{align}
        &\quad U^\top \rbr{ H_1 - \sigma_n I_{nL} }^{-1} (T \otimes I_n) \rbr{ H_1 - \sigma_n I_{nL} }^{-1} U \\
        &= -\frac{\ud}{\ud \kappa} \cbr{ U^\top \rbr{ H_1 + \kappa \sbr{T \otimes I_n}  - \sigma_n I_{nL} }^{-1} U } \bigg|_{\kappa=0}. \label{eq:differentiating resolvent}
    \end{align}
    As in Lemma~\ref{lemma:convergence of UGU}, the quadratic form $U^\top \rbr{ H_1 + \kappa \sbr{T \otimes I_n}  - z I_{nL} }^{-1} U$ satisfies the deterministic estimate
    \begin{equation}
        \Big\| U^\top \rbr{ H_1 + \kappa \sbr{T \otimes I_n}  - z I_{nL} }^{-1} U
        + \sqrt{\Lambda} \sbr{B \odot M_{\kappa}(z)} \sqrt{\Lambda} \Big\|
        = \cO_{\prec}(n^{-1/10}), \qquad M_{\kappa}(z):=M(z;\kappa T),
    \end{equation}
    uniformly for $z$ in a fixed neighborhood of $\sigma$ separated from the corresponding perturbed bulk support and for $\kappa$ in the same neighborhood of zero. We apply this estimate at $z=\sigma_n$, using $\abr{\sigma_n-\sigma}=\cO_{\prec}(n^{-1/10})$.
    Here $M_{\kappa}(z)$ solves the following equation
    \begin{equation}
        M_{\kappa}(z)^{-1} = zB^{-1} - \kappa V^{-1} T V^{-1} + \Diag\rbr{\lambda} - \Diag\cbr{ \lambda \odot \diag(M_{\kappa}(z)) }.
    \end{equation}
    Differentiating with respect to $\kappa$ and restricting to $\kappa=0$, we derive that
    \begin{equation}
        M(\sigma)^{-1} \sbr{\frac{\ud M_{\kappa}(\sigma)}{\ud \kappa}\Bigg|_{\kappa=0}} M(\sigma)^{-1} = V^{-1} T V^{-1} + \Diag\cbr{\lambda\odot\diag\sbr{ \frac{\ud M_{\kappa}(\sigma)}{\ud \kappa}\Bigg|_{\kappa=0} }},
    \end{equation}
    which can be then transformed into
    \begin{align}
        \frac{\ud M_{\kappa}(\sigma)}{\ud \kappa}\Bigg|_{\kappa=0} &= M(\sigma) V^{-1} T V^{-1} M(\sigma)  \\
        &\qquad + M(\sigma) \Diag\cbr{\lambda\odot\diag\sbr{ \frac{\ud M_{\kappa}(\sigma)}{\ud \kappa}\Bigg|_{\kappa=0} }} M(\sigma).
    \end{align}
    Applying $\diag(\cdot)$ onto it, we have that
    \begin{equation}
        \rbr{I_L - M(\sigma)^{\odot 2}\Lambda}\diag\sbr{\frac{\ud M_{\kappa}(\sigma)}{\ud \kappa}\Bigg|_{\kappa=0}} = \diag\sbr{ M(\sigma) V^{-1} T V^{-1} M(\sigma) }.
    \end{equation}
    Therefore,
    \begin{equation}\label{eq:derivatve against kappa}
        \frac{\ud M_{\kappa}(\sigma)}{\ud \kappa}\Bigg|_{\kappa=0} = \cP_{\sigma}\sbr{ V^{-1} T V^{-1} }.
    \end{equation}
    To justify differentiating the deterministic equivalent, we use a difference quotient. Let
    \begin{align}
        \cF_n(\kappa) &:= U^\top \rbr{ H_1 + \kappa \sbr{T \otimes I_n}  - \sigma_n I_{nL} }^{-1} U, \\
        \cF(\kappa) &:= - \sqrt{\Lambda} \sbr{B \odot M_{\kappa}(\sigma_n)} \sqrt{\Lambda}.
    \end{align}
    Under the event $\cE$, after shrinking the fixed neighborhood of zero if necessary, the resolvent $\rbr{ H_1 + \kappa \sbr{T \otimes I_n}  - \sigma_n I_{nL} }^{-1}$ stays uniformly bounded for all $\kappa$ in that neighborhood. Hence, for $\kappa_n=n^{-1/20}$,
    \begin{align}
        \nbr{ \frac{\ud\cF_n}{\ud\kappa}(0) - \frac{\cF_n(\kappa_n)-\cF_n(0)}{\kappa_n} } \le C|\kappa_n|=Cn^{-1/20}.
    \end{align}
    The deterministic equivalent gives $\nbr{\cF_n(\kappa)-\cF(\kappa)}=\cO_{\prec}(n^{-1/10})$ for $\kappa=0,\kappa_n$. Hence
    \begin{equation}
        \nbr{ \frac{\cF_n(\kappa_n)-\cF_n(0)}{\kappa_n} - \frac{\cF(\kappa_n)-\cF(0)}{\kappa_n} } = \cO_{\prec}\rbr{ n^{-1/20} }.
    \end{equation}
    Moreover, since $\cF$ is continuously differentiable in this neighborhood of zero,
    \begin{equation}
        \nbr{ \frac{\cF(\kappa_n)-\cF(0)}{\kappa_n} - \frac{\ud\cF}{\ud\kappa}(0) } \le C|\kappa_n|=Cn^{-1/20}.
    \end{equation}
    Therefore,
    \begin{equation}
        \nbr{ \frac{\ud\cF_n}{\ud\kappa}(0) - \frac{\ud\cF}{\ud\kappa}(0) } = \cO_{\prec}\rbr{ n^{-1/20} }.
    \end{equation}
    From~\eqref{eq:derivatve against kappa} and the continuity in the spectral parameter, we find that
    \begin{equation}
        \frac{\ud\cF}{\ud\kappa}(0) = -\rbr{\sqrt{\Lambda}B\sqrt{\Lambda}} \odot \rbr{ \cP_{\sigma}[V^{-1} T V^{-1}] }+\cO_{\prec}(n^{-1/10}).
    \end{equation}
    Combining this identity with~\eqref{eq:differentiating resolvent} gives the claim.
\end{proof}

\begin{proof}[Proof of Proposition~\ref{prop:detailed-matrix-overlaps}\ref{prop:matrix-self-overlap}]
    By definition,
    \begin{equation}
        H_1 \nu_n + U U^\top\nu_n = \sigma_n \nu_n,
    \end{equation}
    which yields that
    \begin{equation}
        \nu_n = -\rbr{ H_1 - \sigma_n I_{nL} }^{-1} U U^\top\nu_n.
    \end{equation}
    Then
    \begin{align}
        \langle\nu_n^{(l_1)},\nu_n^{(l_2)}\rangle &= \nu_n^\top (e_{l_1} \otimes I_n) (e_{l_2}^\top \otimes I_n) \nu_n \\ 
        &= \nu_n^\top U U^\top \rbr{ H_1 - \sigma_n I_{nL} }^{-1} (e_{l_1} \otimes I_n) (e_{l_2}^\top \otimes I_n) \rbr{ H_1 - \sigma_n I_{nL} }^{-1} U U^\top\nu_n \\
        &= \nu_n^\top U U^\top \rbr{ H_1 - \sigma_n I_{nL} }^{-1} (T\otimes I_n) \rbr{ H_1 - \sigma_n I_{nL} }^{-1} U U^\top\nu_n \label{eq:sub-segement product representation}
    \end{align}
    where $T=\frac{1}{2}e_{l_1}e_{l_2}^\top+\frac{1}{2}e_{l_2}e_{l_1}^\top$; only the symmetric part contributes to the quadratic form. Lemma~\ref{lemma:second-order equivalent} gives
    \begin{equation}
        \Big\|U^\top \rbr{ H_1 - \sigma_n I_{nL} }^{-1} (T\otimes I_n) \rbr{ H_1 - \sigma_n I_{nL} }^{-1} U - \rbr{\sqrt{\Lambda}B\sqrt{\Lambda}} \odot \rbr{ \cP_{\sigma}[V^{-1} T V^{-1}] }\Big\|=\cO_{\prec}(n^{-1/20}).
    \end{equation}
    Moreover, Proposition~\ref{prop:detailed-matrix-overlaps}\ref{prop:overlap-with-U} gives
    \begin{equation}
        \nbr{U^\top \nu_n - s(\nu_n) \frac{w}{\sqrt{w^\top \cQ^\prime(\sigma) w}}}=\cO_{\prec}(n^{-1/10}).
    \end{equation}
    Substituting these two estimates into~\eqref{eq:sub-segement product representation} and telescoping the errors gives the claimed self-overlap formula. Together with Proposition~\ref{prop:detailed-matrix-overlaps}\ref{prop:matrix-overlap-with-X}, this proves Proposition~\ref{prop:sketch-matrix-overlaps}.
\end{proof}

\section{Variational analysis of the matrix Dyson equation at $z=1$}\label{sec:optimizing RS free energy}
Recall from Section~\ref{sec:formal present self-cosnistent eq} that, when the auxiliary perturbation $\Gamma$ is set to zero, the matrix Dyson equation takes the form
\begin{equation}\label{eq: self-consistent eq mat no Gamma}
    M^{-1} = zB^{-1} + \Lambda - \Diag\cbr{ \lambda \odot \diag(M) }, \quad M \in \SA_L(\CC),
\end{equation}
where $\Lambda=\Diag(\lambda)$. The solution for $z\in\CC^+$ is the deterministic object that controls the bulk law, but for the phase transition we need to understand its valid continuation to the special real point $z=1$.
\begin{proposition}\label{prop:special solution at 1}
    When $\SNR\rbr{\lambda,B}\neq1$, there exists some $\epsilon>0$ such that the valid solution $M(z)$ to~\eqref{eq: self-consistent eq mat no Gamma} can be continuously extended to $(1-\epsilon,+\infty)$ with $M(\sigma)\in\SA_{L}(\RR)$ for any $\sigma>1-\epsilon$. Specifically, if $\SNR\rbr{\lambda,B}<1$, we have $M(1)=B$; if $\SNR\rbr{\lambda,B}>1$, we have $0\prec M(1)\prec B$.
\end{proposition}
Together with Proposition~\ref{prop:analytical extension}, this shows that the interval $(1-\epsilon,+\infty)$ lies outside the limiting bulk spectrum. The proof of Proposition~\ref{prop:special solution at 1} has two steps. First, we rewrite the real MDE as the first-order condition for a finite-dimensional variational problem. Second, we analyze the local minimizers of this variational objective and show that the trivial minimizer changes stability exactly when $\SNR(\lambda,B)$ crosses one.

\subsection{Variational formulation}
The variational objective is
\begin{equation}\label{eq:RS energy}
    \cR\rbr{\chi,\sigma} = \log\cbr{\det\sbr{\sigma B^{-1} + \Diag(\lambda\odot(\chi+1))}} + \frac{1}{2}\sum_{l=1}^L \lambda_l\chi_l^2,
\end{equation}
defined on
\begin{equation}\label{eq:RS domain}
    \cD = \cbr{(\chi,\sigma)\in\RR^{L}\times\RR: \sigma B^{-1}+\Diag(\lambda\odot(\chi+1)) \succ 0}.
\end{equation}
This objective is motivated by the replica-symmetric prediction for the Bayesian inference problem~\eqref{eq:observation model}; see, for example,~\cite{montanari2024friendly}. More concretely, \eqref{eq:RS energy} is the finite-dimensional replica-symmetric free energy for the Gaussian-prior version of the model, where $\vec(X^{(1)},\ldots,X^{(L)})\sim\cN(0,B)$.

For $(\chi,\sigma)\in\cD$, define
\begin{equation}\label{eq:M-from-chi}
    M_{\chi}(\sigma):=\sbr{\sigma B^{-1}+\Diag\cbr{\lambda\odot\rbr{\chi+1}}}^{-1}.
\end{equation}
The point of this change of variables is that the stationarity condition for $\cR$ is equivalent to the MDE. Indeed, $\nabla_\chi\cR(\chi,\sigma)=0$ gives $\chi=-\diag(M_{\chi}(\sigma))$, and substituting this identity into~\eqref{eq:M-from-chi} gives~\eqref{eq: self-consistent eq mat no Gamma}.

\begin{lemma}\label{lemma:derivatives of cR}
    If $\sigma>0$, the functional $\chi\mapsto\cR(\chi,\sigma)$ is analytic on $[-1,+\infty)^L$. Its first and second derivatives are
    \begin{align}
        \frac{\partial\cR}{\partial \chi}&=\lambda\odot\sbr{ \chi+\diag\cbr{\sbr{\sigma B^{-1}+\Diag(\lambda\odot(\chi+1))}^{-1}} }, \label{eq:gradient of cR}\\
        \frac{\partial^2\cR}{\rbr{\partial \chi}^2}&=\Diag(\lambda)-\Diag\rbr{\lambda}\cbr{\sbr{\sigma B^{-1}+\Diag(\lambda\odot(\chi+1))}^{-1}}^{\odot 2}\Diag(\lambda). \label{eq:hessian of cR}
    \end{align}
\end{lemma}

\begin{proof}
    For any differentiable family of invertible matrices $C(x)$,
    \begin{equation}
        \frac{\ud}{\ud x}\log\cbr{\det\sbr{C(x)}}=\tr\sbr{C(x)^{-1}\frac{\ud C(x)}{\ud x}}.
    \end{equation}
    Applying this identity with $C(\chi)=\sigma B^{-1}+\Diag(\lambda\odot(\chi+1))$ gives
    \begin{align}
        \frac{\partial\cR}{\partial \chi_l}
        &=\tr\cbr{C(\chi)^{-1}\lambda_le_le_l^\top}+\lambda_l\chi_l
        =\lambda_le_l^\top C(\chi)^{-1}e_l+\lambda_l\chi_l.
    \end{align}
    Differentiating once more gives
    \begin{align}
        \frac{\partial^2\cR}{\partial \chi_l\partial \chi_k}
        &=-\lambda_l\lambda_k e_l^\top C(\chi)^{-1}e_ke_k^\top C(\chi)^{-1}e_l+\lambda_l\mathbf{1}_{l=k}\\
        &=-\lambda_l\lambda_k\cbr{e_l^\top C(\chi)^{-1}e_k}^2+\lambda_l\mathbf{1}_{l=k},
    \end{align}
    which is exactly~\eqref{eq:hessian of cR}.
\end{proof}

The next lemma explains why strict local minimizers of $\cR$ produce the valid real branch of the MDE.
\begin{lemma}\label{lemma:minimizer-to-valid-mde-branch}
    Suppose that for every $\sigma\ge1$ there exists $\chi^\ast(\sigma)\in[-1,0)^L$ such that
    \begin{equation}\label{eq:first and strict second order condition}
        \frac{\partial\cR}{\partial \chi}(\chi^{\ast}(\sigma),\sigma) = 0, \quad \frac{\partial^2\cR}{\rbr{\partial \chi}^2}(\chi^{\ast}(\sigma),\sigma) \succ 0.
    \end{equation}
    Define
    \begin{equation}\label{eq:M-o-definition}
        M^{o}(\sigma) := \sbr{\sigma B^{-1}+\Diag\cbr{\lambda\odot\rbr{\chi^\ast(\sigma)+1}}}^{-1}.
    \end{equation}
    Then the solution $M(z)$ of~\eqref{eq: self-consistent eq mat no Gamma} on $\CC^+$ admits a real analytic continuation to a neighborhood of every point in $[1,+\infty)$, and this continuation agrees with $M^o(\sigma)$ on $[1,+\infty)$. Consequently, $M(z)$ extends to $(1-\epsilon,+\infty)$ for some $\epsilon>0$, and $M(\sigma)\in\SA_L(\RR)$ on this interval.
\end{lemma}

\begin{proof}
    The stationarity condition and Lemma~\ref{lemma:derivatives of cR} imply that
    \begin{equation}\label{eq:M-o-solves-real-MDE}
        \sbr{M^{o}(\sigma)}^{-1} = \sigma B^{-1} + \Lambda - \Diag\cbr{ \lambda \odot \diag\sbr{M^{o}(\sigma)} }.
    \end{equation}
    We next identify this real solution with the valid continuation from $\CC^+$. Fix $\sigma\in[1,+\infty)$, and let
    \begin{equation}
        T(m,z):=-m^{-1}+zB^{-1}+\Lambda-\Diag\cbr{\lambda\odot\diag(m)}.
    \end{equation}
    Then $T(M^o(\sigma),\sigma)=0$. The derivative with respect to the first argument is the linear operator
    \begin{equation}
        \nabla^{(1)}_{R}T(m,z)=m^{-1}Rm^{-1}-\Diag\cbr{\lambda\odot\diag(R)},\qquad R\in\SA_L(\CC).
    \end{equation}
    We claim that $\nabla^{(1)}T(M^o(\sigma),\sigma)$ is invertible. Given $\Delta\in\SA_L(\CC)$, solving
    \begin{equation}
        \Delta=\sbr{M^o(\sigma)}^{-1}R\sbr{M^o(\sigma)}^{-1}-\Diag\cbr{\lambda\odot\diag(R)}
    \end{equation}
    reduces, after multiplying by $M^o(\sigma)$ on both sides and taking diagonals, to
    \begin{equation}
        \sbr{I-\rbr{M^o(\sigma)\odot M^o(\sigma)}\Lambda}\diag(R)
        =\diag\sbr{M^o(\sigma)\Delta M^o(\sigma)}.
    \end{equation}
    By~\eqref{eq:hessian of cR} and the strict Hessian condition in~\eqref{eq:first and strict second order condition},
    \begin{equation}
        \Lambda-\Lambda\rbr{M^o(\sigma)\odot M^o(\sigma)}\Lambda\succ0.
    \end{equation}
    Equivalently, by conjugating with $\Lambda^{-1/2}$, the matrix
    $\Lambda^{1/2}\rbr{M^o(\sigma)\odot M^o(\sigma)}\Lambda^{1/2}$ has spectral
    radius strictly smaller than one. Since this matrix is similar to
    $\rbr{M^o(\sigma)\odot M^o(\sigma)}\Lambda$, we have
    \[
        \rho\rbr{\rbr{M^o(\sigma)\odot M^o(\sigma)}\Lambda}<1.
    \]
    Hence
    $I-\rbr{M^o(\sigma)\odot M^o(\sigma)}\Lambda$ is invertible, so the
    diagonal of $R$ is uniquely determined. The whole matrix is then recovered from
    \begin{equation}
        R=M^o(\sigma)\cbr{\Delta+\Diag\cbr{\lambda\odot\diag(R)}}M^o(\sigma).
    \end{equation}
    This inverse is positivity preserving. Indeed, if $\Delta\succ0$, then $\diag\sbr{M^o(\sigma)\Delta M^o(\sigma)}$ has strictly positive entries. Moreover, the spectral-radius bound above gives the convergent Neumann series
    \[
        \sbr{I-\rbr{M^o(\sigma)\odot M^o(\sigma)}\Lambda}^{-1}
        =
        \sum_{k\ge0}\sbr{\rbr{M^o(\sigma)\odot M^o(\sigma)}\Lambda}^{k},
    \]
    whose coefficients are nonnegative.

    Lemma~\ref{lemma:implicit function} therefore gives an analytic solution $M^o(z)$ to the MDE in a complex neighborhood of $\sigma$. Moreover,
    \begin{equation}
        \frac{\ud M^o}{\ud z}(\sigma)
        =-\cbr{\nabla^{(1)}T(M^o(\sigma),\sigma)}^{-1}B^{-1}\prec0,
    \end{equation}
    where the negativity follows from the positivity-preserving form of the inverse displayed above. Thus, for $\eta>0$ sufficiently small, $\Im M^o(\sigma+i\eta)\prec0$. By uniqueness of the MDE solution in Proposition~\ref{prop:feasibility}, this local branch agrees with $M(z)$ on its intersection with $\CC^+$.

    Since $\sigma\ge1$ was arbitrary, the solution from $\CC^+$ extends analytically to an open set containing $[1,+\infty)$ and agrees there with $M^o$ on the real line. Taking a small interval below $1$ inside this open set and using Schwarz reflection across the real axis gives $\Im M(\sigma)=0$ for all $\sigma\in(1-\epsilon,1)$. This proves the claimed real continuation.
\end{proof}

\subsection{Local minimizers of the variational objective}
It remains to prove the existence of the strict local minimizers required in Lemma~\ref{lemma:minimizer-to-valid-mde-branch}. The result we need is the following.
\begin{lemma}\label{lemma:existence of local minimizers}
    Assume that $\SNR(\lambda,B)\neq1$.
    \begin{enumerate}[label=(\roman*)]
        \item For every $\sigma>1$, the map $\chi\mapsto\cR(\chi,\sigma)$ has a global minimizer $\chi^\ast(\sigma)\in(-1,0)^L$ on $[-1,+\infty)^L$ satisfying~\eqref{eq:first and strict second order condition}.
        \item If $\SNR(\lambda,B)<1$, then $\chi^\ast(1)=-\mathbf{1}_L$ satisfies~\eqref{eq:first and strict second order condition}.
        \item If $\SNR(\lambda,B)>1$, then $\chi\mapsto\cR(\chi,1)$ has a global minimizer $\chi^\ast(1)\in(-1,0)^L$ on $[-1,+\infty)^L$ satisfying~\eqref{eq:first and strict second order condition}.
    \end{enumerate}
\end{lemma}

We prove the three cases separately, and then verify the strict Hessian condition for the interior minimizers.
\begin{lemma}\label{lemma: minimizer z>1}
    For any $\sigma>1$, the map $\chi\mapsto\cR(\chi,\sigma)$ has a global minimizer in $(-1,0)^L$ over the domain $[-1,+\infty)^L$.
\end{lemma}

\begin{proof}
    It suffices to minimize over $[-1,0]^L$ and check that the minimizer cannot lie on the boundary. If $\chi_l\ge0$ for some $l\in[L]$, then
    \begin{align}
        \frac{\partial\cR}{\partial \chi_l}(\chi,\sigma)
        &=\lambda_le_l^\top\sbr{\sigma B^{-1}+\Diag(\lambda\odot(\chi+1))}^{-1}e_l+\lambda_l\chi_l\\
        &>0.
    \end{align}
    If $\chi_l=-1$ for some $l\in[L]$, then, using $B_{ll}=1$ from
    Assumption~\ref{assump:irreducible-normalized-B},
    \begin{align}
        \frac{\partial\cR}{\partial \chi_l}(\chi,\sigma)
        &\le \lambda_l e_l^\top B e_l/\sigma-\lambda_l<0.
    \end{align}
    Therefore every boundary point can be improved by moving into $(-1,0)^L$, and the global minimizer over the compact set $[-1,0]^L$ lies in $(-1,0)^L$.
\end{proof}

\begin{lemma}\label{lemma: minimizer z=1, uninformative}
    If
    \begin{equation}\label{eq:uninformative}
        \SNR\rbr{\lambda,B}=\sigma_{\max}\rbr{\Diag\rbr{\sqrt{\lambda}}\rbr{B\odot B}\Diag\rbr{\sqrt{\lambda}}}<1,
    \end{equation}
    then $\chi^\ast(1)=-\mathbf{1}_L$ is a strict local minimizer of $\cR(\cdot,1)$ and satisfies~\eqref{eq:first and strict second order condition}.
\end{lemma}

\begin{proof}
    By~\eqref{eq:gradient of cR}, the identity $\partial\cR/\partial\chi(-\mathbf{1}_L,1)=0$ holds for all parameters. At this point,
    \begin{equation}
        \frac{\partial^2\cR}{\rbr{\partial \chi}^2}(-\mathbf{1}_L,1)
        =\Lambda-\Lambda\rbr{B\odot B}\Lambda.
    \end{equation}
    The condition~\eqref{eq:uninformative} is equivalent to this Hessian being positive definite.
\end{proof}

When $\SNR(\lambda,B)>1$, the trivial stationary point $-\mathbf{1}_L$ becomes unstable, and an interior minimizer appears.
\begin{lemma}\label{lemma: minimizer in the informative regime}
    If
    \begin{equation}\label{eq:informative}
        \SNR\rbr{\lambda,B}=\sigma_{\max}\rbr{\Diag\rbr{\sqrt{\lambda}}\rbr{B\odot B}\Diag\rbr{\sqrt{\lambda}}}>1,
    \end{equation}
    then $\cR(\cdot,1)$ has a global minimizer $\chi^\ast\in(-1,0)^L$ over $[-1,+\infty)^L$.
\end{lemma}

\begin{proof}
    If $\chi_l\ge0$ for some $l\in[L]$, then
    \begin{align}
        \frac{\partial\cR}{\partial \chi_l}(\chi,1)
        &=\lambda_le_l^\top\sbr{B^{-1}+\Diag(\lambda\odot(\chi+1))}^{-1}e_l+\lambda_l\chi_l>0.
    \end{align}
    Hence the infimum over $[-1,+\infty)^L$ is attained on the compact set $[-1,0]^L$. Moreover, the same derivative check at $\chi_l=0$ shows that any global minimizer satisfies $\chi_l^\ast<0$ for every $l$.

    We next rule out the boundary $\chi_l^\ast=-1$. Since~\eqref{eq:informative} holds, Perron-Frobenius applied to the irreducible nonnegative matrix $\Diag(\sqrt\lambda)(B\odot B)\Diag(\sqrt\lambda)$ gives a nonnegative direction along which the Hessian at $-\mathbf{1}_L$ is negative. Thus $-\mathbf{1}_L$ is not a local minimizer.

    Let $\cA=\cbr{l\in[L]:\chi_l^\ast=-1}$ be the inactive set. Suppose that $\cA$ is nonempty and not all of $[L]$. Choose $l\in\cA$ and $k\in\cA^c$. With $D=\Diag\rbr{\lambda\odot(\chi^\ast+1)}$, we have $D\succeq \mu e_ke_k^\top$ for some $\mu>0$. Since a one-sided derivative at the boundary must be nonnegative,
    \begin{align}
        0\le\frac{\partial\cR}{\partial \chi_l}(\chi^\ast,1)
        &=\lambda_l\cbr{\rbr{B^{-1}+D}^{-1}}_{ll}-\lambda_l\\
        &\le\lambda_l\cbr{\rbr{B^{-1}+\mu e_ke_k^\top}^{-1}}_{ll}-\lambda_l\\
        &=\lambda_l\cbr{B-\frac{Be_ke_k^\top B}{1/\mu+B_{kk}}}_{ll}-\lambda_l\\
        &=-\frac{\lambda_lB_{kl}^2}{1/\mu+1}\le0,
    \end{align}
    where we again used $B_{kk}=B_{ll}=1$. Therefore $B_{kl}=0$ for every $l\in\cA$ and $k\in\cA^c$, contradicting the irreducibility of $B$ in Assumption~\ref{assump:irreducible-normalized-B}. Since $\cA=[L]$ has already been ruled out, we must have $\cA=\emptyset$.
\end{proof}

\begin{lemma}\label{lemma:strict-hessian-interior}
    Let $\sigma\ge1$ and let $\chi_\ast\in(-1,0)^L$ be a local minimizer of $\cR(\cdot,\sigma)$. If
    \begin{equation}
        \frac{\partial\cR}{\partial\chi}(\chi_\ast,\sigma)=0,
    \end{equation}
    then
    \begin{equation}
        \frac{\partial^2\cR}{\rbr{\partial \chi}^2}(\chi_\ast,\sigma)\succ0.
    \end{equation}
\end{lemma}

\begin{proof}
    Since $\chi_\ast$ is an interior local minimizer, the Hessian is positive semidefinite. Suppose, for contradiction, that it is singular. Let
    \begin{equation}
        M_\ast=\sbr{\sigma B^{-1}+\Diag\cbr{\lambda\odot\rbr{\chi_\ast+1}}}^{-1}.
    \end{equation}
    Stationarity and Lemma~\ref{lemma:derivatives of cR} give
    \begin{align}
        \sigma B^{-1} &= M_\ast^{-1}-\Lambda+\Diag\cbr{\lambda\odot\diag(M_\ast)}, \label{eq:condition at sigma star}\\
        0 &= \sigma_{\min}\sbr{\Lambda-\Lambda(M_\ast\odot M_\ast)\Lambda}. \label{eq:condition at sigma star 2}
    \end{align}
    Since $\sigma\ge1$ and $\chi_\ast\in(-1,0)^L$, we have $M_\ast\prec B$.

    From~\eqref{eq:condition at sigma star 2}, the matrix $\Lambda^{1/2}(M_\ast\odot M_\ast)\Lambda^{1/2}$ has largest eigenvalue one. Equivalently, the nonnegative matrix $\Lambda(M_\ast\odot M_\ast)$ has Perron-Frobenius eigenvalue one. Choose a nonzero vector $v\in\RR_+^L$ such that
    \begin{equation}
        \Lambda(M_\ast\odot M_\ast)v=v,
    \end{equation}
    and set $W=\Diag(v)$. For the operators $\cS[R]=\Diag\cbr{\lambda\odot\diag(R)}$ and $\cC_{M_\ast}[R]=M_\ast R M_\ast$, this choice gives
    \begin{equation}
        \cS\cC_{M_\ast}[W]=W.
    \end{equation}
    Taking the inner product of~\eqref{eq:condition at sigma star} with $M_\ast W M_\ast$ yields
    \begin{align}
        \sigma\inner{B^{-1}}{M_\ast W M_\ast}
        &= \inner{W}{M_\ast}+\inner{M_\ast W M_\ast}{\cS[M_\ast-B]}\\
        &= \inner{W}{M_\ast}+\inner{\cS\cC_{M_\ast}[W]}{M_\ast-B}\\
        &= \inner{W}{2M_\ast-B}.
    \end{align}
    On the other hand,
    \begin{align}
        \inner{B^{-1}}{M_\ast W M_\ast}
        &=\inner{W}{M_\ast B^{-1}M_\ast},\\
        B+M_\ast B^{-1}M_\ast-2M_\ast
        &=\sqrt{M_\ast}\sbr{M_\ast^{-1/2}BM_\ast^{-1/2}+M_\ast^{1/2}B^{-1}M_\ast^{1/2}-2I_L}\sqrt{M_\ast}\succ0,
    \end{align}
    where the last inequality follows from $M_\ast\prec B$. Since $W\succeq0$ and $W\neq0$, this implies
    \begin{equation}
        \inner{W}{2M_\ast-B}<\inner{B^{-1}}{M_\ast W M_\ast}.
    \end{equation}
    Therefore
    \begin{equation}
        \sigma=\frac{\inner{W}{2M_\ast-B}}{\inner{B^{-1}}{M_\ast W M_\ast}}<1,
    \end{equation}
    contradicting $\sigma\ge1$.
\end{proof}

\begin{proof}[Proof of Lemma~\ref{lemma:existence of local minimizers}]
    For $\sigma>1$, Lemma~\ref{lemma: minimizer z>1} gives an interior global minimizer, and Lemma~\ref{lemma:strict-hessian-interior} gives the strict Hessian condition. If $\sigma=1$ and $\SNR(\lambda,B)<1$, Lemma~\ref{lemma: minimizer z=1, uninformative} gives the required minimizer directly. If $\sigma=1$ and $\SNR(\lambda,B)>1$, Lemma~\ref{lemma: minimizer in the informative regime} gives an interior global minimizer, and Lemma~\ref{lemma:strict-hessian-interior} again gives the strict Hessian condition.
\end{proof}

\subsection{Proof of Proposition~\ref{prop:special solution at 1}}

\begin{proof}
    By Lemma~\ref{lemma:existence of local minimizers}, the hypotheses of Lemma~\ref{lemma:minimizer-to-valid-mde-branch} hold whenever $\SNR(\lambda,B)\neq1$. Thus the valid solution of~\eqref{eq: self-consistent eq mat no Gamma} extends to $(1-\epsilon,+\infty)$ and is real on this interval.

    It remains only to identify the value at $\sigma=1$. If $\SNR(\lambda,B)<1$, Lemma~\ref{lemma: minimizer z=1, uninformative} gives $\chi^\ast(1)=-\mathbf{1}_L$, and therefore~\eqref{eq:M-o-definition} gives $M(1)=B$. If $\SNR(\lambda,B)>1$, Lemma~\ref{lemma: minimizer in the informative regime} gives $\chi^\ast(1)\in(-1,0)^L$. Hence
    \begin{equation}
        M(1)=\sbr{B^{-1}+\Diag\cbr{\lambda\odot\rbr{\chi^\ast(1)+1}}}^{-1},
    \end{equation}
    where the diagonal perturbation is strictly positive definite. Consequently $0\prec M(1)\prec B$.
\end{proof}

As a byproduct, Proposition~\ref{prop:special solution at 1} also implies that the local minimizers selected above are unique on the valid branch.

\section{Emergence of outlier eigenvalues}\label{sec:emergence of outlier 1}

This section proves the phase-transition statements in Theorems~\ref{thm: main uninformative} and~\ref{thm: main informative}, as well as the outlier-counting statement in Proposition~\ref{prop:main-outlier-count}. The inputs have been established in the preceding sections. Section~\ref{sec:outlier} reduces outlier eigenvalues to the finite-dimensional equation $\det\cQ(\sigma)=0$ and gives the corresponding eigenvector-overlap formulas. Section~\ref{sec:optimizing RS free energy} analyzes the valid solution of the matrix Dyson equation at $z=1$ and shows that it changes behavior exactly when $\SNR(\lambda,B)$, defined in~\eqref{eq:summary SNR def}, crosses one. We now combine these inputs.

Throughout this section, let
\[
    \sigma_{+}:=\sup\supp(\mu)
\]
be the right edge of the limiting bulk spectrum. Whenever it is used below,
$M_+$ denotes the one-sided limit of the valid real solution $M(\sigma)$ as
$\sigma\downarrow\sigma_+$; the existence of this limit follows from the
monotonicity argument in the proof of Proposition~\ref{prop:main-outlier-count}.


We first prove Proposition~\ref{prop:main-outlier-count}. This argument counts the deterministic roots of the outlier equation from Section~\ref{sec:outlier}, and then uses Proposition~\ref{prop:sketch-statement-outlier} to transfer the count to the random eigenvalues. The subcritical simplification is included in the same proof. Above the threshold, the special point $\sigma=1$ becomes a root of the outlier equation. We then identify the corresponding null vector and simplify the overlap formulas before stating the final local proposition used in Theorem~\ref{thm: main informative}.

\subsection{Counting outliers to the right}
\label{subsec:counting-outliers-right}

\begin{proof}[Proof of Proposition~\ref{prop:main-outlier-count}]
    For any $\sigma>\sigma_{+}$, the self-consistent equation~\eqref{eq: self-consistent eq matrix} is solvable at $z=\sigma,\Gamma=0$. In other words, there exists $M:(\sigma_{+},\infty)\to\SA_L^{+}(\RR)$ such that
    \begin{equation}\label{eq: self eq above edge}
        M(\sigma)^{-1} = \sigma B^{-1} + \Diag\rbr{\lambda} - \Diag\cbr{ \lambda \odot \diag(M(\sigma)) },
    \end{equation}
    and $M(\sigma)$ is strictly decreasing in $\sigma$.
    The monotonicity and boundedness of $M(\sigma)$ imply that the one-sided limit
    \begin{equation}
        M(\sigma_{+}):=\lim_{\sigma\downarrow\sigma_{+}} M(\sigma)
    \end{equation}
    exists; we write it as $M_+$. We therefore extend
    $\cQ(\sigma)=I_L-\sqrt{\Lambda}\sbr{B\odot M(\sigma)}\sqrt{\Lambda}$
    continuously to $\sigma=\sigma_+$ by using this boundary value. Letting
    $\sigma\to+\infty$ in~\eqref{eq: self eq above edge}, we also have
    \begin{equation}
        \lim_{\sigma\to+\infty} M(\sigma) = 0.
    \end{equation}
    By Proposition~\ref{prop:sketch-statement-outlier}, outlier eigenvalues outside the bulk are determined by
    \[
        \cQ(\sigma) = I_L - \sqrt{\Lambda} \sbr{B \odot M(\sigma)} \sqrt{\Lambda}.
    \]
    Let $f_1,\ldots,f_L:[\sigma_+,\infty)\to\RR$ be the eigenvalues of this
    continuous extension of $\cQ(\sigma)$, counted continuously with multiplicity.
    Up to multiplicity, the number of real roots of $\det\cQ(\sigma)=0$ on
    $(\sigma_+,\infty)$ is
    \begin{align}
        \abr{\cbr{\sigma>\sigma_+: \det \cQ(\sigma)=0}} &= \sum_{j\in[L]} \abr{\cbr{\sigma>\sigma_+: f_j(\sigma)=0}} \\
        &= \sum_{j\in[L]} \mathbf{1}\cbr{ f_j(\sigma_+)<0 } = L_0.
    \end{align}
    Indeed, $\cQ(\sigma)$ is strictly increasing in $\sigma$, and
    $\lim_{\sigma\to\infty}\cQ(\sigma)=I_L$. Thus each eigenvalue branch can
    cross zero at most once, and it crosses precisely when its boundary value at
    $\sigma_+$ is negative. Since there are only finitely many such roots, we may
    choose $\epsilon_0>0$ so that the roots in $(\sigma_+,\infty)$ all lie to the
    right of $\sigma_++2\epsilon_0$. Proposition~\ref{prop:sketch-statement-outlier}
    then converts this deterministic root count into the asserted high-probability
    count of outlier eigenvalues.

    When $\SNR\rbr{\lambda,B}<1$, Proposition~\ref{prop:special solution at 1} gives $\sigma_{+}<1$ and $M(1)=B$. We claim that
    \begin{equation}\label{eq:subcritical-Q-positive}
        \sqrt{\Lambda}\sbr{B\odot M(\sigma)}\sqrt{\Lambda}\prec I_L,
        \qquad \sigma>\sigma_+ .
    \end{equation}
    To prove this, first note that for every $\sigma>\sigma_+$ the matrix
    \begin{equation}
        I_L-\rbr{M(\sigma)\odot M(\sigma)}\Lambda
    \end{equation}
    is invertible. Indeed, if it were singular, then for some nonzero vector $d$ the matrix
    \[
        A=M(\sigma)\Diag(\lambda\odot d)M(\sigma)
    \]
    would satisfy $\diag(A)=d$ and hence
    \[
        A-M(\sigma)\Diag\cbr{\lambda\odot\diag(A)}M(\sigma)=0,
    \]
    contradicting the invertibility of the stability operator away from $\supp(\mu)$ in Lemma~\ref{lemma:bounds on cL}, after taking the limit from $\sigma+i\eta$ as $\eta\downarrow0$. Since
    \[
        \lambda_{\max}\rbr{\sqrt{\Lambda}\sbr{B\odot B}\sqrt{\Lambda}}
        =\SNR(\lambda,B)<1
    \]
    at $\sigma=1$, continuity in $\sigma$ gives
    \begin{equation}\label{eq:M-square-rho-bound}
        \lambda_{\max}\rbr{\sqrt{\Lambda}\sbr{M(\sigma)\odot M(\sigma)}\sqrt{\Lambda}}<1,
        \qquad \sigma>\sigma_+ .
    \end{equation}
    We now compare the mixed Schur product with the two square Schur products. Let
    \[
        A_B:=\sqrt{\Lambda}\sbr{B\odot B}\sqrt{\Lambda},\qquad
        A_M:=\sqrt{\Lambda}\sbr{M(\sigma)\odot M(\sigma)}\sqrt{\Lambda}.
    \]
    Both are entrywise nonnegative. For nonnegative matrices $A_B,A_M$, the Collatz--Wielandt formula and Cauchy's inequality imply
    \[
        \rho\!\left(\left|\sqrt{\Lambda}\sbr{B\odot M(\sigma)}\sqrt{\Lambda}\right|\right)
        \le
        \sqrt{\rho(A_B)\rho(A_M)} .
    \]
    Since the spectral radius of a matrix is bounded by the spectral radius of its entrywise absolute value, combining this bound with $\rho(A_B)=\SNR(\lambda,B)<1$ and~\eqref{eq:M-square-rho-bound} proves~\eqref{eq:subcritical-Q-positive}. Letting $\sigma\downarrow\sigma_+$ gives
    \[
        \lambda_{\max}\rbr{\sqrt{\Lambda}\sbr{B\odot M_+}\sqrt{\Lambda}}
        \le \sqrt{\SNR(\lambda,B)}<1.
    \]
    Thus no eigenvalue of $\sqrt{\Lambda}(B\odot M_+)\sqrt{\Lambda}$ is larger than one, and $L_0=0$.

    It remains to record the additional conclusion in the informative regime. If $\SNR(\lambda,B)>1$, Proposition~\ref{prop:special solution at 1} gives $\sigma_+<1$, while Lemma~\ref{lemma:null-vector-at-1} below shows that $\cQ(1)$ has a nontrivial kernel. Thus $1$ is a root of $\det\cQ(\sigma)=0$ in $(\sigma_+,\infty)$, and the preceding root count implies $L_0\ge1$. The same high-probability count gives exactly $L_0$ outliers above $\sigma_++\epsilon$, while the bulk law and the no-outside-spectrum estimate from Section~\ref{sec:results on noise mat} give the matching localization of the remaining eigenvalues at the right edge. Hence $\sigma_{L_0+1}(H)\pto\sigma_+$.
\end{proof}

\subsection{Top outlier eigenvalue at $1$}

\begin{lemma}\label{lemma:null-vector-at-1}
    When $\SNR\rbr{\lambda,B}>1$, we always have $ \cQ(1)=I_L - \sqrt{\Lambda} \sbr{B \odot M(1)} \sqrt{\Lambda}\succeq 0$ and $\cQ(1) w =0$ for
    \begin{equation}
        w = \sqrt{\lambda} \odot \rbr{\mathbf{1}_L -\diag\sbr{M(1)}} \in \RR^L.
    \end{equation}
    Moreover, thanks to the irreducibility of $B$ in
    Assumption~\ref{assump:irreducible-normalized-B}, we also have
    $\dim\sbr{\mathrm{ker}\cQ(1)}=1$. Finally, $\cQ(\sigma)\succ0$ for every
    $\sigma>1$.
\end{lemma}

\begin{proof}
    We take several steps to build this lemma.

    \vspace*{2mm}
    \noindent\emph{Step 1. Constructively finding a null vector.}
    As established in Proposition~\ref{prop:special solution at 1}, the solution $M(1)$ must satisfy that $B-M(1)\succ0$. Therefore, $w = \sqrt{\lambda} \odot\diag(B-M(1))$ is non-zero. By definition, $M(1)$ satisfies
    \begin{equation}\label{eq:temp}
        M(1)^{-1} = B^{-1} + \Diag\rbr{\lambda} - \Diag\cbr{ \lambda \odot \diag(M(1)) }.
    \end{equation}
    Multiplying $B$ from the left and $M(1)$ from the right, we conclude that
    \begin{equation}\label{eq:eigenvector eq for the asymptotic resolvent}
        M(1) - B + M(1) \begin{bmatrix}
                \lambda_1(1-M(1)_{11}) & 0 & \cdots & 0 \\
                0 & \lambda_2(1-M(1)_{22}) & \cdots & 0 \\
                & & \cdots & \\
                0 & 0 & \cdots & \lambda_L(1-M(1)_{LL})
            \end{bmatrix} B = 0.
    \end{equation}
    Recall the identity that $\diag[A_1\Diag(b)A_2]=(A_1\odot A_2)b$ for any $A_1,A_2\in\RR^{L},b\in\RR^L$. Now applying $\diag(\cdot)$ to~\eqref{eq:eigenvector eq for the asymptotic resolvent}, we obtain
    \begin{equation}
        -\begin{bmatrix}
            1-M(1)_{11} \\
            1-M(1)_{22} \\
            \cdots \\
            1-M(1)_{LL}
        \end{bmatrix} + (B\odot M(1))\begin{bmatrix}
            \lambda_1(1-M(1)_{11}) \\
            \lambda_2(1-M(1)_{22}) \\
            \cdots \\
            \lambda_L(1-M(1)_{LL})
        \end{bmatrix}=0.
    \end{equation}
    Then multiply $\sqrt{\Lambda}$ from the left to deduce that $\cQ(1) w =0$.

    \vspace*{2mm}
    \noindent\emph{Step 2. Proving $\cQ(1)$ to be positive semi-definite.} The key idea is to transform both $B$ and $M(1)$ so that they become commutative. Let
    \begin{align}
        \Psi_1 &:= \Diag\sbr{\sqrt{\lambda}\odot\sqrt{\mathbf{1}_L-\diag(M(1))}} B \Diag\sbr{\sqrt{\lambda}\odot\sqrt{\mathbf{1}_L-\diag(M(1))}}, \\
        \Psi_2 &:= \Diag\sbr{\sqrt{\lambda}\odot\sqrt{\mathbf{1}_L-\diag(M(1))}} M(1) \Diag\sbr{\sqrt{\lambda}\odot\sqrt{\mathbf{1}_L-\diag(M(1))}}.
    \end{align}
    Then~\eqref{eq:temp} yields that
    \begin{equation}
        \Psi_2^{-1}=\Psi_1^{-1} +I_L,
    \end{equation}
    which further implies that $\Psi_1\Psi_2=\Psi_1-\Psi_2=\Psi_2\Psi_1$. Henceforth, $\Psi_1$ and $\Psi_2$ are commutable, thus simultaneously diagonalizable: there exists orthogonal $Q$ and scalars $\psi_1,\ldots,\psi_L,\xi_1,\ldots,\xi_L>0$ such that
    \begin{align}
        \Psi_1 &= Q \Diag\rbr{\psi_1,\ldots,\psi_L} Q^{\top}, \\
        \Psi_2 &= Q \Diag\rbr{\xi_1,\ldots,\xi_L} Q^{\top}.
    \end{align}
    Then $\Psi_2^{-1}=\Psi_1^{-1} +I_L$ gives
    \begin{equation}
        \xi_l^{-1}=\psi_l^{-1}+1,
        \qquad\text{or equivalently}\qquad
        \xi_l=\frac{\psi_l}{1+\psi_l},
        \quad l=1,\ldots,L.
    \end{equation}
    Subsequently, we claim that $\Psi_1\odot\Psi_2\preceq\Diag(\diag(\Psi_1\Psi_2))$. It can be shown in the following way,
    \begin{align}
        \Psi_1 \odot \Psi_2 &= \sum_{k,l=1}^L \psi_k\xi_l \rbr{q_kq_k^{\top}} \odot \rbr{q_lq_l^{\top}} \\
        &= \sum_{k,l=1}^L \psi_k\xi_l \rbr{q_k \odot q_l}\rbr{q_k \odot q_l}^{\top}.
    \end{align}
    At the same time,
    \begin{align}
        \Diag(\diag(\Psi_1\Psi_2)) &= \sum_{k=1}^L \psi_k\xi_k \Diag\rbr{ \diag\rbr{ q_kq_k^{\top} }} \\
        &= \sum_{k=1}^L \psi_k\xi_k \Diag\rbr{ q_k} \sbr{\sum_{l=1}^L q_lq_l^{\top}} \Diag\rbr{q_k} \\
        &= \sum_{k,l=1}^L \psi_k\xi_k \rbr{q_k \odot q_l}\rbr{q_k \odot q_l}^{\top}.
    \end{align}
    Consequently, by subtracting the two equations,
    \begin{align}
        F :&=\Diag(\diag(\Psi_1\Psi_2))-\Psi_1 \odot \Psi_2 \\
        &= \sum_{1\le l<k\le L} \rbr{\psi_k\xi_k+\psi_l\xi_l-\psi_k\xi_l-\psi_l\xi_k} \rbr{q_k \odot q_l}\rbr{q_k \odot q_l}^{\top} \\
        &= \sum_{1\le l<k\le L} \rbr{\psi_k-\psi_l}\rbr{\xi_k-\xi_l} \rbr{q_k \odot q_l}\rbr{q_k \odot q_l}^{\top} .
    \end{align}
    Since $\xi=\psi/(1+\psi)$ is an increasing function of $\psi>0$, we always have $\rbr{\psi_k-\psi_l}\rbr{\xi_k-\xi_l}\ge0$. Thus $\Psi_1\odot\Psi_2\preceq\Diag(\diag(\Psi_1\Psi_2))$. Lastly, we need to put it back into an inequality of $B$ and $M(1)$. Let $u=\lambda-\lambda\odot\diag(M(1))$, then~\eqref{eq:temp} deduces that $B-M(1)=M(1)\Diag(u)B$ so that
    \begin{align}
        \Diag(\diag(\Psi_1\Psi_2)) &= \Diag\cbr{ \diag\sbr{ \sqrt{\Diag(u)} B \Diag(u) M(1) \sqrt{\Diag(u)} }} \\
        &= \sqrt{\Diag(u)} \, \Diag\cbr{ \diag\sbr{  B \Diag(u) M(1) }} \sqrt{\Diag(u)} \\
        &= \sqrt{\Diag(u)} \, \Diag\cbr{ \mathbf{1}_L - \diag(M(1))} \sqrt{\Diag(u)}
    \end{align}
    On the other hand,
    \begin{align}
        \Psi_1 \odot \Psi_2 &= \Diag(u) \sbr{B \odot M(1)} \Diag(u).
    \end{align}
    Therefore, $\Psi_1\odot\Psi_2\preceq\Diag(\diag(\Psi_1\Psi_2))$ implies that $\sqrt{\Lambda}\sbr{B \odot M(1)}\sqrt{\Lambda}\preceq I_L$, i.e. $\cQ(1)\succeq 0$.

    \vspace*{2mm}
    \noindent\emph{Step 3. Establishing the null space dimension.} It remains to show that $\dim\sbr{\mathrm{ker}\cQ(1)}=1$. Continue using the notations in Step 2. Since $B$ is irreducible and the diagonal conjugation defining $\Psi_1$ has strictly positive diagonal entries, $\Psi_1$ is irreducible as well. Recall that
    \begin{equation}
        F =  \sum_{1\le l<k\le L} \rbr{\psi_k-\psi_l}\rbr{\xi_k-\xi_l} \rbr{q_k \odot q_l}\rbr{q_k \odot q_l}^{\top}
    \end{equation}
    is always PSD. Now take $z\in\mathrm{ker}F\subseteq\RR^L$. It follows that
    \begin{equation}
        \sum_{1\le l<k\le L} \rbr{\psi_k-\psi_l}\rbr{\xi_k-\xi_l} \sbr{ \rbr{q_k \odot q_l}^{\top} z}^2 = z^{\top} F z = 0.
    \end{equation}
    Consequently, $\psi_k\neq\psi_l$ immediately deduces that $\rbr{q_k \odot q_l}^{\top} z=0$. Furthermore, letting $D_z=\Diag(z)$, the previous fact implies that
    \begin{align}
        D_z \Psi_1 - \Psi_1 D_z &= \rbr{\sum_{l} q_lq_l^{\top}} D_z \rbr{\sum_k \psi_k q_kq_k^{\top}} - \rbr{\sum_{l} \psi_l q_lq_l^{\top}} D_z \rbr{\sum_k q_kq_k^{\top}} \\
        &= \sum_{k,l} \sbr{ (\psi_k-\psi_l) q_l^{\top} D_z q_k } q_l q_k^{\top} = 0,
    \end{align}
    where the last step is due to $q_l^{\top} D_z q_k=\rbr{q_k \odot q_l}^{\top}z$. Since $D_z$ is diagonal, it holds that
    \begin{equation}
        z_i (\Psi_1)_{ij} = (D_z \Psi_1)_{ij} = (\Psi_1 D_z)_{ij} = z_j (\Psi_1)_{ij}.
    \end{equation}
    Lastly, since $\Psi_1$ is irreducible, we can find $L-1$ edges $e_1,\ldots,e_{L-1}\in[L]^{2}$ to connect every index such that we always have $(\Psi_1)_{e_j}\neq 0$ for $1\le j\le L-1$. Since $e_1,\ldots,e_{L-1}$ connect every index, it is enough to conclude that $z_1=z_2=\cdots=z_L$. Therefore, we learn that $\mathrm{ker}F=\mathrm{span}(\mathbf{1}_L)$, which further implies that $\dim\sbr{\mathrm{ker}\cQ(1)}=1$.

    Since $M(\sigma)$ is strictly decreasing in $\sigma$ on $(\sigma_+,\infty)$, the matrix $\cQ(\sigma)$ is strictly increasing. Therefore $\cQ(\sigma)\succ\cQ(1)\succeq0$ for every $\sigma>1$.
\end{proof}

\subsection{Simplifying the overlap formulas}

Section~\ref{subsec:matrix-overlaps} gives matrix overlap formulas for a general simple outlier. At the special outlier $\sigma=1$, Lemma~\ref{lemma:null-vector-at-1} gives an explicit null vector $w$, and the remaining task is to simplify the two-resolvent term that appears in the self-overlap formula. Recall the notion $\cP_{\sigma}$ from~\eqref{eq:P_sigma operator}. This linear operator naturally arises when we differentiate such an equation
\begin{equation}
    M_{\kappa}(z)^{-1} = zB^{-1} - \kappa V^{-1} T V^{-1} + \Diag\rbr{\lambda} - \Diag\cbr{ \lambda \odot \diag(M_{\kappa}(z)) }.
\end{equation}
When $\sigma=1$, we can establish the following identity. Using it, we can bypass $\cP_{1}$ when simplifying the expression for the asymptotic limits of $\matop(\nu_n)^\top\matop(\nu_n)$. It is proved purely exploiting linear algebra.

\begin{lemma}\label{eq:simplify cP_1}
    For any symmetric $T\in\SA_L(\RR)$, it holds that
    \begin{equation}\label{eq:algebra res}
        w^\top\cbr{ \rbr{\sqrt{\Lambda}B\sqrt{\Lambda}} \odot \rbr{ \cP_{1}\sbr{ V^{-1} T V^{-1} } } }w = \tr\sbr{T (I-V^{-1} M(1) V^{-1})},
    \end{equation}
    where $w$ is the null vector from Lemma~\ref{lemma:null-vector-at-1}.
\end{lemma}

\begin{proof}
    Proving this identity is pretty involved. We begin by introducing a few notations for shorthand. Recall that $M := M(1)$ is given subject to
    \begin{equation}\label{eq:algebra def M}
        M^{-1} = B^{-1} + \Diag\rbr{\lambda} - \Diag\cbr{ \lambda \odot \diag(M) }.
    \end{equation}
    Let $\Delta:=\cP_{1}\sbr{ V^{-1} T V^{-1} }\in\SA_L(\RR)$. Since $\cP_1$ is defined in order to differentiate $M_{\kappa}(1)$ with respect to $\kappa$, the following equation actually implicitly defines $\Delta$,
    \begin{equation}\label{eq:algebra def Delta}
        M^{-1} \Delta M^{-1} = V^{-1} T V^{-1} + \Diag(\lambda\odot\diag(\Delta)).
    \end{equation}
    These two equations suffice for all subsequent derivation. Let
    \begin{equation}
        s:=\lambda-\lambda\odot\diag(M)\in\RR^L, \quad d := \diag(\Delta)\in\RR^L, \quad U := V^{-1} T V^{-1}\in M_L(\RR).
    \end{equation}
    Then the two sides of~\eqref{eq:algebra res} becomes
    \begin{align}
        \mathrm{LHS} &= s^\top \rbr{ B \odot \Delta} s,\\
        \mathrm{RHS} &= \tr\sbr{T (I-V^{-1} M V^{-1})}.
    \end{align}
    Using~\eqref{eq:algebra def Delta}, we can write
    \begin{equation}
        U=M^{-1}\Delta M^{-1}-\Diag(\lambda\odot d).
    \end{equation}
    Hence
    \begin{equation}\label{eq: algebra proof eq1}
        MV^{-1}TV = MUB = \Delta M^{-1}B - M\Diag(\lambda\odot d)B.
    \end{equation}
    Taking diagonals and using
    \begin{equation}
        M^{-1}B = I+\Diag(s)B,
    \end{equation}
    we obtain
    \begin{equation}
        \diag(\Delta M^{-1}B) = \diag(\Delta)+\diag(\Delta\Diag(s)B).
    \end{equation}
    Moreover, using $\diag(\Delta\Diag(s)B)=(B\odot\Delta)s$, we can go from~\eqref{eq: algebra proof eq1} to find that
    \begin{equation}
        \diag(MV^{-1}TV) = d+(B\odot\Delta)s-\diag\rbr{M\Diag(\lambda \odot d)B},
    \end{equation}
    which can be further rearranged into
    \begin{equation}\label{eq: algebra proof eq2}
        (B\odot\Delta)s-\diag(MV^{-1}TV) = \diag\rbr{M\Diag(\lambda \odot d)B}-d.
    \end{equation}
    In the following, we will verify that $s$ is orthogonal to the right-handed side term of~\eqref{eq: algebra proof eq2}. To do so, let's derive the following:
    \begin{align}
        &\quad s^\top \diag\rbr{M\Diag(\lambda \odot d)B} \\
        &= \tr\sbr{ \Diag(s) M \Diag(\lambda \odot d) B} \\
        &= \tr\sbr{ \Diag(\lambda \odot d) B} - \tr\sbr{M\Diag(\lambda\odot d)} \\
        &= \lambda^\top d - \sbr{\lambda\odot\diag(M)}^\top d \\
        &= s^\top d, \label{eq: algebra proof eq3}
    \end{align}
    where the first equality is by the definition of those algebra notations, the second equality is from $\Diag(s)=M^{-1}-B^{-1}$ as seen from~\eqref{eq:algebra def M}, the third equality is because of $\diag(B)=\mathbf{1}_L$ and the last equality uses the definition of $s$.

    Henceforth, continuing from~\eqref{eq: algebra proof eq2} and~\eqref{eq: algebra proof eq3}, we conclude that
    \begin{align}
        \mathrm{LHS} &= s^\top \rbr{ B \odot \Delta} s \\
        &= s^\top \diag(MV^{-1}TV) \\
        &= \tr\sbr{ \Diag(s) MV^{-1}TV } \\
        &= \tr\sbr{ (M^{-1}-B^{-1}) MV^{-1}TV } \\
        &= \tr\sbr{T (I-V^{-1} M V^{-1})} = \mathrm{RHS},
    \end{align}
    where we have used $B=V^2$ in the middle.
\end{proof}



We can now combine the outlier statement from Section~\ref{sec:outlier} with the two special-point identities above. This gives the eigenvalue and overlap asymptotics at the outlier $\sigma=1$.

\begin{proposition}\label{prop:asymptotic-at-1}
    Let $\sigma_n$ be the top eigenvalue of $H$, and let
    $\nu_n\in\RR^{nL}$ be a corresponding unit eigenvector.
    Split $\nu_n$ into $L$ segments written as $\nu_n^{(1)},\ldots,\nu_n^{(L)}\in\RR^n$. Recall the sign flipping $s(\nu_n)$ in~\eqref{eq:sign breaking} and denote
    \begin{equation}\label{eq:simplification at 1}
    \begin{aligned}
        w &:= \sqrt{\lambda} \odot \rbr{\mathbf{1}_L -\diag\sbr{M(1)}} \in \RR^L, \\
        c^{\ast} &:= \sum_{l=1}^L \lambda_lM_{ll}(1)\sbr{1-M_{ll}(1)}>0, \\
        \Sigma_1 &:= \frac{1}{\sqrt{c^{\ast}}} \cbr{ V^{-1} \sbr{B - M(1)} } \in \RR^{L \times L}, \\
        \Sigma_2 &:= \frac{1}{c^{\ast}}\cbr{ I - V^{-1} M(1) V^{-1}} \in\SA_L^+(\RR).
    \end{aligned}
    \end{equation}
    Moreover, $c^\ast=w^\top\cQ'(1)w$.
    Then $\abr{\sigma_n-1} = \cO_{\prec}(n^{-1/10})$, and
    \begin{align}
        \nbr{ \frac{s(\nu_n)}{\sqrt{n}} \matop(\nu_n)^\top X - \Sigma_1 } = \cO_{\prec}(n^{-1/10}), \label{eq:overlap-nu-X}\\
        \nbr{ \matop(\nu_n)^\top \matop(\nu_n) - \Sigma_2 } = \cO_{\prec}(n^{-1/20}). \label{eq:overlap-nu-nu}
    \end{align}
\end{proposition}
Compared to the overlaps that can be obtain via Cauchy residue theorems for a general outlier eigenvalue, shown in Proposition~\ref{prop:sketch-statement-overlap} and~\ref{prop:sketch-matrix-overlaps}, the expressions are much more simplified at the special value $z=1$. We manage to guess the final simplified expressions, via state evolution of a Bayes optimal approximate message passing algorithm under a Gaussian prior distribution (see Appendix~\ref{sec:state evolution heuristic}), but the simplification process is rigorously proved via linear algebra identities in Lemma~\ref{eq:simplify cP_1}.

\begin{proof}[Proof of Proposition~\ref{prop:asymptotic-at-1}]
    The convergence of $\abr{\sigma_n-1}$ follows from Proposition~\ref{prop:sketch-statement-outlier}. Use Lemma~\ref{lemma:null-vector-at-1} to obtain the simplified expression for $w$. Recall the statements of Proposition~\ref{prop:detailed-matrix-overlaps} to find that
    \begin{equation}
        \left\|\matop(\nu_n)^\top\matop(\nu_n)- \frac{1}{c^\ast} \Tilde{\Sigma}_2\right\|=\cO_{\prec}(n^{-1/20}), \quad c^{\ast} = w^\top\cQ^\prime(1)w,
    \end{equation}
    for some $\Tilde{\Sigma}_2\in\SA_L^+(\RR)$ defined through $\cP_1$. Utilizing the identity stated in Lemma~\ref{eq:simplify cP_1} about $\cP_1$, we can simplify $\Tilde{\Sigma}_2$ to
    \begin{equation}
        \Tilde{\Sigma}_2 = I-V^{-1} M(1) V^{-1}.
    \end{equation}
    Since $\nbr{\nu_n}=1$ is well-normalized,
    \begin{align}
        c^{\ast} &= \tr\rbr{\Tilde{\Sigma}_2} \\
        &= \tr\sbr{ I - M(1) B^{-1} } \\
        &= \tr\cbr{ M(1) \sbr{\Diag(\lambda)-\Diag\sbr{\lambda\odot\diag(M(1))}} } \\
        &= \sum_{l=1}^L \lambda_l M_{ll}(1) (1-M_{ll}(1)).
    \end{align}
    In this way, we have obtained all the simplification in~\eqref{eq:simplification at 1}. Then the quantitative convergence in~\eqref{eq:overlap-nu-X} and~\eqref{eq:overlap-nu-nu} follow immediately from Proposition~\ref{prop:detailed-matrix-overlaps}\ref{prop:matrix-overlap-with-X} and Proposition~\ref{prop:detailed-matrix-overlaps}\ref{prop:matrix-self-overlap} respectively.
\end{proof}

\subsection{Proofs of the main spectral theorems}
\label{subsec:proof-main-spectral-theorems}

\begin{proof}[Proof of Theorem~\ref{thm: main uninformative}]
    First suppose $\SNR(\lambda,B)<1$. Proposition~\ref{prop:special solution at 1} gives $\sigma_+<1$. The argument leading to~\eqref{eq:subcritical-Q-positive} shows that the deterministic outlier equation has no root to the right of $\sigma_+$. Hence, for every fixed $\epsilon>0$, with high probability there is no eigenvalue of $H$ above $\sigma_++\epsilon$. Together with the bulk convergence from Theorem~\ref{thm:main ESD}, this implies $\sigma_1(H)\pto\sigma_+$.

    Now suppose $\SNR(\lambda,B)>1$. Proposition~\ref{prop:special solution at 1} again gives $\sigma_+<1$, and Lemma~\ref{lemma:null-vector-at-1} shows that $\det\cQ(1)=0$ with a one-dimensional kernel spanned by $w$. The crossing at $1$ is simple because Proposition~\ref{prop:asymptotic-at-1} gives
    $w^\top\cQ'(1)w=c^\ast>0$. Proposition~\ref{prop:sketch-statement-outlier}
    therefore gives a simple eigenvalue of $H$ converging to $1$. The same lemma
    also gives $\cQ(\sigma)\succ0$ for every $\sigma>1$, so the determinant
    equation has no root to the right of $1$. Thus the outlier at $1$ is the top
    eigenvalue, and $\sigma_1(H)\pto1$. All other outliers are associated with
    roots of $\det\cQ(\sigma)=0$ in $[\sigma_+,1)$, while the remaining
    eigenvalues stay at the bulk edge by the no-outside-spectrum estimate.
    Therefore $\sigma_2(H)$ converges in probability to a deterministic point in
    $[\sigma_+,1)$.
\end{proof}

\begin{proof}[Proof of Theorem~\ref{thm: main informative}]
    The existence of the valid solution $M(1)$ with $0\prec M(1)\prec B$ is Proposition~\ref{prop:special solution at 1}. By the informative part of Theorem~\ref{thm: main uninformative}, the top eigenvalue is simple and converges to $1$. Proposition~\ref{prop:asymptotic-at-1} gives the unit-eigenvector limits
    \[
        \frac{s(\nu_n)}{\sqrt n}\matop(\nu_n)^\top X\pto\Sigma_1,
        \qquad
        \matop(\nu_n)^\top\matop(\nu_n)\pto\Sigma_2.
    \]
    After choosing the global sign so that $s(\nu_n)=1$ and setting
    $\hat X=\sqrt n\,\matop(\nu_n)$, these are exactly
    \[
        \frac{1}{n}\hat X^\top X\pto\Sigma_1,
        \qquad
        \frac{1}{n}\hat X^\top\hat X\pto\Sigma_2.
    \]
    Finally, since $M(1)\prec B$, the matrix $\Sigma_1=(c^{\ast})^{-1/2}V^{-1}(B-M(1))$ is nonzero, so the top eigenvector has nontrivial overlap with the signal.
\end{proof}

\section{Derivation of the spectral algorithm}
\label{sec:derivation}
In this section, we derive a spectral method for signal recovery in \eqref{eq:observation model}. To this end, we linearize an appropriate message passing algorithm around the ``non-informative" fixed point. This general strategy has been successful for diverse high-dimensional models, including community detection, single-index and multi-index models, contextual block models, and spiked models with heterogeneous noise. In the present multi-view problem, the linearization produces the matrix $H$ studied in the main text, whose correlated noise component is described by the matrix Dyson equation introduced in Section~\ref{sec:formal present self-cosnistent eq}.

We start with a general family of AMP algorithms to estimate $X$ from $A$.
Similar algorithms have been explored in several recent works
\cite{nandy2024multimodal,rossetti2023approximate}. The formulation below
follows recent work by two of the present authors
\cite{yang2024fundamental}. The general iterations take the form
%
%
%
%
%
\begin{align}
    m_{i}^{t+1} &:= \sum_{k=1}^n\sqrt{\frac{\lambda_l}{n}}A^{(l)}_{i,k}\cE^{(l)}_t(m^t_{k})-\lambda_l\mathsf{d}_t^{(l)}\cE^{(l)}_{t-1}(m^{t-1}_{i}), \label{eq:coupled AMP update} \\
    \mathsf{d}_t^{(l)} &:= \frac{1}{n}\sum_{k=1}^n\partial_l\cE^{(l)}_t(m^t_{k}), \label{eq:coupled AMP Onsager}
\end{align}
where $\cE_0,\cE_1,\cE_2,\ldots:\RR^L\to\RR^L$ are a series of \emph{non-linear denoisers}. Assuming that the latent features $X_i^{(1:L)}$ are sampled i.i.d from a distribution $p$, one can select denoisers sequentially to maximize the correlation with the latent signal. This denoiser is related to a natural fixed-dimensional denoising problem.
Specifically, for hidden $\mathscr{X}\sim p(x)$, if we observe $m_l:=\lambda_l q_l\mathscr{X}_l+\cN(0,\lambda_lq_l)$ for every $l\in[L]$, the optimal denoiser $\hat{\mathscr{X}}$ is the posterior expectation
\begin{align}
    \hat{\mathscr{X}}\equiv\cE_{\mathsf{Bayes}}\rbr{m;q} := \frac{\int x \exp\cbr{ -\sum_{l=1}^L \rbr{m_{l}-\lambda_{l}q_{l}x_{l} }^2 / 2\lambda_{l}q_{l} } \ud p(x)}{\int\exp\cbr{ -\sum_{l=1}^L \rbr{m_{l}-\lambda_{l}q_{l}x_{l} }^2 / 2\lambda_{l}q_{l} } \ud p(x)} \in \RR^L, \quad\forall m\in\RR^L.
\end{align}
The \emph{Bayes optimal} AMP algorithm keeps track of a series of state evolution parameters $q^0,q^1,q^2,\ldots\in[0,+\infty)^L$ and adopts $\cE_t(\cdot)=\cE_{\mathsf{Bayes}}\rbr{\cdot;q^t}$ in~\eqref{eq:coupled AMP update} and~\eqref{eq:coupled AMP Onsager}. Heuristics from spin glass theory suggest that this algorithm should be optimal among computationally efficient algorithms for recovering the latent signal in \eqref{eq:observation model}.


We derive the relevant spectral operator by linearizing this iterative algorithm. Assuming $\cE_{\mathsf{Bayes}}\rbr{0;q} \equiv \EE[x] = 0$ and $\nabla_m \cE_{\mathsf{Bayes}}\rbr{0;q} \equiv \EE[xx^\top]=B$ for any $q\in[0,+\infty)^L$, such an expansion
\begin{equation}
    \cE_{\mathsf{Bayes}}(m;q) \approx \cE_{\mathsf{Bayes}}\rbr{0;q}+\nabla_m\cE_{\mathsf{Bayes}}\rbr{0;q} m\equiv Bm
\end{equation}
linearizes the iterations while preserving the correlation structure contained in the prior $p(x)$. Assuming that the algorithm converges to a fixed point, the limit can be written, after the change of variables detailed in Appendix~\ref{subsec: linearizing AMP}, as an eigenvalue equation for the matrix $H$ defined in~\eqref{eq:H}. The decomposition $H=H_0+H_1$ used throughout the spectral analysis is given in~\eqref{eq:H_H0H1}, with $H_0$ and $H_1$ defined in~\eqref{eq:form of H0}--\eqref{eq:form of H1}. Using a slightly different perspective, we can also obtain the same linearized AMP matrix by imposing a Gaussian prior distribution on the spikes; see Appendix~\ref{sec:state evolution heuristic}.

\section{Discussion}

We have analyzed the spectral behavior of the linearized AMP matrix for the multi-view spiked Wigner model. The main conclusion is that the emergence of an informative spectral direction is governed by the explicit quantity $\SNR(\lambda,B)$ in~\eqref{eq:summary SNR def}. Below the threshold $\SNR(\lambda,B)=1$, the largest eigenvalue remains at the right edge of the limiting bulk spectrum; above the threshold, a distinguished outlier appears at the point $1$, and its eigenvector has nontrivial matrix-valued overlaps with the latent signals. This gives an explicit spectral weak-recovery threshold for the linearized AMP method.

The proof also identifies the deterministic mechanism behind this transition. The bulk spectrum is controlled by a matrix Dyson equation associated with the correlated Gaussian noise component, while the outliers are described by a finite-dimensional determinant equation involving its solution. At the special point $z=1$, the matrix Dyson equation admits an explicit characterization that connects the sign of $\SNR(\lambda,B)-1$ to the creation of a root of the outlier equation. This is the step that turns the general outlier theory into the sharp phase-transition formula.

On the technical side, the proof keeps the matrix Dyson equation in an explicit
finite-dimensional form. This makes it possible to use the same deterministic
objects in the bulk law, the spike-direction deterministic equivalents, and the
outlier equation, which is what ultimately produces a single spectral
characterization of the transition.

The variational analysis in Section~\ref{sec:optimizing RS free energy} suggests a deeper structural connection between the matrix Dyson equation for the linearized AMP matrix and the TAP free energy associated with the underlying inference problem. The proof uses the variational characterization directly, but the TAP interpretation gives a useful perspective on why the same finite-dimensional objects appear both in the spectral analysis and in the fixed-point description of the inference problem.

Finally, we expect the spectral threshold $\SNR(\lambda,B)=1$, with $\SNR(\lambda,B)$ defined in \eqref{eq:summary SNR def}, to be the algorithmic threshold for the multi-view spiked model \eqref{eq:observation model}. A rigorous proof of this conjecture is an interesting direction for future inquiry.

\medskip
\noindent\textbf{Acknowledgments.}
SS gratefully acknowledges support from NSF grant DMS CAREER 2239234, ONR grant N00014-23-1-2489, and AFOSR grant FA9950-23-1-0429. Y.M.L. gratefully acknowledges support from a Harvard College Professorship, the Harvard FAS Dean's Fund for Promising Scholarship, and DARPA grant DIAL-FP-038. The authors thank Hang Du, Henry Hu, and Saba Lepsveridze for stimulating discussions and for sharing an early version of their manuscript.

\bibliographystyle{alpha}
\bibliography{references}

\newcommand{\etalchar}[1]{$^{#1}$}
\begin{thebibliography}{KMM{\etalchar{+}}13}

\bibitem[AEK19]{ajanki2019stability}
Oskari~H Ajanki, L{\'a}szl{\'o} Erd{\H{o}}s, and Torben Kr{\"u}ger.
\newblock Stability of the matrix dyson equation and random matrices with correlations.
\newblock {\em Probability Theory and Related Fields}, 173:293--373, 2019.

\bibitem[AEK20]{alt2020dyson}
Johannes Alt, L{\'a}szl{\'o} Erd{\H{o}}s, and Torben Kr{\"u}ger.
\newblock The dyson equation with linear self-energy: spectral bands, edges and cusps.
\newblock {\em Documenta Mathematica}, 25:1421--1539, 2020.

\bibitem[AEKN19]{alt2019location}
Johannes Alt, L{\'a}szl{\'o} Erd{\"o}s, Torben~H Kr{\"u}ger, and Yuriy Nemish.
\newblock Location of the spectrum of kronecker random matrices.
\newblock In {\em Annales de l'institut Henri Poincare}, volume~55, 2019.

\bibitem[AEKS20]{alt2020correlated}
Johannes Alt, L{\'a}szl{\'o} Erd{\"o}s, Torben~H Kr{\"u}ger, and Dominik~J Schr{\"o}der.
\newblock Correlated random matrices: Band rigidity and edge universality.
\newblock {\em Annals of Probability}, 48(2), 2020.

\bibitem[AGZ10]{anderson2010introduction}
Greg~W Anderson, Alice Guionnet, and Ofer Zeitouni.
\newblock {\em An introduction to random matrices}.
\newblock Number 118. Cambridge university press, 2010.

\bibitem[BBAP05]{baik2005phase}
Jinho Baik, G{\'e}rard Ben~Arous, and Sandrine P{\'e}ch{\'e}.
\newblock Phase transition of the largest eigenvalue for nonnull complex sample covariance matrices.
\newblock {\em Annals of Probability}, pages 1643--1697, 2005.

\bibitem[BCSvH24]{bandeira2024matrix}
Afonso~S Bandeira, Giorgio Cipolloni, Dominik Schr{\"o}der, and Ramon van Handel.
\newblock Matrix concentration inequalities and free probability ii. two-sided bounds and applications.
\newblock {\em arXiv preprint arXiv:2406.11453}, 2024.

\bibitem[BGN11]{benaych2011eigenvalues}
Florent Benaych-Georges and Raj~Rao Nadakuditi.
\newblock The eigenvalues and eigenvectors of finite, low rank perturbations of large random matrices.
\newblock {\em Advances in Mathematics}, 227(1):494--521, 2011.

\bibitem[Bol14]{bolthausen2014iterative}
Erwin Bolthausen.
\newblock An iterative construction of solutions of the tap equations for the sherrington--kirkpatrick model.
\newblock {\em Communications in Mathematical Physics}, 325(1):333--366, 2014.

\bibitem[CLM22]{chen2022global}
Shuxiao Chen, Sifan Liu, and Zongming Ma.
\newblock Global and individualized community detection in inhomogeneous multilayer networks.
\newblock {\em The Annals of Statistics}, 50(5):2664--2693, 2022.

\bibitem[DAM17]{deshpande2017asymptotic}
Yash Deshpande, Emmanuel Abbe, and Andrea Montanari.
\newblock Asymptotic mutual information for the balanced binary stochastic block model.
\newblock {\em Information and Inference: A Journal of the IMA}, 6(2):125--170, 2017.

\bibitem[DHL26]{du2026twoview}
Hang Du, Henry Hu, and Saba Lepsveridze.
\newblock Optimal spectral algorithms for correlated two-view models in high dimensions.
\newblock {\em personal communication}, 2026.

\bibitem[EKY13]{erdHos2013averaging}
L{\'a}szl{\'o} Erd{\H{o}}s, Antti Knowles, and Horng-Tzer Yau.
\newblock Averaging fluctuations in resolvents of random band matrices.
\newblock In {\em Annales Henri Poincar{\'e}}, volume~14, pages 1837--1926. Springer, 2013.

\bibitem[EKYY13]{erdos2013local}
L{\'a}szl{\'o} Erdos, Antti Knowles, Horng-Tzer Yau, and Jun Yin.
\newblock The local semicircle law for a general class of random matrices.
\newblock {\em Electron. J. Probab}, 18(59):1--58, 2013.

\bibitem[Erd19]{erdos2019matrix}
Laszlo Erdos.
\newblock The matrix dyson equation and its applications for random matrices.
\newblock {\em arXiv preprint arXiv:1903.10060}, 2019.

\bibitem[EY17]{erdHos2017dynamical}
L{\'a}szl{\'o} Erd{\H{o}}s and Horng-Tzer Yau.
\newblock {\em A dynamical approach to random matrix theory}, volume~28.
\newblock American Mathematical Soc., 2017.

\bibitem[Flu84]{flury1984common}
Bernhard~N Flury.
\newblock Common principal components in k groups.
\newblock {\em Journal of the American Statistical Association}, 79(388):892--898, 1984.

\bibitem[FVRS22]{feng2022unifying}
Oliver~Y Feng, Ramji Venkataramanan, Cynthia Rush, and Richard~J Samworth.
\newblock A unifying tutorial on approximate message passing.
\newblock {\em Foundations and Trends{\textregistered} in Machine Learning}, 15(4):335--536, 2022.

\bibitem[GHL26]{gong2026fundamental}
Shuyang Gong, Dong Huang, and Zhangsong Li.
\newblock Fundamental limits of community detection in contextual multi-layer stochastic block models.
\newblock {\em arXiv preprint arXiv:2602.08173}, 2026.

\bibitem[HFS07]{helton2007operator}
J~William Helton, Reza~Rashidi Far, and Roland Speicher.
\newblock Operator-valued semicircular elements: solving a quadratic matrix equation with positivity constraints.
\newblock {\em International Mathematics Research Notices}, 2007(9):rnm086--rnm086, 2007.

\bibitem[HLN07]{hachem2007deterministic}
Walid Hachem, Philippe Loubaton, and Jamal Najim.
\newblock Deterministic equivalents for certain functionals of large random matrices.
\newblock {\em The Annals of Applied Probability}, 17(3):875--930, 2007.

\bibitem[Joh01]{johnstone2001distribution}
Iain~M Johnstone.
\newblock On the distribution of the largest eigenvalue in principal components analysis.
\newblock {\em The Annals of statistics}, 29(2):295--327, 2001.

\bibitem[KMM{\etalchar{+}}13]{krzakala2013spectral}
Florent Krzakala, Cristopher Moore, Elchanan Mossel, Joe Neeman, Allan Sly, Lenka Zdeborov{\'a}, and Pan Zhang.
\newblock Spectral redemption in clustering sparse networks.
\newblock {\em Proceedings of the National Academy of Sciences}, 110(52):20935--20940, 2013.

\bibitem[KY13]{knowles2013isotropic}
Antti Knowles and Jun Yin.
\newblock The isotropic semicircle law and deformation of wigner matrices.
\newblock {\em Communications on Pure and Applied Mathematics}, 66(11):1663--1749, 2013.

\bibitem[KZ25]{keup2025optimal}
Christian Keup and Lenka Zdeborov{\'a}.
\newblock Optimal thresholds and algorithms for a model of multi-modal learning in high dimensions.
\newblock {\em Journal of Statistical Mechanics: Theory and Experiment}, 2025(9):093302, 2025.

\bibitem[LAL19]{luo2019optimal}
Wangyu Luo, Wael Alghamdi, and Yue~M Lu.
\newblock Optimal spectral initialization for signal recovery with applications to phase retrieval.
\newblock {\em IEEE Transactions on Signal Processing}, 67(9):2347--2356, 2019.

\bibitem[Leh99]{lehner1999computing}
Franz Lehner.
\newblock Computing norms of free operators with matrix coefficients.
\newblock {\em American Journal of Mathematics}, 121(3):453--486, 1999.

\bibitem[Li25]{li2025algorithmic}
Zhangsong Li.
\newblock The algorithmic phase transition in symmetric correlated spiked wigner model.
\newblock {\em arXiv preprint arXiv:2511.06040}, 2025.

\bibitem[LL20]{lu2020phase}
Yue~M Lu and Gen Li.
\newblock Phase transitions of spectral initialization for high-dimensional non-convex estimation.
\newblock {\em Information and Inference: A Journal of the IMA}, 9(3):507--541, 2020.

\bibitem[LM19]{lelarge2019fundamental}
Marc Lelarge and L{\'e}o Miolane.
\newblock Fundamental limits of symmetric low-rank matrix estimation.
\newblock {\em Probability Theory and Related Fields}, 173:859--929, 2019.

\bibitem[MKK24]{mergny2024spectral}
Pierre Mergny, Justin Ko, and Florent Krzakala.
\newblock Spectral phase transition and optimal pca in block-structured spiked models.
\newblock In {\em Proceedings of the 41st International Conference on Machine Learning}, pages 35470--35491, 2024.

\bibitem[MM17]{matias2017statistical}
Catherine Matias and Vincent Miele.
\newblock Statistical clustering of temporal networks through a dynamic stochastic block model.
\newblock {\em Journal of the Royal Statistical Society Series B: Statistical Methodology}, 79(4):1119--1141, 2017.

\bibitem[MM18]{mondelli2018fundamental}
Marco Mondelli and Andrea Montanari.
\newblock Fundamental limits of weak recovery with applications to phase retrieval.
\newblock In {\em Conference On Learning Theory}, pages 1445--1450. PMLR, 2018.

\bibitem[MS24]{montanari2024friendly}
Andrea Montanari and Subhabrata Sen.
\newblock A friendly tutorial on mean-field spin glass techniques for non-physicists.
\newblock {\em Foundations and Trends{\textregistered} in Machine Learning}, 17(1):1--173, 2024.

\bibitem[MV21]{montanari2021estimation}
Andrea Montanari and Ramji Venkataramanan.
\newblock Estimation of low-rank matrices via approximate message passing.
\newblock {\em The Annals of Statistics}, 49(1), 2021.

\bibitem[MZ25]{mergny2025spectral}
Pierre Mergny and Lenka Zdeborov{\'a}.
\newblock Spectral thresholds in correlated spiked models and fundamental limits of partial least squares.
\newblock {\em arXiv preprint arXiv:2510.17561}, 2025.

\bibitem[NM24]{nandy2024multimodal}
Sagnik Nandy and Zongming Ma.
\newblock Multimodal data integration and cross-modal querying via orchestrated approximate message passing.
\newblock {\em arXiv preprint arXiv:2407.19030}, 2024.

\bibitem[PWBM18]{perry2018optimality}
Amelia Perry, Alexander~S Wein, Afonso~S Bandeira, and Ankur Moitra.
\newblock Optimality and sub-optimality of pca i: Spiked random matrix models.
\newblock {\em The Annals of Statistics}, 46(5):2416--2451, 2018.

\bibitem[Ree20]{reeves2020information}
Galen Reeves.
\newblock Information-theoretic limits for the matrix tensor product.
\newblock {\em IEEE Journal on Selected Areas in Information Theory}, 1(3):777--798, 2020.

\bibitem[RR23]{rossetti2023approximate}
Riccardo Rossetti and Galen Reeves.
\newblock Approximate message passing for the matrix tensor product model.
\newblock {\em arXiv preprint arXiv:2306.15580}, 2023.

\bibitem[Ver18]{vershynin2018high}
Roman Vershynin.
\newblock {\em High-dimensional probability: An introduction with applications in data science}, volume~47.
\newblock Cambridge university press, 2018.

\bibitem[Wai19]{wainwright2019high}
Martin~J Wainwright.
\newblock {\em High-dimensional statistics: A non-asymptotic viewpoint}, volume~48.
\newblock Cambridge university press, 2019.

\bibitem[YLS25]{yang2024fundamental}
Xiaodong Yang, Buyu Lin, and Subhabrata Sen.
\newblock Fundamental limits of community detection from multi-view data: multi-layer, dynamic and partially labeled block models.
\newblock {\em The Annals of Statistics}, 53(6):2728--2756, 2025.

\bibitem[ZJVM24]{zhang2024spectral}
Yihan Zhang, Hong~Chang Ji, Ramji Venkataramanan, and Marco Mondelli.
\newblock Spectral estimators for structured generalized linear models via approximate message passing.
\newblock In {\em The Thirty Seventh Annual Conference on Learning Theory}, pages 5224--5230. PMLR, 2024.

\bibitem[ZM24]{zhang2024matrix}
Yihan Zhang and Marco Mondelli.
\newblock Matrix denoising with doubly heteroscedastic noise: Fundamental limits and optimal spectral methods.
\newblock {\em Advances in Neural Information Processing Systems}, 37:93060--93117, 2024.

\bibitem[ZMV22]{zhang2022precise}
Yihan Zhang, Marco Mondelli, and Ramji Venkataramanan.
\newblock Precise asymptotics for spectral methods in mixed generalized linear models.
\newblock {\em arXiv preprint arXiv:2211.11368}, 2022.

\end{thebibliography}

\appendix

\newpage

\section{Linearizing Bayes optimal AMP}

\subsection{Expansion around the origin}\label{subsec: linearizing AMP}
Continued from the derivation in the start of Section~\ref{sec:main results}, we provide a detailed derivation for the object in~\eqref{eq:H}. Recall that a Bayes optimal AMP algorithm proceeds by
\begin{align}
    m_{i,l}^{t+1} &:= \sum_{k=1}^n\sqrt{\frac{\lambda_l}{n}}A^{(l)}_{i,k}\cE^{(l)}_t(m^t_{k})-\lambda_l\mathsf{d}_t^{(l)}\cE^{(l)}_{t-1}(m^{t-1}_{i}), \\
    \mathsf{d}_t^{(l)} &:= \frac{1}{n}\sum_{k=1}^n\partial_l\cE^{(l)}_t(m^t_{k}),
\end{align}
where the denoisers are explicitly given as
\begin{align}\label{eq:Bayes optimal denoiser}
    \cE_t(m) \equiv \cE_{\mathsf{Bayes}}\rbr{m;q^t} := \frac{\int x \exp\cbr{ -\sum_{l=1}^L \rbr{m_{l}-\lambda_{l}q_{l}^{t}x_{l} }^2 / 2\lambda_{l}q_{l}^{t} } \ud p(x)}{\int\exp\cbr{ -\sum_{l=1}^L \rbr{m_{l}-\lambda_{l}q_{l}^{t}x_{l} }^2 / 2\lambda_{l}q_{l}^{t} } \ud p(x)} \in \RR^L.
\end{align}
Here $q^0,q^1,\ldots$ are a series of state evolution parameters that are sequentially defined by
\begin{equation}
    q^{t+1} = \EE\sbr{ \mathscr{X} \odot \cE_{\mathsf{Bayes}}\rbr{ \lambda\odot q^t \odot \mathscr{X} + \sqrt{\lambda\odot q^t}\odot\mathscr{W};q^t} }
\end{equation}
It is easy to verify that $q_{\ast}=0$ is always a trivial fixed point of this equation, in which case the algorithmic iterates are also fixed at $m_{\ast}=0$.

To proceed, the high-level idea is to do Taylor expansion for $\cE_{\mathsf{Bayes}}$ around $m=0$. Via taking derivative in $m_{l_2}$, we find
\begin{align*}
    \frac{\partial \cE_{\mathsf{Bayes}}^{(l_1)}}{\partial m_{l_2}}\rbr{m;q} &= \frac{\int x_{l_1}x_{l_2} \exp\cbr{ -\sum_{l=1}^L \rbr{m_{l}-\lambda_{l}q_{l}x_{l} }^2 / 2\lambda_{l}q_{l} } \ud p(x)}{\int\exp\cbr{ -\sum_{l=1}^L \rbr{m_{l}-\lambda_{l}q_{l}x_{l} }^2 / 2\lambda_{l}q_{l} } \ud p(x)} \\
    &\quad - \frac{\int x_{l_1} \exp\cbr{ -\sum_{l=1}^L \rbr{m_{l}-\lambda_{l}q_{l}x_{l} }^2 / 2\lambda_{l}q_{l} } \ud p(x)}{\int\exp\cbr{ -\sum_{l=1}^L \rbr{m_{l}-\lambda_{l}q_{l}x_{l} }^2 / 2\lambda_{l}q_{l} } \ud p(x)}\\
    &\qquad \times \frac{\int x_{l_2} \exp\cbr{ -\sum_{l=1}^L \rbr{m_{l}-\lambda_{l}q_{l}x_{l} }^2 / 2\lambda_{l}q_{l} } \ud p(x)}{\int\exp\cbr{ -\sum_{l=1}^L \rbr{m_{l}-\lambda_{l}q_{l}x_{l} }^2 / 2\lambda_{l}q_{l} } \ud p(x)}.
\end{align*}
Henceforth, when taking $m_{\ast}=0$ and $q_{\ast}=0$, the partial derivative boils down to the correlations under the original prior $p(x)$,
\begin{equation*}
    \frac{\partial \cE_{\mathsf{Bayes}}^{(l_1)}}{\partial m_{l_2}}\rbr{0;0} = \EE_p\sbr{X_{l_1}X_{l_2}} = B_{l_1l_2}.
\end{equation*}
A Taylor expansion of this denoising function at $m\approx0$ yields that
\begin{align*}
    \cE^{(l)}_{\mathsf{Bayes}}(m;0) &= \cE^{(l)}_{\mathsf{Bayes}}(0;0) + \sum_{l_1\in[L]} \frac{\partial \cE_{\mathsf{Bayes}}^{(l_1)}}{\partial m_{l_2}}\rbr{0;0} m_{l_1} +o(|m|).
\end{align*}
Now suppose we choose the denoising function to be the linear component of this Taylor expansion, $\bar{\cE}^{(l)}(m)=\sum_{l_1} \EE_p\sbr{X^{(l)}X^{(l_1)}}m_{l_1}$ for any $l\in[L]$. Then the Onsager term \eqref{eq:coupled AMP Onsager} would be greatly simplified to $\mathsf{d}_t^{(l)}\equiv 1$. Assume the updating equation~\eqref{eq:coupled AMP update} has reached a new equilibrium with the linearized denoisers, i.e. $m^t=m^\ast$ for any $t$, then $m^\ast$ must solve the following linear equations
\begin{equation}\label{eq:linearized AMP equilibrium}
    m^\ast_{:,l}=\rbr{\sqrt{\frac{\lambda^{(l)}}{n}}A^{(l)}-\lambda^{(l)}I_n}\rbr{m^\ast_{:,l}+\sum_{l_1\neq l} B_{ll_1} m^\ast_{:,l_1}}, \quad\forall l\in[L].
\end{equation}
To simplify this set of equations furthermore, we rearrange $m^\ast\in\RR^{n\times L}$ into a vector of dimension $nL$ by
\begin{equation*}
    \mathsf{m}^\ast=\begin{bmatrix}
        m^\ast_{:,1}\\
        m^\ast_{:,2}\\
        \cdots\\
        m^\ast_{:,L}
    \end{bmatrix}\in\RR^{nL}.
\end{equation*}
Then \eqref{eq:linearized AMP equilibrium} can be written into a matrix form
\begin{align}
    \mathsf{m}^\ast&=\cbr{\sbr{\diag\rbr{\sqrt{\Lambda/n}}\otimes I_n}\mathsf{A}-\diag\rbr{\Lambda}\otimes I_n}\cbr{B\otimes I_n}\mathsf{m}^\ast,\label{eq:linearized AMP matrix-form}\\
    \mathsf{A}&=\diag\cbr{A^{(1)},\ldots,A^{(L)}}\in\RR^{nL\times nL},\notag\\
    \Lambda &=\rbr{\lambda^{(1)},\ldots,\lambda^{(L)}}\in\RR^L,\notag
\end{align}
where $\otimes$ denotes Kronecker product between matrices with
\begin{equation*}
   B\otimes I_n =\begin{bmatrix}
        I_n & \EE_p\sbr{X^{(1)}X^{(2)}}I_n & \cdots & \EE_p\sbr{X^{(1)}X^{(L)}}I_n\\
        \EE_p\sbr{X^{(1)}X^{(2)}}I_n & I_n & \cdots & \EE_p\sbr{X^{(2)}X^{(L)}}I_n\\
        \cdots & \cdots  & \cdots & \cdots\\
        \EE_p\sbr{X^{(1)}X^{(L)}}I_n & \EE_p\sbr{X^{(2)}X^{(L)}}I_n & \cdots & I_n\\
    \end{bmatrix}\in\RR^{nL\times nL}.
\end{equation*}

\begin{remark}
    If $L=1$, then~\eqref{eq:linearized AMP equilibrium} directly boils down to
    \begin{equation*}
        m^\ast = \rbr{\sqrt{\frac{\lambda}{n}}A-\lambda I_n} m^\ast \quad\quad\Longrightarrow\quad\quad \frac{1}{\sqrt{n}}A m^\ast=\frac{\lambda+1}{\sqrt{\lambda}}m^\ast.
    \end{equation*}
    So we are actually solving for the eigenvector of $A/\sqrt{n}$ when the corresponding eigenvalue is $(\lambda+1)/\sqrt{\lambda}$. This is exactly the largest eigenvalue predicted by BBP phase transition when $\lambda>1$.
\end{remark}

The last step is to plug in $B=V^2$. It turns out that $\rbr{V\otimes I_n}m^\ast$ should be an eigenvector of $H$ given in \eqref{eq:H} with eigenvalue exactly $1$.

\subsection{Gaussian prior}\label{sec:state evolution heuristic}

While Section~\ref{subsec: linearizing AMP} presents one way of linearizing AMP via expanding the non-linear denoisers around the origin, here is an alternative way by directly placing a Gaussian prior $\cN(0,B)$ as $p(x)$. In this case, those Bayes optimal denoisers~\eqref{eq:Bayes optimal denoiser} automatically become linear in $m$. Benefiting from this alternative, we can make educated guesses on the simplified asymptotic overlaps shown in Theorem~\ref{thm: main informative}.

Let $m_t^{(1)},\ldots,m_t^{(L)}\in\RR^n$ denote the running variables in our AMP algorithm. In our previous note, we have characterized its updating rule as
\begin{align}
    m_l^{t+1}&=\sqrt{\frac{\lambda_l}{n}} A^{(l)} \cE^{(l)}_t\rbr{m_t^{(1)},\ldots,m_t^{(L)}}-\lambda_l\mathsf{d}_t^{(l)}\cE^{(l)}_{t-1}\rbr{m_t^{(1)},\ldots,m_t^{(L)}} \in \RR^{n},\\
    \mathsf{d}_t^{(l)}&=\frac{1}{n}\sum_{i=1}^n \partial_l\cE^{(l)}_t\rbr{m_{t,i}^{(1)},\ldots,m_{t,i}^{(L)}} \in \RR.
\end{align}
By imposing a multivariate Gaussian prior $\cN(0,B)$ on each $X_i^{(1)},\ldots,X_i^{(L)}$, a closed form for the Bayes optimal denoisers can be obtained. Suppose the state evolution iterates at time $t$ is $q_t\in[0,1]^L$, it follows from~\eqref{eq:Bayes optimal denoiser} that
\begin{equation}
    \cE_t(a_1,\ldots,a_L):=\rbr{\Diag\rbr{\lambda\odot q_t}+B^{-1}}^{-1}a,\quad \forall a_1,\ldots,a_L\in\RR,
\end{equation}
which also deduces that $\mathsf{d}_t^{(l)}=\sbr{\rbr{\Diag\rbr{\lambda\odot q_t}+B^{-1}}^{-1}}_{ll}$ for any $l$ and $t$. Then the series of state evolution parameters evolve according to such a rule
\begin{align}
    q_{t+1}^{(l)} &= \EE\sbr{ X^{(l)} \cE^{(l)}_t(\lambda \odot q_t \odot X+\cN(0,\diag(\lambda \odot q_t))) } \\
    &= \sbr{\rbr{\Diag\rbr{\lambda\odot q_t}+B^{-1}}^{-1} \Diag\rbr{\lambda\odot q_t} B}_{ll} \\
    &= 1 - \sbr{ \rbr{\Diag\rbr{\lambda\odot q_t}+B^{-1}}^{-1} }_{ll}.
\end{align}

\begin{conjecture}
   We conjecture that this choice of AMP algorithm will converge to equilibrium, as $t\to\infty$. By equilibrium, we mean the following two statements:
   \begin{enumerate}
       \item The state evolution parameters converge $q_t \to q_{\ast}\in\RR^L$, where $q_{\ast}$ is the solution to the following fixed-point equation
       \begin{equation}
           q_{\ast} = 1 - \diag\sbr{\rbr{\Diag\rbr{\lambda\odot q_{\ast}}+B^{-1}}^{-1}}.
       \end{equation}
       In comparison with our self-consistent equations \eqref{eq: self-consistent eq matrix}, it holds that $q_{\ast}^{(l)}=1-M_{ll}(1)$.

       \item Let $m_t\in\RR^{nL}$ be the concatenation of $m_t^{(1)},\ldots,m_t^{(L)}$. Then it also converges to equilibrium, as $m_t\to m_{\ast}\in\RR^{nL}$, where $m_{\ast}$ satisfies that
       \begin{equation}\label{eq:AMP variable equilibrium}
            m_{\ast}^{(l)}=\sqrt{\frac{\lambda_l}{n}} A^{(l)} \cE_\ast^{(l)}(m_{\ast})-\lambda_l\rbr{1-q_{\ast}^{(l)}}\cE_\ast^{(l)}(m_{\ast}),
        \end{equation}
        where we use $\cE_{\ast}$ to denote the denoiser with the limiting state evolution parameters $q_{\ast}$. Moreover, we have the following distributional result
        \begin{equation}\label{eq:AMP distributional}
            m_{\ast}^{(l)} \overset{d.}{\approx} \lambda_lq_{\ast}^{(l)} X^{(l)}+\sqrt{\lambda_l q_{\ast}^{(l)}}\cN(0,I_n),\quad l\in[L].
        \end{equation}
        Moreover, we also have $X^{(l)\top}\cE^{(l)}_\ast\rbr{m_{\ast}}/n \overset{d.}{\approx} q_{\ast}^{(l)}$.
   \end{enumerate}
\end{conjecture}

By some additional algebraic transformation, we can go from~\eqref{eq:AMP variable equilibrium} to have
\begin{equation}
    \cE_\ast\rbr{m_{\ast}} = \rbr{B\otimes I_n} \begin{bmatrix}
        \sqrt{\frac{\lambda_1}{n}}A^{(1)}-\lambda_1 I_n & & 0\\
        & \ldots &\\
        0 & & \sqrt{\frac{\lambda_L}{n}}A^{(L)}-\lambda_L I_n
    \end{bmatrix} \cE_\ast\rbr{m_{\ast}}.
\end{equation}
where we denote
\begin{equation}
    \cE_\ast\rbr{m_{\ast}} := \sbr{ \rbr{\Diag\rbr{\lambda\odot q_{\ast}}+B^{-1}}^{-1} \otimes I_n } m_{\ast} \in \RR^{nL}.
\end{equation}
We know that $X^{(l)\top}\cE^{(l)}_\ast\rbr{m_{\ast}} /n \approx q_{\ast}^{(l)}$.
By comparing to the definition of $H$, we find that
\begin{equation}
    \bar{\nu}^\ast = n^{-1/2} \rbr{V\otimes I_n}^{-1} \cE_\ast\rbr{m_{\ast}}\in\RR^{nL}
\end{equation}
is an eigenvector of $H$ with eigenvalue $1$. And $\nu^\ast=\frac{\bar{\nu}^\ast}{\norm{\bar{\nu}^\ast}}$ further normalizes its norm to $1$.

At this stage, \eqref{eq:AMP distributional} becomes a powerful tool to compute some overlaps. Firstly,
\begin{align}
    u_l^\top\bar{\nu}^\ast &= \frac{\sqrt{\lambda_l}}{n} \sbr{\rbr{Ve_l}\otimes X^{(l)}}^\top \rbr{V\otimes I_n}^{-1} \cE_\ast\rbr{m_{\ast}} \\
    &= \frac{\sqrt{\lambda_l}}{n} \sbr{e_l\otimes X^{(l)}}^\top \cE_\ast\rbr{m_{\ast}} \\
    &= \frac{\sqrt{\lambda_l}}{n} X^{(l)\top}\cE^{(l)}_\ast\rbr{m_{\ast}} \\
    &\approx \sqrt{\lambda_l}q_{\ast}^{(l)}.
\end{align}
Similarly, it also follows that
\begin{align}
    \bar{\nu}^{\ast\top} \bar{\nu}^\ast &= \frac{1}{n} \cE_\ast\rbr{m_{\ast}}^\top \rbr{B\otimes I_n}^{-1} \cE_\ast\rbr{m_{\ast}} \\
    &\approx \tr\cbr{ \frac{1}{n}  \rbr{B^{-1} \otimes I_n} \cE_\ast\rbr{m_{\ast}} \cE_\ast\rbr{m_{\ast}}^\top } \\
    &= \tr\bigg\{ B^{-1}  \sbr{\diag\rbr{\lambda\odot q_{\ast}}+B^{-1}}^{-1} \sbr{ \Diag(\lambda\odot q_{\ast}) + \Diag(\lambda\odot q_{\ast}) B \Diag(\lambda\odot q_{\ast}) } \\
    &\qquad\qquad \sbr{\diag\rbr{\lambda\odot q_{\ast}}+B^{-1}}^{-1} \bigg\} \\
    &= \tr\cbr{\Diag\rbr{\lambda\odot q_{\ast}}\sbr{\Diag\rbr{\lambda\odot q_{\ast}}+B^{-1}}^{-1}} \\
    &=\sum_{l=1}^L \lambda_lq_{\ast}^{(l)}(1-q_{\ast}^{(l)}).
\end{align}
Therefore, we should conclude that as $n\to\infty$
\begin{equation}
    u_l^\top\nu^\ast \to \frac{\sqrt{\lambda_l}q_{\ast}^{(l)}}{\sqrt{\sum_{l=1}^L \lambda_lq_{\ast}^{(l)}(1-q_{\ast}^{(l)})}}.
\end{equation}
This is how we manage to guess $\Sigma_1$ in Proposition~\ref{prop:asymptotic-at-1}.

Moreover, we can also compute
\begin{align}
    &\quad \bar{\nu}^{\ast\top} (e_{l_1}e_{l_2}^\top \otimes I_n) \bar{\nu}^\ast \\
    &= \frac{1}{n} \cE_\ast\rbr{m_{\ast}}^\top (V^{-1}e_{l_1}e_{l_2}^\top V^{-1} \otimes I_n) \cE_\ast\rbr{m_{\ast}} \\
    &\approx \tr\cbr{ \frac{1}{n}  \rbr{ V^{-1}e_{l_1}e_{l_2}^\top V^{-1} \otimes I_n} \cE_\ast\rbr{m_{\ast}} \cE_\ast\rbr{m_{\ast}}^\top } \\
    &= \tr\bigg\{ V^{-1}e_{l_1}e_{l_2}^\top V^{-1}  \sbr{\diag\rbr{\lambda\odot q_{\ast}}+B^{-1}}^{-1} \\
    &\qquad\qquad \sbr{ \Diag(\lambda\odot q_{\ast}) + \Diag(\lambda\odot q_{\ast}) B \Diag(\lambda\odot q_{\ast}) }  \sbr{\diag\rbr{\lambda\odot q_{\ast}}+B^{-1}}^{-1} \bigg\} \\
    &= \tr\cbr{ V^{-1}e_{l_1}e_{l_2}^\top V \Diag\rbr{\lambda\odot q_{\ast}}\sbr{\diag\rbr{\lambda\odot q_{\ast}}+B^{-1}}^{-1}} \\
    &= e_{l_1}^\top (I-V^{-1}MV^{-1}) e_{l_2}.
\end{align}
This is how we manage to guess $\Sigma_2$ in Proposition~\ref{prop:asymptotic-at-1}. Note that our proof to it is completely rigorous via Propositions~\ref{prop:sketch-matrix-overlaps}, and the linear algebra identity in Lemma~\ref{eq:simplify cP_1}.

\section{Standard facts for the self-consistent equation}\label{sec:proof of stability}

\begin{proof}[Proof of Proposition~\ref{prop:feasibility}]
    Our first step is to transform the equation into the exact form of matrix Dyson equations in \cite{ajanki2019stability}. 
    With transformation $M_0=-V^{-1} M V^{-1}$, the equation \eqref{eq: self-consistent eq matrix} can be restated as
    \begin{equation}\label{eq: transformed self-consistent eq}
        -M_0^{-1} = zI_L - \Gamma + V \Lambda V + V \, \Diag\cbr{ \lambda \odot \diag\rbr{V M_0 V} } \, V,
    \end{equation}
    which is exactly the same as equation (2.2) in \cite{ajanki2019stability}. Then the results in \cite{helton2007operator} suffice to establish the feasibility and uniqueness of the solution $M_0(z;\Gamma)$ to the equation~\eqref{eq: transformed self-consistent eq} under the additional constraint that $\Im M_0(z;\Gamma)\succ0$. Then $M(z;\Gamma):=-VM_0(z;\Gamma)V$ naturally becomes a solution to \eqref{eq: self-consistent eq matrix} with $\Im M(z;\Gamma)\prec 0$.

    What's more, \cite[Proposition 2.1]{ajanki2019stability} also implies that $M_0(z;\Gamma)$ admits a Stieltjes transform representation. In more detail, there exists a compactly-supported $\SA_L^+(\RR)$-valued measure $\rho_{\Gamma}$ on $\RR$ such that $\rho_{\Gamma}\rbr{\RR}=I_L$ and
    \begin{equation}
        M_0(z;\Gamma) = \int_{\RR} \frac{\rho_{\Gamma}(\ud \tau)}{\tau-z} \in \SA_L(\CC), \quad\forall z\in\CC^+.
    \end{equation}
    This representation then induces the following
    \begin{align}
        M(z;\Gamma) = - \int_{\RR} \frac{V\rho_{\Gamma}(\ud \tau)V}{\tau-z} \in \SA_L(\CC),
    \end{align}
    Therefore, we can finish the proof by taking $\nu_{\Gamma} = V \rho_{\Gamma} V$.
\end{proof}

The proof of Proposition~\ref{prop:stability}, including the invertibility estimate for the stability operator, is given in Section~\ref{sec:formal present self-cosnistent eq}.

\section{Proofs for the information-theoretical thresholds}\label{sec:IT proof}

Our information-theoretical impossibility results can be deduced easily based on \cite{yang2024fundamental}.

To this end, consider the following functional for any $q\in[0,+\infty)^L$,
\begin{equation}\label{eq:IT condition functional}
     \cH(q) := \sum_{l=1}^L \frac{\lambda_l (q_l^2-2q_l)}{4} - \EE\log\sbr{ \int_{\RR^L} \exp\rbr{ -\frac{1}{2}\sum_{l=1}^L (\sqrt{\lambda_l q_l}\mathscr{X}_l+\mathscr{W}_l-\sqrt{\lambda_lq_l}x_l)^2 } \ud p(x) },
\end{equation}
where the expectation is taken over $\mathscr{X}\sim p(x)$ and $\mathscr{W}\sim\cN(0,I_L)$. This functional depends on both the prior $p(x)$ and $(\lambda,B)$, and is motivated from the replica-symmetric prediction on the free energy of the observation model~\eqref{eq:observation model}.
Under the following generic condition on $p(x)$,  the information-theoretic threshold coincides with the spectral threshold $\SNR(\lambda,B)\lessgtr1$.

\begin{enumerate}
    \item[\hypertarget{IT}{$(\star)$}] As long as $\SNR(\lambda,B)<1$, the functional $\cH(q)$ is uniquely minimized at $q=0$.
\end{enumerate}

While $\SNR(\lambda,B)<1$ in fact ensures $q=0$ to be a local minimum of $\cH(q)$, the prior distribution $p(x)$ has to satisfy additional conditions to guarantee that $q=0$ is the global minimizer of the free energy landscape.


\begin{proposition}\label{prop:impossibility result generic}
    Under Assumption~\ref{assump:distributional-spike}, suppose further that condition \hyperlink{IT}{$(\star)$} holds. Then for any estimator $\hat{X}(A)\in\RR^{n \times L}$ computed from the observation model~\eqref{eq:observation model}, its  overlap with the planted signal converges to $0$ i.e.,
    \begin{equation}\label{eq:no overlap info-theory}
        \frac{1}{n} \hat{X}(A)^\top X \pto 0_{L \times L},
    \end{equation}
    when $\SNR(\lambda,B)<1$.
\end{proposition}

\begin{proof}[Proof of Proposition~\ref{prop:impossibility result generic}]
    Theorem A.2 of \cite{yang2024fundamental} derives the asymptotic limit of the minimal mean squared error for $X^{(l)} X^{(l)\top}$ for any $l\in[L]$,
    \begin{equation}\label{eq: MMSE to 1}
        \inf_{\hat{X}(A)} \frac{1}{n^2} \EE\nbr{ \hat{X}^{(l)}\hat{X}^{(l)\top} - X^{(l)} X^{(l)\top} }_F^2 \pto 1-(q_l^{\ast})^2,
    \end{equation}
    where $q^{\ast}$ is the global minimizer of the functional $\cH(q)$ in~\eqref{eq:IT condition functional}.
    Although all the results of \cite{yang2024fundamental} are stated for priors supported on $\cbr{\pm 1}^L$, the free energy method adopted in proving their Theorem A.2 does not require $x\in\cbr{\pm 1}^L$ at all.
    Under condition \hyperlink{IT}{$(\star)$}, we learn that $q^{\ast}\equiv0$ as long as $\SNR(\lambda,B)<1$.
    
    Next, suppose that we manage to find an estimator $\hat{X}$ such that
    \begin{equation}
        \hat{X}^{(l)\top} X^{(l)}/n\pto\alpha>0.
    \end{equation}
    Then writing $\beta := \lim\frac{1}{n^2} \EE \nbr{ \hat{X}^{(l)}\hat{X}^{(l)\top} }_F^2>0$, we have
    \begin{align}
        &\quad \frac{1}{n^2} \EE\nbr{ t\hat{X}^{(l)}t\hat{X}^{(l)\top} - X^{(l)} X^{(l)\top} }_F^2 \\
        &= \frac{1}{n^2} \EE \nbr{ X^{(l)}X^{(l)\top} }_F^2 + \frac{t^4}{n^2} \EE \nbr{ \hat{X}^{(l)}\hat{X}^{(l)\top} }_F^2 - 2t^2\rbr{\hat{X}^{(l)\top} X^{(l)}/n}^2 \\
        &\pto 1-2t^2\alpha^2 + t^4\beta<1
    \end{align}
    if $0<t^2<2\alpha^2/\beta$. This is contradictory to~\eqref{eq: MMSE to 1}. So we can conclude this proposition.
\end{proof}


\begin{proof}[Proof of Proposition~\ref{prop:impossibility result strictly sub-gaussian}]
    We need to deal with the two conditions separately.
    \begin{enumerate}
        \item When condition \hyperlink{IT-concave}{$(\sharp)$} holds, we can proceed to show that condition \hyperlink{IT}{$(\star)$} indeed holds as well. Firstly, by direct differentiation, we can verify that
        \begin{align}
            \nabla \cH (q) &= \frac{\lambda_l}{2} \sbr{ q_l - T(\lambda \odot q) }, \\
            \nabla^2 \cH (0) &= \Lambda - \Lambda (B \odot B) \Lambda.
        \end{align}
        The detailed computation can be found in Proposition A.7 of \cite{yang2024fundamental}.
        Therefore, when $\SNR(\lambda,B)<1$, we immediately have that $\nabla \cH (0)=0$ and $\nabla^2 \cH (0)\succeq0$, so that $q=0$ becomes a local minimum of $\cH(q)$. Subsequently, with the help of condition \hyperlink{IT-concave}{$(\sharp)$}, we can use a contradiction argument.
        Suppose that $\cH(q)$ has some global minimizer $q^{\ast}\neq0$ elsewhere but still with $q^{\ast}\in[0,+\infty)^L$.
        There must exist an index $l\in[L]$ such that $q_l^{\ast}>0$ and $\nabla_l\cH (q^{\ast})=0$.
        As implied by condition \hyperlink{IT-concave}{$(\sharp)$}, the following uni-variate function
        \begin{equation}
            t\in[0,1] \to f(t)=\nabla_l \cH (tq^{\ast})
        \end{equation}
        should be strictly concave yet $f(0)=f(1)=0$. So we must have some small $\epsilon>0$ such that $f(\epsilon)<0$, i.e. $\nabla_l\cH(\epsilon q^{\ast})<0$. This is against the previous condition that $\nabla^2 \cH (0)\succeq0$. In conclusion, $\cH(q)$ cannot have any global minimizer but $0$.

        \item When condition \hyperlink{IT-strict-sg}{$(\flat)$} holds, we can use Theorem A.12 of \cite{yang2024fundamental} to derive that 
        \begin{equation}
            \inf_{\hat{X}(A)} \frac{1}{n^2} \EE\nbr{ \hat{X}^{(l)}\hat{X}^{(l)\top} - X^{(l)} X^{(l)\top} }_F^2 \pto 1,
        \end{equation}
        for any $l\in[L]$. It is proved via the second moment method. Then the rest of the proof is the same as Proposition~\ref{prop:impossibility result generic}.
    \end{enumerate}
\end{proof}

\begin{proof}[Proof of Corollary~\ref{remark:IT example}]
    In case~\ref{remark:IT example a}, it is crucial to realize that the functional $\cH(q)$ admits a closed-form expression, namely
    \begin{equation}
        \cH(q) = \sum_{l=1}^L \frac{\lambda_l (q_l^2-2q_l)}{4} - \frac{1}{2}\log\cbr{ \det\sbr{\rbr{B^{-1}+\Diag\rbr{\lambda \odot q}}^{-1}} } + C,
    \end{equation}
    where $C$ is a constant that does not depend on $q$. Via a change-of-variable $q=\chi+1$, it is equivalent to our objective~\eqref{eq:RS energy} with $z=1$. Then condition \hyperlink{IT}{$(\star)$} follows immediately from Lemma~\ref{lemma: minimizer z=1, uninformative}.

    In case~\ref{remark:IT example b}, Proposition A.14 and Lemma A.15 of \cite{yang2024fundamental} together yield our claim.
\end{proof}

\section{Gaussian concentration for the noise resolvent}
\label{app:resolvent-concentration}

This appendix records the standard Gaussian concentration input used in Section~\ref{sec:results on noise mat}. Although the entries of $H_1$ are correlated after conjugation by $V\otimes I_n$, the underlying variables are independent Gaussian entries of the matrices $W^{(l)}$, and the usual Lipschitz concentration argument applies.

\begin{lemma}\label{lemma:concentration}
    For any $z\in\CC^+$, $\Gamma\in\SA_L(\RR)$, and any $A\in\CC^{nL\times nL}$ with $\nbr{A}=1$, the normalized trace $\frac{1}{n}\tr\sbr{ A G(z;\Gamma) }$ concentrates around its mean. Namely, for any $\delta>0$,
    \begin{equation}
        \PP\rbr{ \abr{ \frac{1}{n}\tr\sbr{ A G(z;\Gamma) } - \EE \cbr{ \frac{1}{n}\tr\sbr{ A G(z;\Gamma) } } } \ge \delta } \le 2 \exp\rbr{- C \rbr{\Im z}^4 n^2 \delta^2}.
    \end{equation}
    In the display above, $C:= C(B,\lambda)=1/\rbr{\nbr{V}^4 L \max_l \lambda_l}$.
    Moreover, for any deterministic $A\in\CC^{nL\times nL}$,
    \begin{equation}\label{eq:test-trace-concentration}
        \tr\sbr{A\rbr{G(z;\Gamma)-\EE G(z;\Gamma)}}
        =\cO_{\prec}\rbr{\frac{\|A\|_F}{\sqrt n(\Im z)^2}}.
    \end{equation}
\end{lemma}

\begin{proof}
    The core step is to view $\frac{1}{n}\tr\sbr{ A G(z;\Gamma) }$ as a function of the independent Gaussian scalars $\{W_{ij}^{(l)}: l\in[L], 1\le i\le j\le n\}$. Specifically, let
    \begin{equation}
        \Phi(W) = \frac{1}{n}\tr\sbr{ A \rbr{ H_1(W) + \Gamma \otimes I_n - z I_{nL} }^{-1} }.
    \end{equation}
    We estimate the Lipschitz constant of $\Phi$ with respect to the Frobenius norm on the collection $W$. Since $\nbr{A}=1$, for another realization $\bar W$,
    \begin{align}
        &\quad \abr{\Phi(W)-\Phi(\bar{W})} \\
        &= \frac{\sqrt{L}}{n} \abr{ \tr\sbr{ A \rbr{ H_1(W) + \Gamma \otimes I_n - z I_{nL} }^{-1} \rbr{H_1(W)-H_1(\bar{W})} \rbr{ H_1(W) + \Gamma \otimes I_n - z I_{nL} }^{-1} } } \\
        &\le \frac{\sqrt{L}}{\sqrt{n}} \nbr{ \rbr{ H_1(W) + \Gamma \otimes I_n - z I_{nL} }^{-1} \rbr{H_1(W)-H_1(\bar{W})} \rbr{ H_1(W) + \Gamma \otimes I_n - z I_{nL} }^{-1} }_F \\
        &\le \frac{\sqrt{L}}{\sqrt{n}\rbr{\Im z}^2} \nbr{ H_1(W)-H_1(\bar{W}) }_F \\
        &\le \frac{1}{n} \frac{\nbr{V}^2 \sqrt{L\max_l \lambda_l}}{\rbr{\Im z}^2} \rbr{\sum_{l=1}^L  \nbr{W^{(l)}-\bar{W}^{(l)}}_F }.
    \end{align}
    We used $\|\rbr{ H_1 + \Gamma \otimes I_n - z I_{nL} }^{-1}\|\le 1/\Im z$ and the concrete form of $H_1$ in~\eqref{eq:form of H1}. Thus the Lipschitz constant of $\Phi$ is of order $1/n$. The result follows from standard concentration for Lipschitz functions of independent Gaussian variables, for instance \cite[Theorem 2.26]{wainwright2019high} or \cite[Theorem 5.2.3]{vershynin2018high}.

    The proof of~\eqref{eq:test-trace-concentration} is the same, but without the normalization by $n$ and with the Frobenius norm of the test matrix retained. For
    \[
        \Psi_A(W)=\tr\sbr{A\rbr{H_1(W)+\Gamma\otimes I_n-zI_{nL}}^{-1}},
    \]
    the preceding resolvent identity and Cauchy-Schwarz give
    \[
        \abr{\Psi_A(W)-\Psi_A(\bar W)}
        \le \frac{C(B,\lambda)\|A\|_F}{\sqrt n(\Im z)^2}
        \rbr{\sum_{l=1}^L\|W^{(l)}-\bar W^{(l)}\|_F}.
    \]
    Gaussian concentration for this Lipschitz functional gives~\eqref{eq:test-trace-concentration}.
\end{proof}

\end{document}